\documentclass[11pt,letterpaper,twoside,reqno,nosumlimits,final]{amsart}

\usepackage[usenames,dvipsnames]{xcolor}
\usepackage{fancyhdr}
\usepackage{amsmath,amsfonts,amsbsy,amsgen,amscd,mathrsfs,amssymb,amsthm}
\usepackage{stackengine}
\usepackage{url}

\usepackage{caption}
\usepackage{subcaption}

\usepackage{mathtools}

\mathtoolsset{showonlyrefs}

\usepackage[font=small,margin=0.25in,labelfont={sc},labelsep={colon}]{caption}

\usepackage{tikz}
\usepackage{microtype}
\usepackage{enumitem}

\definecolor{dark-gray}{gray}{0.3}
\definecolor{dkgray}{rgb}{.4,.4,.4}
\definecolor{dkblue}{rgb}{0,0,.5}
\definecolor{medblue}{rgb}{0,0,.75}
\definecolor{rust}{rgb}{0.5,0.1,0.1}

\usepackage[colorlinks=true]{hyperref}

\hypersetup{urlcolor=rust}
\hypersetup{citecolor=dkblue}
\hypersetup{linkcolor=dkblue}

\usepackage{setspace}

\usepackage{graphicx}
\usepackage{booktabs,longtable,tabu} 
\usepackage{multirow} 

\usepackage{float}

\usepackage[full]{textcomp}

\usepackage[scaled=.98,sups,osf]{XCharter}\usepackage[scaled=1.04,varqu,varl]{inconsolata}\usepackage[type1]{cabin}\usepackage[utopia,vvarbb,scaled=1.07]{newtxmath}
\usepackage[cal=boondoxo]{mathalfa}
\usepackage[T1]{fontenc}

\usepackage[draft]{changes}
\definechangesauthor[color=blue]{ml}
\definechangesauthor[color=red]{jt}

\usepackage{bm} 

\graphicspath{{figures/}}

\newtheorem{bigthm}{Theorem}

\newtheorem{theorem}{Theorem}[section]
\newtheorem{lemma}[theorem]{Lemma}

\newtheorem{proposition}[theorem]{Proposition}
\newtheorem{fact}[theorem]{Fact}

\newtheorem{corollary}[theorem]{Corollary}

\theoremstyle{definition}

\newtheorem{definition}[theorem]{Definition}
\newtheorem{example}[theorem]{Example}
\newtheorem{remark}[theorem]{Remark}

\newcommand{\term}{\emph}

\numberwithin{equation}{section} 
\numberwithin{figure}{section}
\numberwithin{table}{section}

\floatstyle{plaintop}
\newfloat{recipe}{thp}{lor}
\floatname{recipe}{Recipe}
\numberwithin{recipe}{section}

\providecommand{\mathbold}[1]{\bm{#1}}

\renewcommand{\phi}{\varphi}

\newcommand{\eps}{\varepsilon}

\newcommand{\half}{\tfrac{1}{2}}

\newcommand{\econst}{\mathrm{e}}

 \newcommand{\zerovct}{\vct{0}}

\newcommand{\set}[1]{\mathsf{#1}}

\newcommand{\ball}[1]{\mathsf{B}^{#1}}

\providecommand{\mathbbm}{\mathbb} 
\newcommand{\R}{\mathbbm{R}}

\newcommand{\polar}{\circ}

\newcommand{\abs}[1]{\left\vert {#1} \right\vert}

\newcommand{\diff}[1]{\mathrm{d}{#1}}
\newcommand{\idiff}[1]{\, \diff{#1}}

\newcommand{\grad}{\nabla} 

\newcommand{\Hess}{\operatorname{Hess}}

\newcommand{\Prob}[1]{\mathbb{P}\left\{{#1}\right\}}

\newcommand{\Expect}{\operatorname{\mathbb{E}}}

\DeclareMathOperator{\Var}{Var}

\stackMath \setstackgap{S}{1pt}

\newcommand{\intvol}{\mathrm{V}}
\newcommand{\rmv}{\bar{\intvol}}
\newcommand{\rotv}{\mathring{\intvol}}
\newcommand{\rmdelta}{\bar{\delta}}
\newcommand{\rotdelta}{\mathring{\delta}}
\newcommand{\rmI}{\bar{I}}
\newcommand{\rotI}{\mathring{I}}
\newcommand{\wills}{\mathrm{W}}
\newcommand{\rmwills}{\bar{\wills}}
\newcommand{\rotwills}{\mathring{\wills}}
\newcommand{\rmmu}{\bar{\mu}}
\newcommand{\rotmu}{\mathring{\mu}}

\newcommand{\intvolJ}{J_{\theta}}
\newcommand{\rotJ}{\mathring{J}_{\theta}}
\newcommand{\rmJ}{\bar{J}_{\theta}}

\newcommand{\rmsigma}{\bar{\sigma}}
\newcommand{\rotsigma}{\mathring{\sigma}}

\newcommand{\rmDelta}{\bar{\Delta}}
\newcommand{\rotDelta}{\mathring{\Delta}}

\newcommand{\vct}[1]{\mathbold{#1}}
\newcommand{\mtx}[1]{\mathbold{#1}}

\newcommand{\ip}[2]{\left\langle {#1},\ {#2} \right\rangle}

\newcommand{\norm}[1]{\left\Vert {#1} \right\Vert}
\newcommand{\normsq}[1]{\norm{#1}^2}

\DeclareMathOperator{\dist}{dist}
\DeclareMathOperator{\proj}{proj}

\newcommand{\Gr}{\mathrm{Gr}}
\newcommand{\SO}{\mathrm{SO}}
\newcommand{\RM}{\mathrm{SE}}
\newcommand{\Af}{\operatorname{Af}}

\title[Phase Transitions in Integral Geometry]{Sharp Phase Transitions \\ in Euclidean Integral Geometry}

\author[M.~Lotz and J.~A.~Tropp]{Martin Lotz and Joel~A.~Tropp}

\date{5 May 2020. Revised 15 December 2021.}


\begin{document}

\begin{abstract} The intrinsic volumes of a convex body are fundamental invariants that
capture information about the average volume of the projection of the
convex body onto a random subspace of fixed dimension.  The intrinsic volumes
also play a central role in integral geometry formulas that describe how
moving convex bodies interact.
Recent work has demonstrated that the sequence of intrinsic volumes concentrates
sharply around its centroid, which is called the central intrinsic volume.
The purpose of this paper is to derive finer concentration
inequalities for the intrinsic volumes and related sequences.
These concentration results have striking implications for high-dimensional integral geometry.
In particular, they uncover new phase transitions in formulas for random projections,
rotation means, random slicing, and the kinematic formula.
In each case, the location of the
phase transition is determined by
reducing each convex body to a
single summary parameter.
\end{abstract}

\maketitle

\section{Introduction}

In his 1733 treatise~\cite{leclerc1733essai}, \emph{Essai d'arithm{\'e}tique morale},
Georges-Louis Leclerc, Comte de Buffon, initiated the field of geometric probability.
The arithmetic is ``moral'' because it assesses
the degree of hope or fear that an uncertain event should instill,
or---more concretely---whether a game of chance is fair. In a discussion of the game \emph{franc-carreau}, Buffon calculated the probability that a needle, dropped at random on a plank floor,
touches more than one plank (Figure~\ref{fig:buffon-schematic}).
In 1860, Barbier~\cite{barbier1860note} gave a visionary
solution to Buffon's needle problem, based on invariance properties of the random model.
Advanced further by the insight of Crofton~\cite{Cro89:Probability}, this approach evolved into
the field of integral geometry.  See the elegant book~\cite{KR97:Introduction-Geometric}
of Klain \& Rota for an introduction to the subject.

\begin{figure}[t]
\begin{center}
\includegraphics[width=0.75\textwidth]{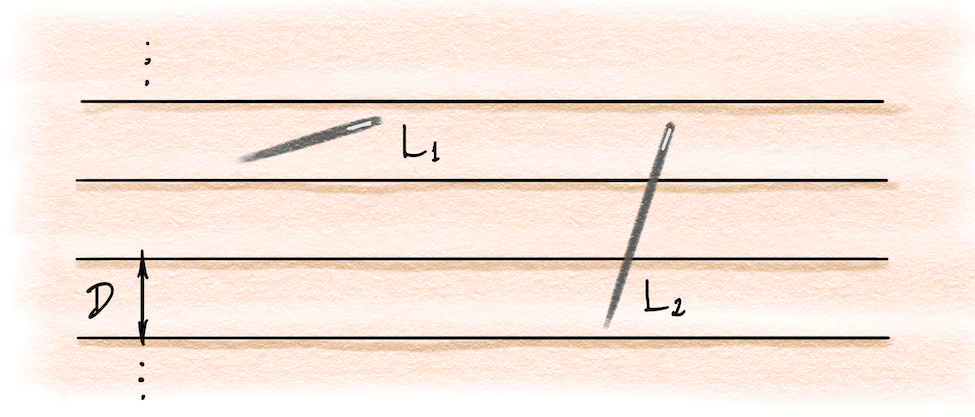} \caption{\textbf{Buffon's Needle.}  What is the probability that a needle,
dropped at random on a plank floor, touches more than one plank?
The needle is less likely to cross a boundary when it is short (needle $L_1$),
relative to the width $D$ of a plank.  A long needle ($L_2$) is more likely to cross
a boundary.} \label{fig:buffon-schematic}
\end{center}
\end{figure}

Let us contrast two different ways of describing the solution
to Buffon's problem.  First, we can write the intersection probability $p(L, D)$ exactly
in terms of the length $L$ of the needle and the width $D$ of the planks:
\begin{equation}\label{eqn:buffon-exact}\tag{A}
p(L,D) =  \frac{2}{\pi} \cdot \begin{cases}
    L/D, & L\leq D;\\
      \arccos\left(D/L\right) +  (L/D) \big(1 - \sqrt{1 - ( D/L )^2  } \big), & L\geq D.
  \end{cases}
\end{equation}
As a consequence, a needle of length $L = \pi D/4$
is equally likely to touch one plank or two.

Even without the precise result, it is clear that very short needles are unlikely
to touch two planks, while very long needles are likely to do so.
We can capture this intuition more vividly with an alternative formula.
For each $\alpha \in [0, 0.36]$,
\begin{equation}
\label{eqn:buffon-approx}\tag{B}
\begin{aligned}
L/D &\leq \phantom{1-{}}(2/\pi) \cdot \alpha
&\text{implies}&&	p(L,D) &\leq \alpha ;\\
L/D &\geq (1-2/\pi) \cdot \alpha^{-1} &\text{implies}&& p(L,D) &\geq 1-\alpha.
\end{aligned}
\end{equation}
This result states that there are two complementary regimes for the length
of the needle.  Depending on which case is active, the probability that the
needle crosses a boundary is either very small or very large.  In between,
there is a wide transition region where neither condition is in force.

A major concern of integral geometry is to identify expressions
of Type~\eqref{eqn:buffon-exact} for more involved problems.  This program has led to
a large corpus of exact results, written in terms of geometric invariants, such as
intrinsic volumes or quermassintegrals.  Unfortunately, the classic formulas are difficult
to interpret or to instantiate, except in the simplest cases (e.g., problems in
dimension two or three).  See~\cite{SW08:Stochastic-Integral,Sch14:Convex-Bodies}
for a comprehensive account of this theory.

Meanwhile, the field of asymptotic convex geometry operates in high dimensions.
The main goal of this activity is to
discover qualitative phenomena,
involving interpretable geometric quantities,
more in the spirit of Type~\eqref{eqn:buffon-approx}.
On the other hand, the analytic and probabilistic methods
that are common in this research often lead to statements that include
unspecified constants.  See~\cite{artstein2015asymptotic}
for a recent textbook on asymptotic convex geometry.

This paper establishes precise---but interpretable---approximations of the famous Kubota, Crofton, rotation mean, and kinematic formulas,
which describe how moving convex sets interact.
Our theorems can be viewed as the Type~\eqref{eqn:buffon-approx}
counterparts of results that have only been expressed
in Type~\eqref{eqn:buffon-exact} formulations.
The main ingredient is a set of novel measure concentration results
for the intrinsic volumes of a convex body.
The proofs combine methodology from integral geometry and from asymptotic
convex geometry. 

The most intriguing new insight from our analysis is that each formula splits into two complementary regimes,
where the probability of a geometric event is either negligible
or overwhelming.
Our approach identifies the exact location of the change-point between
the two regimes in terms of a simple summary parameter
that we can (in principle) calculate.
Moreover, in high dimensions, the width of the transition region becomes so narrow
that we can interpret the results as \term{sharp phase transitions}
for problems in Euclidean integral geometry.

\section{Weighted Intrinsic Volumes}

This section begins with basic concepts from convex geometry.
Then we introduce the intrinsic volumes, which are canonical measures of the content of a convex body.
Afterward, we describe different ways to reweight the intrinsic
volumes to make integral geometry formulas more transparent.
Sections~\ref{sec:main-conc} and~\ref{sec:main-phase}
present the main results of the paper.

\subsection{Notation}

This paper uses standard concepts from convexity; see~\cite{Roc70:Convex-Analysis,Bar02:Course-Convexity,SW08:Stochastic-Integral,Sch14:Convex-Bodies}.
The appendices contain background on invariant measures (Appendix~\ref{sec:invariant})
and integral geometry (Appendices~\ref{sec:elements} and~\ref{sec:formulas}).
Here are the basics.

We work in the Euclidean space $\R^n$, equipped with the usual inner product $\ip{\cdot}{\cdot}$
and norm $\norm{\cdot}$. A \term{convex body} is a compact, convex subset of $\R^n$.
The empty set $\emptyset$ is also a convex body.

The \term{Cartesian product} of two sets $\set{S}, \set{T} \subseteq \R^n$ is the set $\set{S} \times \set{T} := \{ (\vct{x}, \vct{y}) \in \R^{2n} : \vct{x} \in \set{S}, \vct{y} \in \set{T} \}$.   The \term{Minkowski sum} is the set
$\set{S} + \set{T} := \{ \vct{x} + \vct{y} \colon \vct{x} \in \set{S}, \vct{y} \in \set{T} \}$.
For a scalar $\lambda \in \R$, we can form the \term{dilation} $\lambda \set{S} := \{ \lambda \vct{s} : \vct{s} \in \set{S} \}$.
Given a linear subspace $\set{L} \subseteq \R^n$, with orthogonal complement $\set{L}^\perp$,
define the \term{orthogonal projection} $\set{S}|\set{L} := \{ \vct{x} \in \set{L} : \set{S} \cap (\vct{x} + \set{L}^\perp) \neq \emptyset \}$. Unless otherwise stated, the term \term{subspace} always means a linear subspace.

The \term{dimension} of a nonempty convex set $\set{K} \subseteq \R^n$ is the dimension of its \term{affine hull},
the smallest affine space that contains $\set{K}$.  When $\dim(\set{K}) = i$, the $i$-dimensional volume $\intvol_i(\set{K})$ is the Lebesgue
measure of $\set{K}$, relative to its affine hull.
When $\set{K}$ is $0$-dimensional (i.e., a single
point), we set $\intvol_0(\set{K}) = 1$.

The notation $\ball{n} := \{ \vct{x} \in \R^n : \norm{\vct{x}} \leq 1 \}$ is reserved for the Euclidean unit ball.
We write $\kappa_n$ for the volume and $\omega_n$ for the surface area of $\ball{n}$.  They satisfy the formulas
\begin{align} \label{eqn:ball-sphere}
\kappa_n &:= \intvol_n(\ball{n}) = \frac{\pi^{n/2}}{\Gamma(1 + n/2)}
&\text{and}&&
\omega_n &:= n \kappa_n = \frac{2\pi^{n/2}}{\Gamma(n/2)}.
\end{align}
As usual, $\Gamma$ denotes the gamma function. 
We instate the convention from combinatorics that a sequence,
indexed by integers, is automatically extended with zero outside its support.
For example, $\kappa_{i} = 0$ for integers $i < 0$,
and the binomial coefficient $\binom{n}{i} = 0$ for integers $i < 0$ and $i > n$.

The operator $\mathbbm{P}$ computes the probability of an event,
while $\Expect$ returns the expectation of a random variable.
The symbol $\sim$ means ``has the distribution.''
We sometimes use infix notation for the minimum
($\wedge$) and maximum ($\vee$) of two numbers.  

\begin{figure}[t]
\begin{center}
\includegraphics[width=0.9\textwidth]{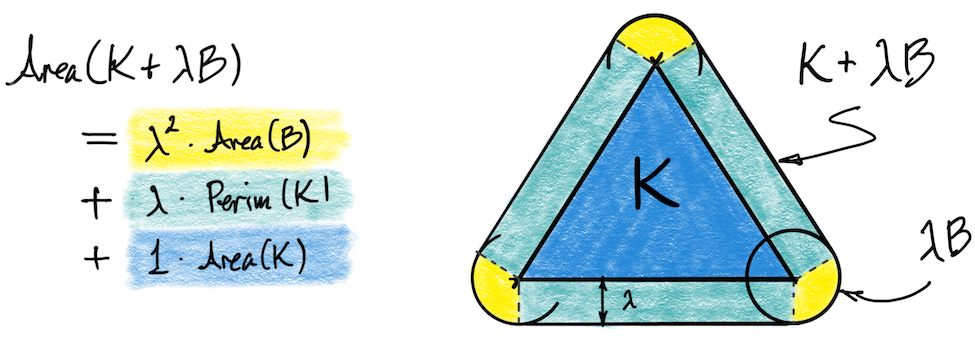}
\caption{\textbf{Steiner's Formula.}  In two dimensions, the sum of a convex body $\set{K}$
and a scaled disc $\lambda \ball{}$ is a quadratic polynomial in the radius $\lambda$.
The coefficients depend on the area and the perimeter of $\set{K}$.  In higher dimensions,
the polynomial expansion involves the intrinsic volumes of the convex body.} \label{fig:steiner}
\end{center}
\end{figure}

\subsection{Intrinsic Volumes}
\label{sec:intvol-def}

Let $\set{K} \subset \R^n$ be a nonempty convex body.
\term{Steiner's formula} states that
the volume of the sum of $\set{K}$ and a scaled Euclidean ball
$\lambda \ball{n}$
can be expressed as a polynomial in the radius $\lambda$ of the ball:
\begin{equation}\label{eqn:steiner-intro}
\intvol_n(\set{K}+\lambda \ball{n}) = \sum_{i=0}^n \lambda^{n-i} \kappa_{n-i} \cdot \intvol_i(\set{K})
\quad\text{for all $\lambda \geq 0$.}
\end{equation}
See Figure~\ref{fig:steiner} for an illustration. The coefficients $\intvol_i(\set{K})$ are called \term{intrinsic volumes} of the
convex body, and the Steiner formula~\eqref{eqn:steiner-intro} serves as their definition.
For the empty set, $\intvol_i(\emptyset) = 0$ for every $i$.

The intrinsic volumes are normalized to be independent
of the ambient dimension in which the convex body
is embedded, so we do not need to specify the dimension in the notation.
In particular, when $\dim(\set{K}) = i$, the $i$th intrinsic volume $\intvol_i(\set{K})$ coincides
with the $i$-dimensional Lebesgue measure of the set,
so the definitions are consistent.

Intrinsic volumes share many properties
of the ordinary volume (Appendix~\ref{sec:intvol-properties}).
In particular, the intrinsic volumes are nonnegative,
they increase with respect to set inclusion, and they
are invariant under rigid motions (i.e., translations, rotations, and reflections).
Furthermore, the $i$th intrinsic volume $\intvol_i$
is positive-homogeneous of degree $i$.
That is, $\intvol_i(\lambda \set{K}) = \lambda^i \,\intvol_i(\set{K})$
for all $\lambda \geq 0$.

Several of the intrinsic volumes have familiar interpretations:
$\intvol_{n-1}(\set{K})$ is half the $(n-1)$-dimensional \term{surface area},
and $\intvol_1(\set{K})$ is proportional to the \term{mean width},
with a factor depending on $n$.  The \term{Euler characteristic}
$\intvol_0(\set{K})$ indicates whether the convex body is nonempty.

Figure~\ref{fig:steiner} suggests that the $i$th intrinsic
volume should reflect the $i$-dimensional content of the body.
Indeed, Kubota's formula (Fact~\ref{fact:projection-intvol})
shows the $\set{V}_i(\set{K})$ is proportional to
the average $i$-dimensional volume of the projections
of $\set{K}$ onto $i$-dimensional subspaces.
Crofton's formula (Fact~\ref{fact:slicing-intvol})
gives a dual representation in terms of the $(n-i)$-dimensional
affine slices of the convex body.

Finally, we introduce the \term{total intrinsic volume},
also known as the \term{Wills functional}~\cite{Wil73:Gitterpunktanzahl,Had75:Willssche}:
\begin{equation} \label{eqn:wills}
\wills(\set{K}) := \sum_{i=0}^n \intvol_i(\set{K}).
\end{equation}
The total intrinsic volume reflects contributions to
the content of the convex body from all dimensions.
It also allows us to compare the size of convex bodies
that have different dimensions.

The total intrinsic volume functional~\eqref{eqn:wills}
is nonnegative, monotone with respect to set inclusion,
and invariant under rigid motions.
Obviously, the total intrinsic volume $\wills(\set{K})$ is comparable
with the maximum intrinsic volume, $\max_i \intvol_i(\set{K})$,
up to a dimensional factor, $\dim( \set{K} )$.  Better estimates
follow from concentration of intrinsic volumes (Theorem~\ref{thm:intvol-conc}).

\subsection{Weighted Intrinsic Volumes}
\label{sec:wvol}

For applications in integral geometry,
Nijenhuis~\cite{nijenhuis1974chern} recognized that
it is valuable to reweight the intrinsic volume sequence;
see also~\cite[pp.~176--177]{SW08:Stochastic-Integral}.
Our work builds on this idea, but we use
different normalizations.
The correct choice of weights depends on which transformation group we are considering,
either rotations or rigid motions.

\begin{definition}[Weighted Intrinsic Volumes] \label{def:wintvol}
Let $\set{K} \subset \R^n$ be a convex body in dimension $n$.  For indices $i = 0,1,2, \dots,n$,
the \term{rotation volumes} $\rotv_i(\set{K})$ and the \term{total rotation volume} $\rotwills(\set{K})$ are the numbers
\begin{align}
\rotv_i(\set{K}) &:= \frac{\omega_{n+1}}{\omega_{i+1}} \intvol_{n-i}(\set{K})
&\text{and}&&
\rotwills(\set{K}) &:= \sum_{i=0}^n \rotv_i(\set{K}). \label{eqn:rotvol-def-intro}
\intertext{The \term{rigid motion volumes} $\rmv_i(\set{K})$ and the \term{total rigid motion volume} $\rmwills(\set{K})$ are} \rmv_i(\set{K}) &:= \frac{\omega_{n+1}}{\omega_{i+1}} \intvol_i(\set{K})
&\text{and}&&
\rmwills(\set{K}) &:= \sum_{i=0}^n \rmv_i(\set{K}).  \label{eqn:rmvol-def-intro}
\end{align}
The notation---a ring for rotations and a bar for rigid motions---is intended to be mnemonic.
\end{definition}

Evidently, the rotation volumes and the rigid motion volumes
depend on the dimension $n$; they are \emph{not} intrinsic.
In the case of rigid motion volumes, this is merely because
of the dispensable factor $\omega_{n+1}$,
which is included to simplify some expressions.
Nevertheless, to lighten notation, we only specify the ambient dimension
when it is required for clarity.

Otherwise, the rotation volumes and rigid motion volumes enjoy the same basic
properties as the intrinsic volumes (Appendix~\ref{sec:intvol-properties}).
In particular, they are nonnegative, monotone, and rigid motion invariant.
The weighted intrinsic volumes also inherit homogeneity properties from the
intrinsic volumes.  Note, however, that the $i$th rotation volume $\rotv_i$
is homogeneous of degree $n - i$ because its indexing reverses that of the
intrinsic volumes.

As we will see, the core formulas of integral geometry simplify
dramatically when they are written using the rotation volumes
and the rigid motion volumes.  We believe that the elegance
of the resulting phase transitions justifies
these unusual normalizations.

\subsection{Relationship with Intrinsic Volumes}

There are many ways to connect intrinsic volumes with rotation
volumes and rigid motion volumes.  In particular, we note two
integral formulas that relate the total volumes.
For a convex body $\set{K} \subset \R^n$,
\begin{equation}\label{eqn:wills-as-moments}
  \rotwills(\set{K}) = \omega_{n+1}\int_0^\infty \wills(s^{-1} \set{K}) \, s^n \econst^{-\pi s^2} \idiff{s}
  \quad\text{and}\quad
  \rmwills(\set{K}) = \omega_{n+1}\int_0^\infty  \wills(s\set{K}) \, \econst^{-\pi s^2} \idiff{s}.
\end{equation}
The identities~\eqref{eqn:wills-as-moments} are obtained from Definition~\ref{def:wintvol}
and the formula~\eqref{eqn:ball-sphere} for the surface area of a ball
by expanding the gamma function in $1/\omega_{i+1}$ using the Euler integral.
In words, the total rotation volume and the total rigid motion volume
reflect how the total intrinsic volume varies with dilation.

\subsection{Characteristic Polynomials}

As with many other sequences, it can be useful to summarize
the volumes using a generating function.
For a convex body $\set{K} \subset \R^n$ and a variable $t \in \R$,
define the characteristic polynomials of the (weighted) intrinsic volumes:
\begin{align}
\chi_{\set{K}}(t) &:= \sum_{i=0}^n t^{n-i} \, \intvol_i(\set{K})
	= t^n \, \wills(t^{-1}\set{K});
	\label{eqn:intvol-charpoly} \\
\mathring{\chi}_{\set{K}}(t) &:= \sum_{i=0}^n t^{n-i} \,\rotv_i(\set{K}) = \rotwills(t \set{K});
	\label{eqn:rotvol-charpoly} \\
\bar{\chi}_{\set{K}}(t) &:= \sum_{i=0}^n t^{n-i}\, \rmv_i(\set{K}) = t^n \,\rmwills(t^{-1} \set{K}).
	\label{eqn:rmvol-charpoly}
\end{align}
On a couple occasions, we will use these polynomials to streamline computations.
The main results of this paper can also be framed in terms of
characteristic polynomials.  For brevity, we will not
pursue this observation.

\subsection{Examples}

It is usually quite difficult to compute all of the intrinsic volumes of a convex body,
but we can provide explicit formulas in some simple cases.  The discussion in this section
supports our numerical work, but it is tangential to the main arc of development.

\begin{example}[Euclidean Ball: Intrinsic Volumes] \label{ex:ball}
The intrinsic volumes of the scaled Euclidean ball $\lambda \ball{n}$ take the form
\begin{equation} \label{eqn:intvol-ball}
\intvol_i(\lambda \ball{n}) = \binom{n}{i} \frac{\kappa_n}{\kappa_{n-i}} \cdot \lambda^i
\quad\text{for all $\lambda \geq 0$.} \end{equation}
The relation~\eqref{eqn:intvol-ball} is an easy consequence of the Steiner formula~\eqref{eqn:steiner-intro}.
The total intrinsic volume~\eqref{eqn:wills} of the ball satisfies
\begin{equation} \label{eqn:ball-wills}
\wills(\lambda\ball{n}) =  \kappa_n \lambda^n + \omega_n \int_0^\infty (\lambda+s)^{n-1} \,\econst^{-\pi s^2} \idiff{s}.
\end{equation}
This expression can be derived by expanding the gamma function in the factor $1/\kappa_{n-i}$
as an Euler integral; it also follows from the integral representation in Corollary~\ref{cor:intvol-metric}.
While this integral cannot be evaluated in closed form, it can be approximated using Laplace's method
or numerical quadrature.

\end{example}

\begin{example}[Euclidean Ball: Rigid Motion Volumes]\label{ex:ball-rm}
By definition~\eqref{eqn:rmvol-def-intro},
the rigid motion volumes of the scaled Euclidean ball $\lambda \ball{n}$
satisfy
\begin{equation*}
\frac{\rmv_i(\lambda\ball{n})}{\omega_{n+1}}= \binom{n}{i}\frac{\kappa_n}{\kappa_{n-i}\omega_{i+1}} \cdot \lambda^i.
\end{equation*}
After some gymnastics with gamma functions,
this expression converts into a form that is reminiscent
of a binomial distribution or a beta distribution.

To obtain a compact expression for the total rigid motion volume,
use the integral representation~\eqref{eqn:wills-as-moments},
the total intrinsic volume computation~\eqref{eqn:ball-wills}, and the volume computations~\eqref{eqn:ball-sphere}
to arrive at
\begin{equation} \label{eqn:rmvol-ball-integral}
 \rmwills(\lambda \ball{n}) = \omega_n \left[\frac{\lambda^n}{n}+\int_0^{\pi/2} (\lambda \sin\theta + \cos\theta )^{n-1} \idiff{\theta} \right].
\end{equation}
If desired, one may invoke Laplace's method to obtain an asymptotic approximation
for the total rigid motion volume.
\end{example}

\begin{example}[Parallelotopes: Intrinsic Volumes] \label{ex:ptope}
Given a vector $\vct{\lambda} := (\lambda_1, \dots, \lambda_n)$ of nonnegative
side lengths, we construct the parallelotope
$\set{R}_{\vct{\lambda}} := [0,\lambda_1] \times \cdots \times [0, \lambda_n] \subset \R^n$.
Its intrinsic volumes are
\begin{equation*} \label{eqn:ptope-intvol}
\intvol_i(\set{R}_{\vct{\lambda}}) = e_i(\vct{\lambda})
\quad\text{for $i = 0, 1, 2, \dots, n$.}
\end{equation*}
We have written $e_i$ for the $i$th elementary symmetric polynomial.
Thus, the total intrinsic volume is
$$
\wills(\set{R}_{\vct{\lambda}}) = \prod_{i=1}^n (1 + \lambda_i).
$$
These formulas follow from an expression (Fact~\ref{fact:product-intvol})
for the intrinsic volumes of a Cartesian product.
\end{example}

\begin{example}[Scaled Cube: Weighted Intrinsic Volumes] \label{ex:cube}
Let $\set{Q}^n := [0,1]^n \subset \R^n$ be the standard cube.
For each $\lambda \geq 0$ and each index $i = 0,1,2, \dots, n$,
\begin{equation} \label{eqn:intvol-cube}
\intvol_i(\lambda \set{Q}^n) = \binom{n}{i} \cdot \lambda^i
\quad\text{and}\quad
\wills(\lambda \set{Q}^n) = (1+\lambda)^n
\quad\text{for all $\lambda \geq 0$.} \end{equation}
This result specializes Example~\ref{ex:ptope}.

It is also straightforward to determine the weighted intrinsic volumes of a scaled cube:
\begin{align*}
\rotv_i(\lambda \set{Q}^n) &= \binom{n+1}{n-i} \frac{\kappa_{n+1}}{\kappa_{i+1}} \cdot \lambda^{n-i}
	= \intvol_{n-i}(\lambda \ball{n+1}); \\
\rmv_i(\lambda \set{Q}^n) &= \binom{n+1}{n-i} \frac{\kappa_{n+1}}{\kappa_{i+1}} \cdot \lambda^i
	= \lambda^{n} \, \intvol_{n-i}(\lambda^{-1} \ball{n+1}).
\end{align*}
These calculations follow from Definition~\ref{def:wintvol},
the homogeneity properties of the intrinsic volumes,
and the calculations~\eqref{eqn:intvol-ball} and~\eqref{eqn:intvol-cube}. In view of~\eqref{eqn:ball-wills}, the total volumes satisfy
\begin{align}
\rotwills(\lambda \set{Q}^n) &= \wills(\lambda \ball{n+1})-\intvol_{n+1}(\lambda \ball{n+1})
	= \omega_{n+1}\int_0^\infty (\lambda+s)^{n} \,\econst^{-\pi s^2} \idiff{s};
	\label{eqn:cube-rotwills} \\
\rmwills(\lambda \set{Q}^n) &= \lambda^{n} \cdot \big[\wills(\lambda^{-1} \ball{n+1})-\intvol_{n+1}(\lambda^{-1} \ball{n+1}) \big]
	= \omega_{n+1}\int_0^\infty (1+\lambda s)^{n} \,\econst^{-\pi s^2} \idiff{s}.
	\label{eqn:cube-rmwills}
\end{align}
These identities also follow from~\eqref{eqn:wills-as-moments}.
\end{example}

\section{Main Results I: Concentration}
\label{sec:main-conc}

With our coauthors Michael McCoy, Ivan Nourdin, and Giovanni Peccati,
we recently established the surprising fact that the sequence of intrinsic volumes
concentrates sharply around its centroid~\cite[Thm.~1.11]{LMNPT20:Concentration-Euclidean}.
The main technical achievement of this paper is a set
of concentration properties for the rotation volumes
and the rigid motion volumes.
In Section~\ref{sec:main-phase}, we present phase
transition formulas that follow from these concentration results.

\subsection{Weighted Intrinsic Volume Random Variables}

It is convenient to construct random variables that reflect
the shape of the sequences of weighted intrinsic volumes (Definition~\ref{def:wintvol}).
This approach gives us access to the language and methods
of probability theory.

\begin{definition}[Weighted Intrinsic Volume Random Variables] \label{def:wintvol-rv}
Let $\set{K} \subset \R^n$ be a nonempty convex body.
The \term{rotation volume random variable} $\rotI_{\set{K}}$ 
and the \term{rigid motion volume random variable} $\rmI_{\set{K}}$
take values in $\{0,1,2,\dots,n\}$ according to the distributions \begin{align} \label{eqn:wvol-rv-def} \Prob{ \rotI_{\set{K}} = n - i } &= \frac{\rotv_i(\set{K})}{\rotwills(\set{K})}
&\text{and}&&
\Prob{ \rmI_{\set{K}} = n - i } &= \frac{\rmv_i(\set{K})}{\rmwills(\set{K})}.
\end{align}
The \term{intrinsic volume random variable} $I_{\set{K}}$ is defined analogously.
Note that each random variable reverses the indexing of the corresponding sequence.
\end{definition}

Figure~\ref{fig:intvoldist} illustrates the distribution of the weighted
intrinsic volume random variables for several convex bodies.  The
picture indicates that each of the random variables is concentrated.
In each case, the point of concentration is the expectation
of the random variable, which merits its own terminology.

\begin{definition}[Central Volumes] \label{def:central-vols}
Let $\set{K} \subset \R^n$ be a nonempty convex body.
The \term{central rotation volume} $\rotdelta(\set{K})$
is the expectation of the rotation volume random variable $\rotI_{\set{K}}$,
while the \term{central rigid motion volume} $\rmdelta(\set{K})$
is the expectation of the rigid motion random variable $\rmI_{\set{K}}$:
\begin{align*} \label{eqn:central-volumes}
\rotdelta(\set{K}) &:= \Expect \rotI_{\set{K}} = \frac{1}{\rotwills(\set{K})} \sum_{i=0}^{n} (n-i) \, \rotv_i(\set{K})
   	= \frac{\diff{}}{\diff{t}}\Big\vert_{t=1} \log \mathring{\chi}_{\set{K}}(t); \\
\rmdelta(\set{K}) &:= \Expect \rmI_{\set{K}} = \frac{1}{\rmwills(\set{K})} \sum_{i=0}^{n} (n-i) \, \rmv_i(\set{K})
	= \frac{\diff{}}{\diff{t}}\Big\vert_{t=1} \log \bar{\chi}_{\set{K}}(t).
\end{align*}
The central intrinsic volume $\delta(\set{K})$ is defined analogously.
The characteristic polynomials $\mathring{\chi}$ and $\bar{\chi}$ are
given in~\eqref{eqn:rotvol-charpoly} and~\eqref{eqn:rmvol-charpoly}.
\end{definition}

Each of the central volumes lies in the interval $[0, n]$.
If we rescale a fixed body $\set{K}$, the central rotation volume $\rotdelta(\lambda \set{K})$
increases monotonically from $0$ to $n$ as the scale parameter $\lambda$ increases from $0$ to $\infty$.
Oppositely, the central rigid motion volume $\rmdelta(\lambda \set{K})$
decreases monotonically from $n$ to $0$. This feature may be confusing, but it reflects a duality between the two notions of volume.

\begin{figure}[t]
\begin{subfigure}[b]{0.47\textwidth}
\includegraphics[width=\textwidth]{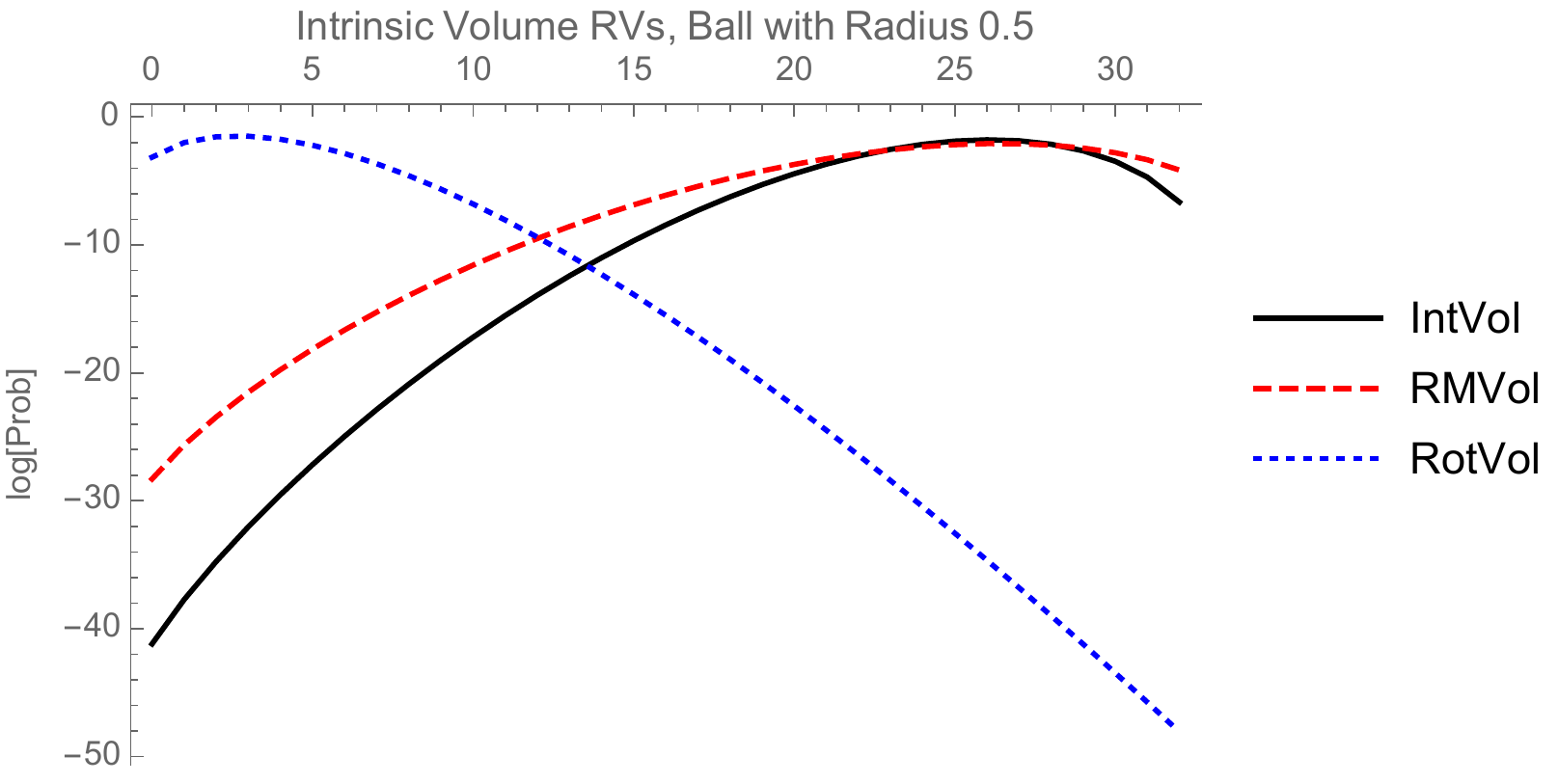}
\end{subfigure} \hspace{1pc}
\begin{subfigure}[b]{0.47\textwidth}
\includegraphics[width=\textwidth]{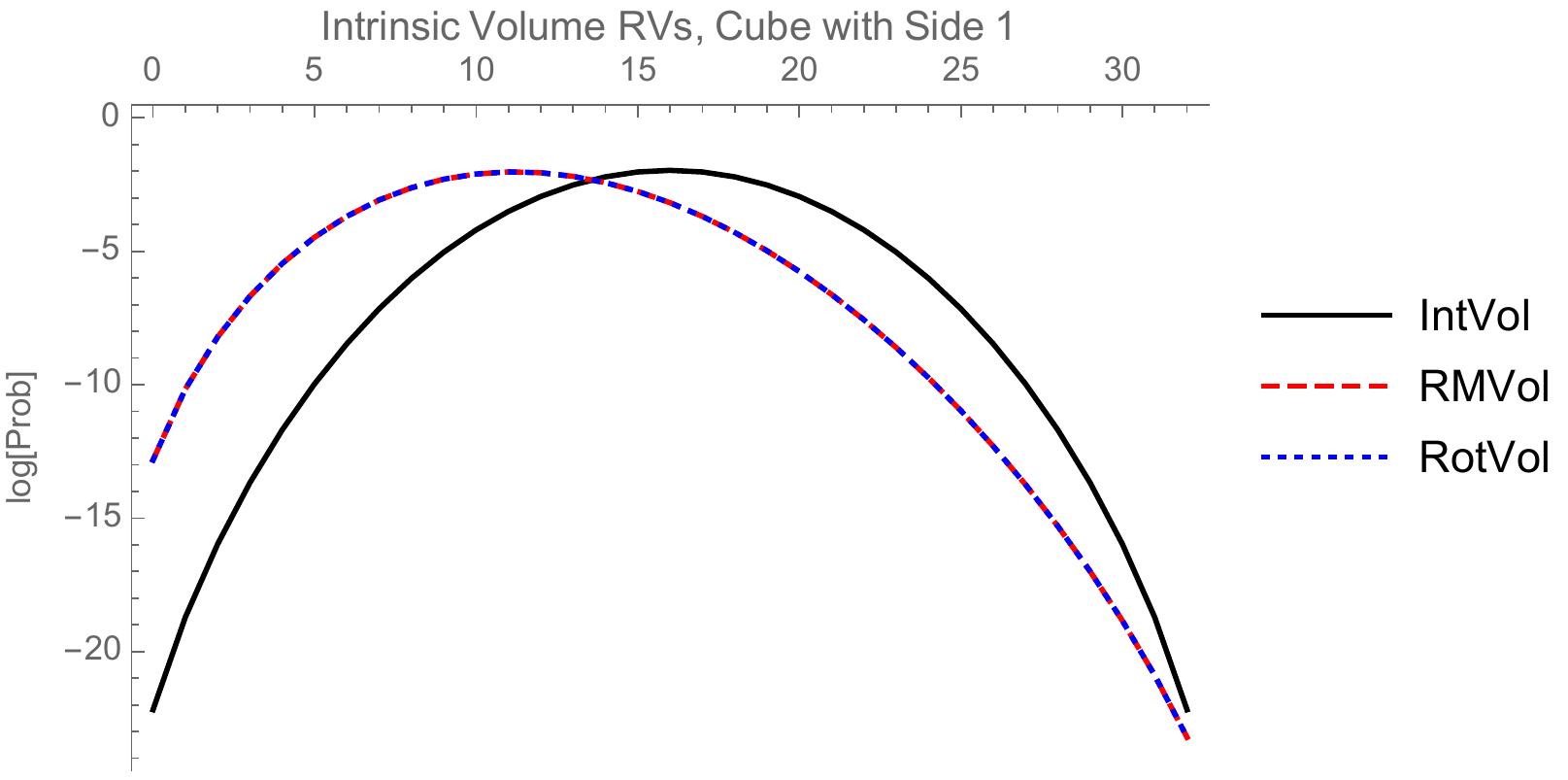}
\end{subfigure} \\
\begin{subfigure}[b]{0.47\textwidth}
\includegraphics[width=\textwidth]{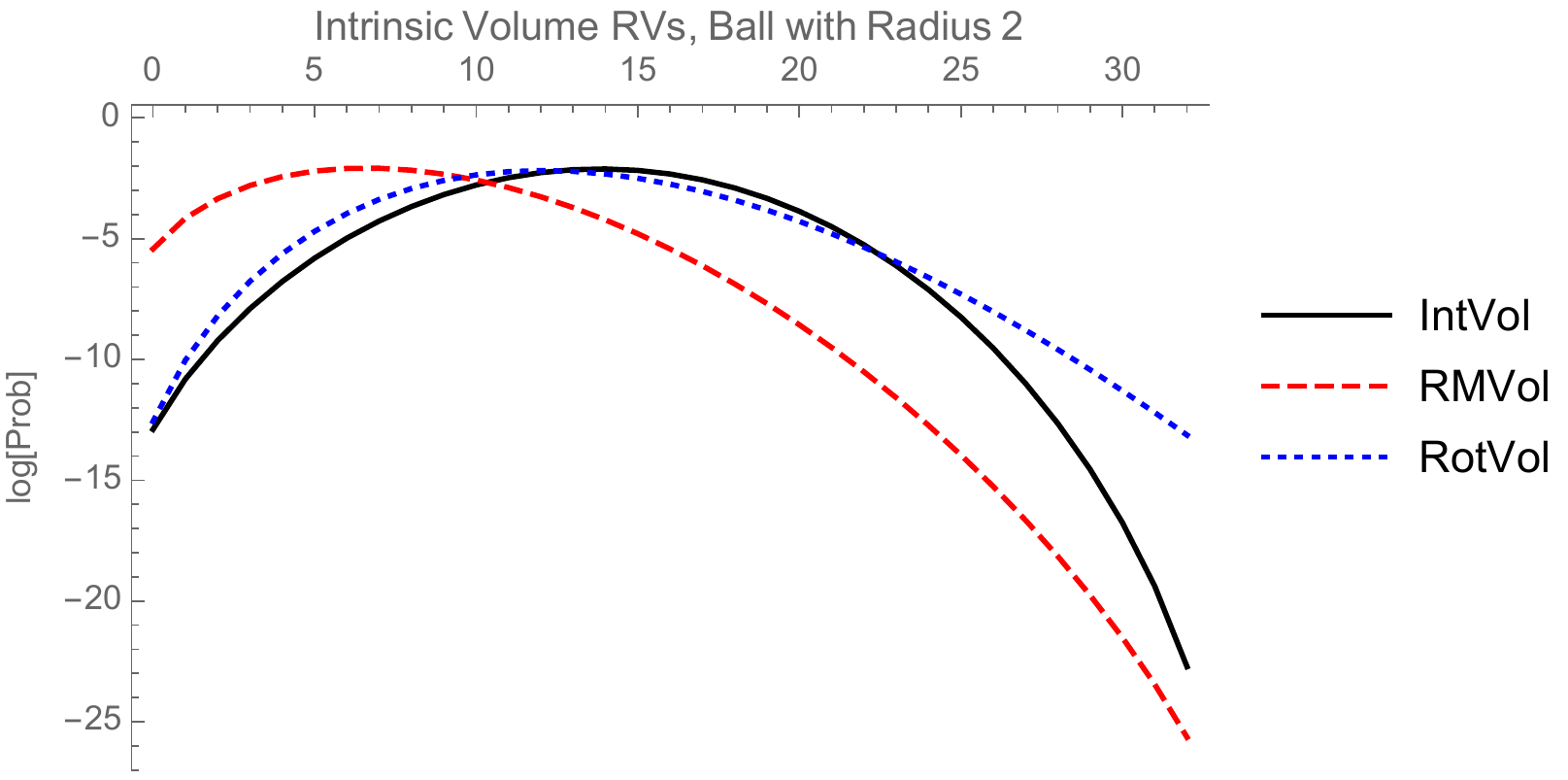}
\end{subfigure}  \hspace{1pc}
\begin{subfigure}[b]{0.47\textwidth}
\includegraphics[width=\textwidth]{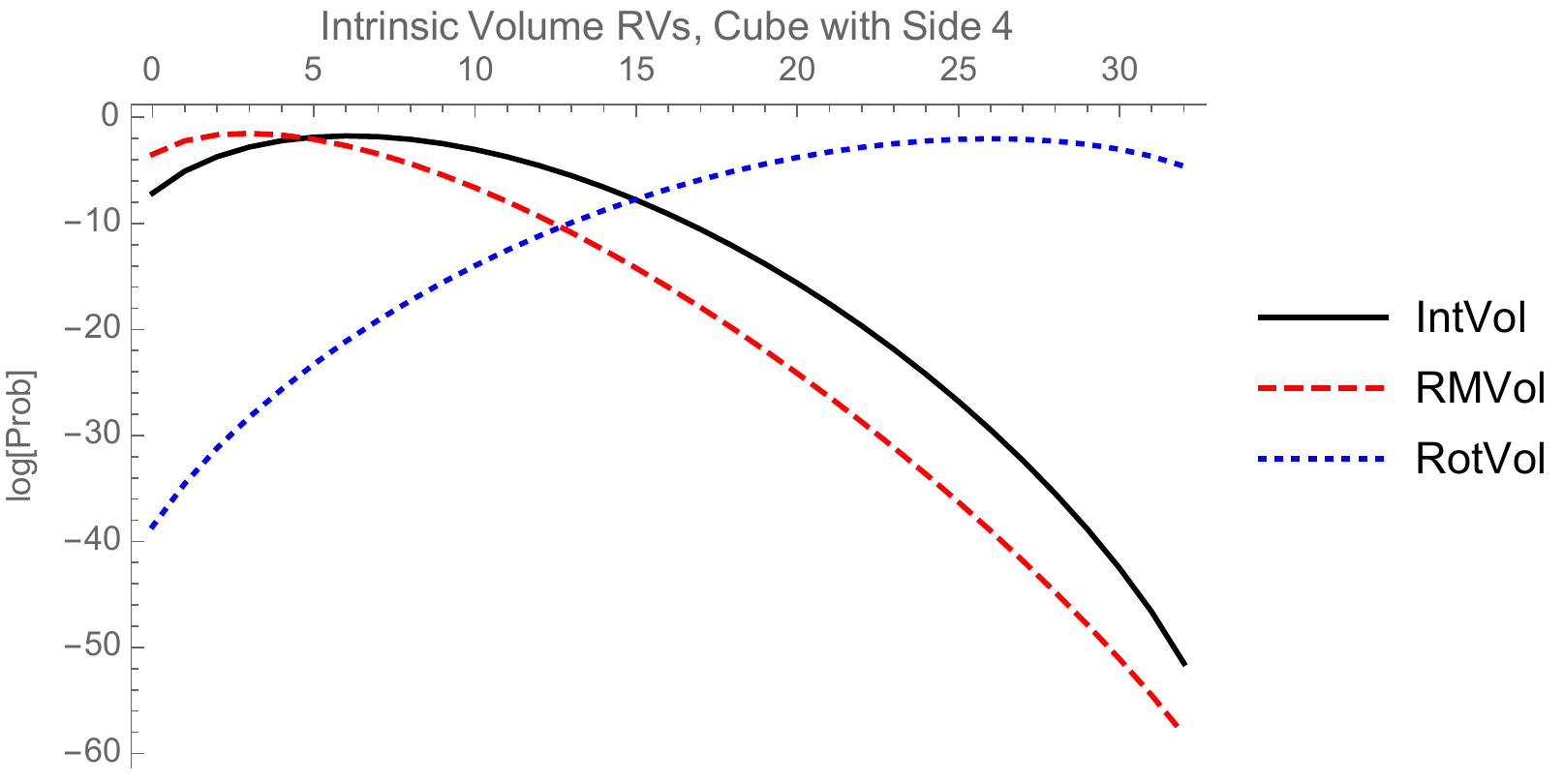}
\end{subfigure}
\caption{\textbf{Weighted Intrinsic Volume Random Variables.}  The distribution
of the weighted intrinsic volume random variables (Definition~\ref{def:wintvol-rv})
for scaled balls and scaled cubes in $\R^{32}$,
calculated using Examples~\ref{ex:ball}, ~\ref{ex:ball-rm}, and~\ref{ex:cube}.
The horizontal axis is the value of the random variable;
the vertical axis is the {natural logarithm} of the probability that it takes that value.
}
\label{fig:intvoldist}
\end{figure}

\begin{example}[Scaled Cube: Central Rotation Volume]\label{ex:scaled-cube-delta}
From the discussion in Example~\ref{ex:cube}, we can derive expressions for the central rotation volumes of the scaled cube:
\begin{equation*}
  \rotdelta(\lambda \set{Q}^n)
  	= \frac{\diff{}}{\diff{t}}\Big\vert_{t=1} \log \mathring{\chi}_{\lambda \set{Q}^n}(t)
= \lambda n \cdot \frac{\int_0^\infty (\lambda+s)^{n-1} \,\econst^{-\pi s^2} \idiff{s}}{\int_0^\infty (\lambda+s)^{n} \,\econst^{-\pi s^2} \idiff{s}}.
\end{equation*}
We have used the expression~\eqref{eqn:cube-rotwills} for the total rotation volume.
This formula can be evaluated asymptotically in an appropriate parameter regime.
For each $\zeta > 0$,
\begin{equation}\label{eqn:central-scaled-cube}
  \frac{\rotdelta\big(\zeta \sqrt{2n/\pi} \, \set{Q}^n \big)}{n} \longrightarrow \frac{2}{1+\sqrt{1+\zeta^{-2}}}
  \quad\text{as $n\to \infty$.}
\end{equation}
The statement~\eqref{eqn:central-scaled-cube} follows from a routine application of Laplace's method.
There is a related expression for the asymptotics of the central rigid motion volume of a scaled cube.
\end{example}

\begin{example}[Scaled Ball: Central Rigid Motion Volume]\label{ex:scaled-ball-delta}
From the discussion in Example~\ref{ex:ball-rm}, we can derive expressions for the
central rigid volume volumes of the scaled ball.  Indeed, \begin{equation*}
  \rmdelta(\lambda \ball{n})
  	= \frac{\diff{}}{\diff{t}}\Big\vert_{t=1} \log \bar{\chi}_{\lambda \ball{n}}(t)
= \frac{\lambda^n + (n-1)\lambda  \int_0^{\pi/2} ( \lambda \sin\theta + \cos\theta )^{n-2} \sin\theta \idiff{\theta}}
	{n^{-1} \lambda^n + \int_0^{\pi/2} ( \lambda \sin\theta + \cos\theta )^{n-1} \idiff{\theta}}.
\end{equation*}
We have used the expression~\eqref{eqn:rmvol-ball-integral}.
For each $\lambda>0$, \begin{equation}\label{eqn:central-scaled-ball}
   \frac{\rmdelta(\lambda \ball{n} )}{n} \longrightarrow \frac{1}{1 + \lambda^2}
  \quad\text{as $n\to \infty$.}
\end{equation}
The statement~\eqref{eqn:central-scaled-ball} follows after a careful application of Laplace's method.
\end{example}

\subsection{Weighted Intrinsic Volumes Concentrate}
\label{sec:wvol-conc-intro}

The concentration behavior visible in Figure~\ref{fig:intvoldist}
is a generic property of every convex body.  The following theorems
quantify the rate at which the distribution of the weighted intrinsic
volume random variable decays away from the central volume.

\begin{bigthm}[Rotation Volumes: Concentration] \label{thm:rotvol-intro}
Consider a nonempty convex body $\set{K} \subset \R^n$ with rotation volume random variable $\rotI_{\set{K}}$
and central rotation volume $\rotdelta(\set{K})$.
For all $t \geq 0$,
$$
\Prob{ \abs{ \phantom{\big|\!} \smash{\rotI_{\set{K}} - \rotdelta(\set{K}) }} \geq t }
	\leq 2 \exp\left( \frac{-t^2/2}{\rotsigma^2(\set{K}) + t/3} \right) \quad\text{where}\quad
	\rotsigma^2(\set{K}) := \rotdelta(\set{K}).
$$ 
\end{bigthm}

\noindent
Theorem~\ref{thm:rotvol-intro} is a consequence of Theorem~\ref{thm:rotvol-conc}.

According to Theorem~\ref{thm:rotvol-intro}, the rotation volume random variable
$\rotI_{\set{K}}$ exhibits Bernstein-type concentration around the central rotation volume $\rotdelta(\set{K})$.
Initially, for small $t$, the probability decays at least as fast as a Gaussian random variable with variance
$\rotdelta(\set{K})$.  In other words, the bulk of the rotation volume distribution is concentrated on about $\rotdelta(\set{K})^{1/2}$ indices near $\rotdelta(\set{K})$.
When $t \approx \rotdelta(\set{K})$, the decay shifts from Gaussian
to exponential with variance one.  In fact, the proof of the result demonstrates
that stronger Poisson-type decay occurs.

A parallel result, with a similar interpretation, holds for the rigid motion volumes.

\begin{bigthm}[Rigid Motion Volumes: Concentration] \label{thm:rmvol-intro}
Consider a nonempty convex body $\set{K} \subset \R^n$ with rigid motion volume random variable $\rmI_{\set{K}}$
and central rigid motion volume $\rmdelta(\set{K})$.
For all $t \geq 0$,
$$
\Prob{ \abs{ \rmI_{\set{K}} - \rmdelta(\set{K}) } \geq t }
	\leq 2 \exp\left( \frac{-t^2/4}{\rmsigma^2(\set{K}) + t/3} \right)
	\quad\text{where}\quad
	\rmsigma^2(\set{K}) := \rmdelta(\set{K}) \wedge \big((n + 1) - \rmdelta(\set{K})\big).
$$ 
\end{bigthm}

\noindent
Theorem~\ref{thm:rmvol-intro} follows from Theorem~\ref{thm:rmvol-conc}.

\subsection{Metric Formulations}

Although it is difficult to calculate weighted intrinsic volumes,
the concentration theory demonstrates that the central volume is
already an adequate summary of the entire distribution.  Fortunately,
the central volume has an alternative expression that is more tractable.
In the following results, $\dist_{\set{K}}(\vct{x})$ reports the
Euclidean distance from the point $\vct{x}$ to the set $\set{K}$.

\begin{proposition}[Rotation Volumes: Metric Formulation] \label{prop:rotvol-metric}
The total rotation volume and the central rotation volume
of a nonempty convex body $\set{K} \subset \R^n$
may be calculated as
\begin{align*}
\rotwills(\set{K}) &= \frac{\omega_{n+1}}{2} \int_{\R^n} \econst^{-2\pi \dist_{\set{K}}(\vct{x})} \idiff{\vct{x}}; \\
n - \rotdelta(\set{K}) &= \frac{\omega_{n+1}}{2 \rotwills(\set{K})} \int_{\R^n} 2\pi\dist_{\set{K}}(\vct{x})\, \econst^{-2\pi\dist_{\set{K}}(\vct{x})} \idiff{\vct{x}}.
\end{align*}
\end{proposition}

\begin{proposition}[Rigid Motion Volumes: Metric Formulation] \label{prop:rmvol-metric}
The total rigid motion volume and the central rigid motion volume
of a nonempty convex body $\set{K} \subset \R^n$ may be calculated as
\begin{align*}
\rmwills(\set{K}) &= \int_{\R^n} \big[ 1 + \dist_{\set{K}}^2(\vct{x}) \big]^{-(n+1)/2} \idiff{\vct{x}}; \\
(n+1) - \rmdelta(\set{K}) &= \frac{n+1}{\rmwills(\set{K})} \int_{\R^n} \big[ 1 + \dist_{\set{K}}^2(\vct{x}) \big]^{-(n+3)/2} \idiff{\vct{x}}.
\end{align*}
\end{proposition}

\noindent
The proof of Proposition~\ref{prop:rotvol-metric} appears
in Section~\ref{sec:rotvol-distint}.
The proof of Proposition~\ref{prop:rmvol-metric}
is in Section~\ref{sec:rmvol-distint}.
An analogous result for intrinsic volumes
appears as Corollary~\ref{cor:intvol-metric}.

The central rotation volume is essentially given by the entropy of the log-concave probability measure
induced by the distance to the body.
The central rigid motion volume has a related structure, but it is expressed in terms of a concave measure.
Owing to the strikingly different form of these measures, the concentration results
require distinct technical ingredients.

\subsection{Proof Strategy}

Both Theorem~\ref{thm:rotvol-intro} and Theorem~\ref{thm:rmvol-intro}
follow from the same species of argument,
which we execute in Sections~\ref{sec:genfun}--\ref{sec:rmvol-conc}.
Let us give a summary.

To study the behavior of the weighted intrinsic volume random variable,
we use a refined version of the entropy method. The argument depends on the insight, from~\cite{LMNPT20:Concentration-Euclidean}, that Steiner's formula~\eqref{eqn:steiner-intro} allows us to pass
from the discrete random variable on $\{0, 1, 2, \dots, n\}$
to a continuous random variable on $\R^n$ with a concave measure.
To understand the fluctuations of the continuous distribution,
we apply modern variance bounds for concave measures,
in the spirit of the classic Borell--Brascamp--Lieb inequality~\cite{Bor75:Convex-Set,BL76:Extensions-Brunn-Minkowski}.
In each case, the details are somewhat different.

\subsection{Intrinsic Volumes Concentrate Better}

Using the same methodology,
we have also established a concentration inequality
for the intrinsic volume random variable that improves
significantly over our results in~\cite{LMNPT20:Concentration-Euclidean}.
This material is not directly relevant to our
phase transition project, so we have postponed the statement
and proof to Appendix~\ref{app:intvol}.

\subsection{Log-Concavity of Weighted Intrinsic Volumes}

As a counterpoint to our main results, we remark that the
weighted intrinsic volume sequences are log-concave.

\begin{proposition}[Weighted Intrinsic Volumes: Log-Concavity] \label{prop:lc}
Let $\set{K} \subset \R^n$ be a convex body.  For each index $i = 1, 2, 3, \dots, n - 1$,
\begin{align}
\rotv_i(\set{K})^2 &\geq \rotv_{i-1}(\set{K}) \cdot \rotv_{i+1}(\set{K}); \\
\rmv_i(\set{K})^2 &\geq \rmv_{i-1}(\set{K}) \cdot \rmv_{i+1}(\set{K}).
\end{align}
\end{proposition}

Proposition~\ref{prop:lc} gives an alternative sense
in which the weighted intrinsic volume sequences concentrate.
Among other things, it ensures that both sequences are unimodular.
See Section~\ref{sec:intvol-ulc} for related results
about the intrinsic volume sequence.

\begin{proof}[Proof Sketch]
The Alexandrov--Fenchel inequalities~\cite[Sec.~7.3]{Sch14:Convex-Bodies}
imply that the intrinsic volumes of a convex body form an
ultra-log-concave sequence~\cite{Che76:Processus-Gaussiens,McM91:Inequalities-Intrinsic}.
Using Definition~\ref{def:wintvol} and the volume calculation~\eqref{eqn:ball-sphere},
we quickly conclude that the weighted intrinsic volumes are log-concave.
\end{proof}

After this paper was written, Aravind et al.~\cite{AMM21:Concentration-Inequalities}
proved that all ultra-log-concave sequences exhibit Poisson-type concentration.
This provides an alternative approach to proving slightly weaker forms
of some of the concentration theorems in this paper.

\section{Main Results II: Phase Transitions}
\label{sec:main-phase}

The primary goal of this paper is to reveal that several classic
formulas from integral geometry exhibit sharp phase transitions
in high dimensions.  In each setting, the formula collapses into two
complementary cases when we invoke the concentration properties
of weighted intrinsic volumes.  The width of the transition region between
the two cases becomes negligible as the dimension increases.

\subsection{Rotations}

We begin with two models that involve averaging over random rotations.
The first involves projection onto a random subspace,
and the second involves the sum of a convex body and
a randomly rotated convex body.  In each case,
the concentration of rotation volumes
leads to phase transition phenomena.

\begin{figure}[t]
\begin{center}
\includegraphics[width=0.67\textwidth]{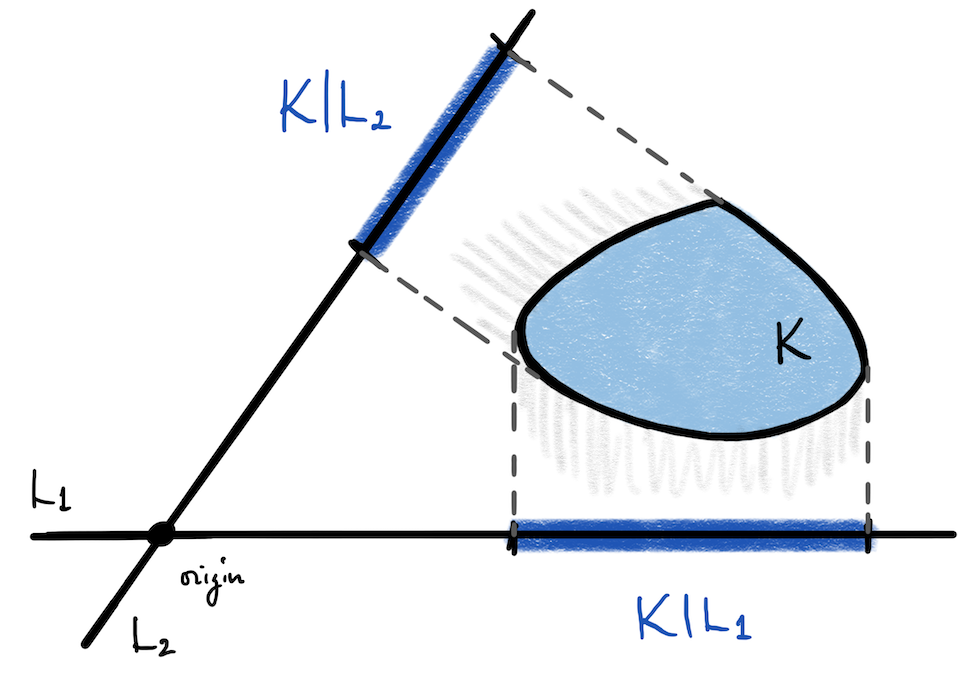}
\caption{\textbf{Projection onto a Random Subspace.}  This diagram shows
the orthogonal projections of a two-dimensional convex body $\set{K}$ onto
the one-dimensional subspaces $\set{L}_1$ and $\set{L}_2$.
The projection formula, Fact~\ref{fact:proj-formula-intro},
states the expected length of the projection, averaged over all one-dimensional subspaces,
is proportional to the perimeter of $\set{K}$.  The formula also treats problems
in higher dimensions.} \label{fig:random-projection}
\end{center}
\end{figure}

\begin{figure}[t]
\begin{center}
\begin{subfigure}[b]{0.45\textwidth}
\includegraphics[width=\textwidth]{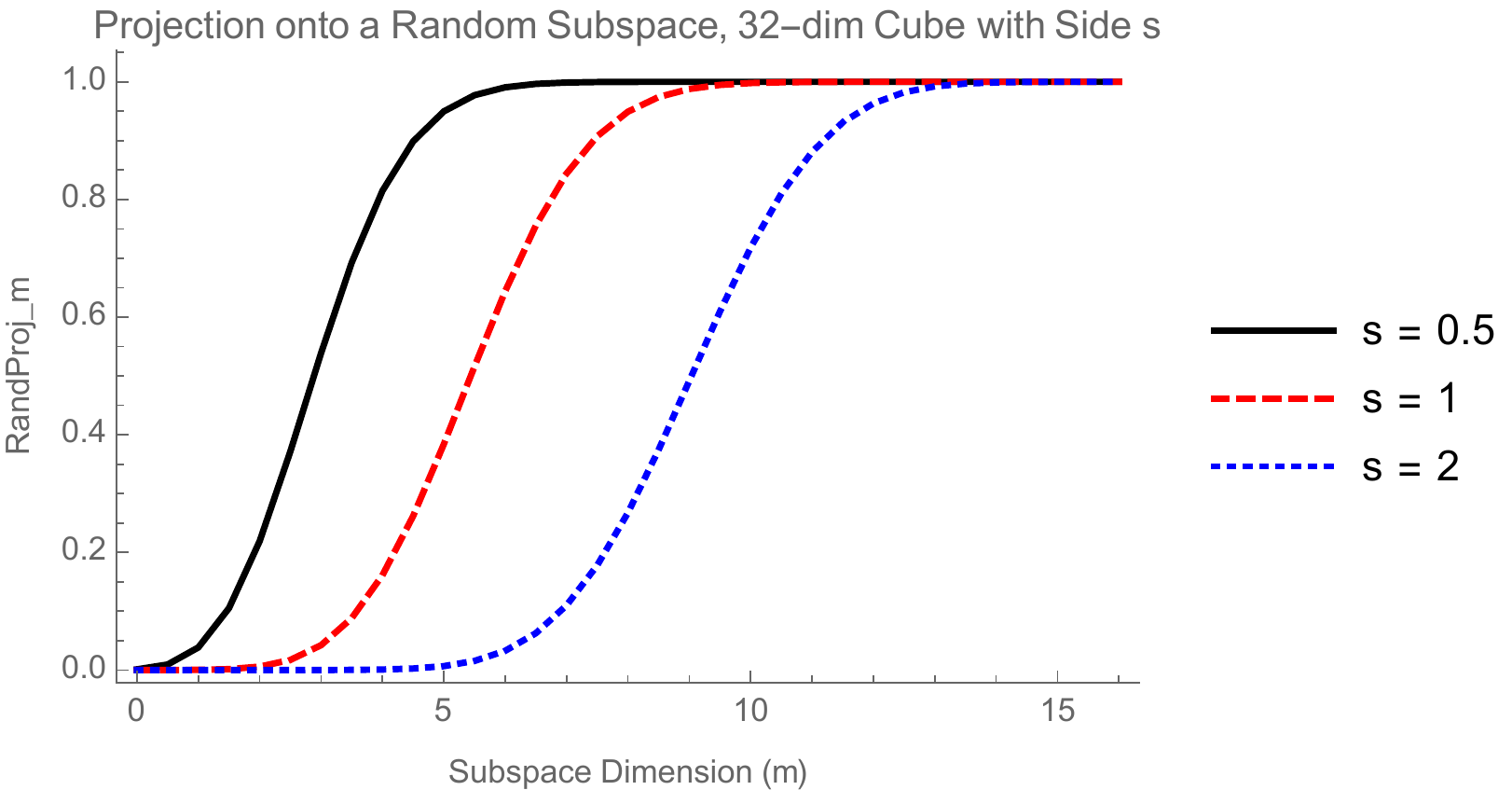}
\end{subfigure}
\hspace{1pc}
\begin{subfigure}[b]{0.45\textwidth}
\includegraphics[width=\textwidth]{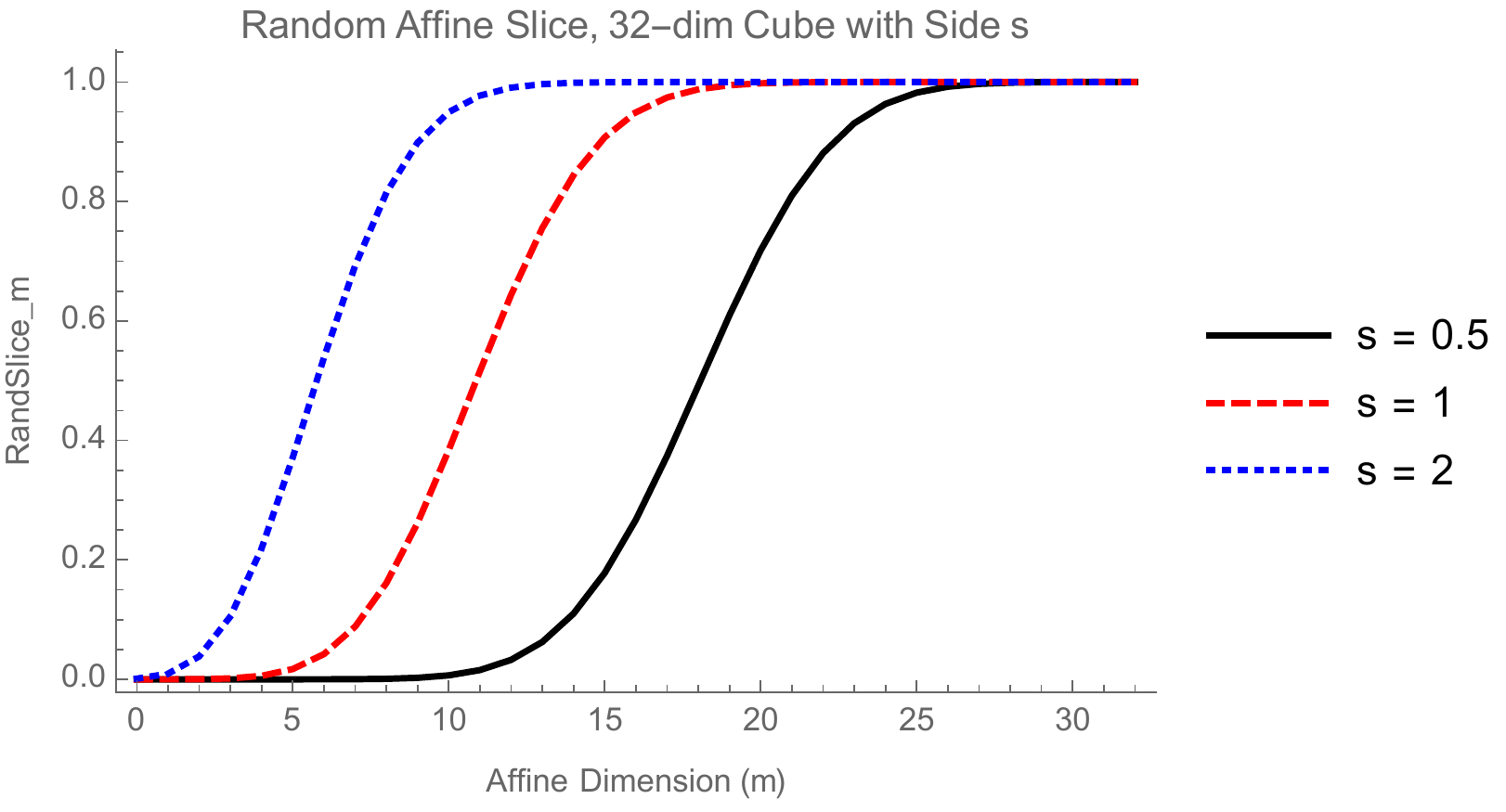}
\end{subfigure} \\
\begin{subfigure}[b]{0.45\textwidth}
\includegraphics[width=\textwidth]{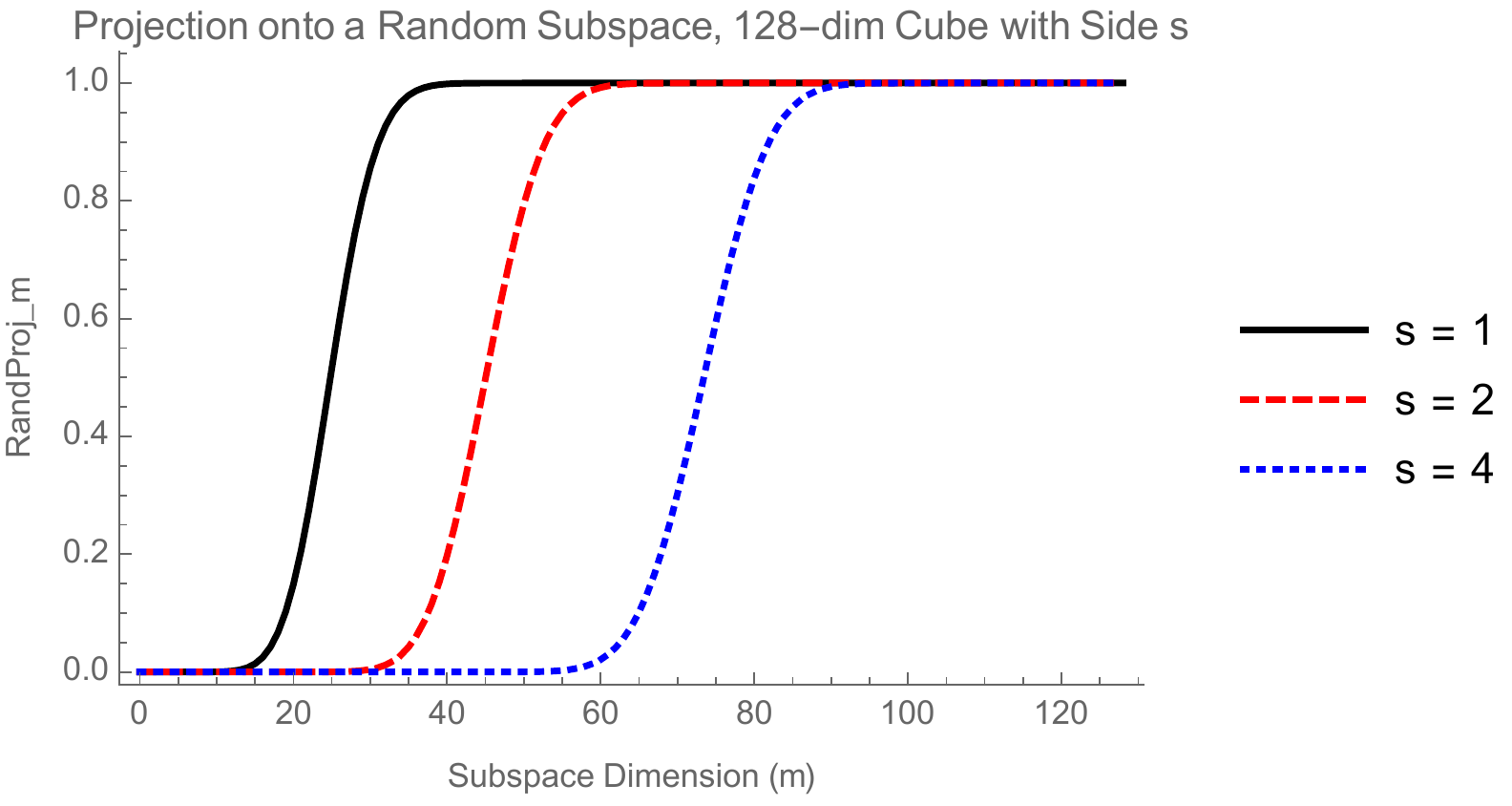}
\end{subfigure}
\hspace{1pc}
\begin{subfigure}[b]{0.45\textwidth}
\includegraphics[width=\textwidth]{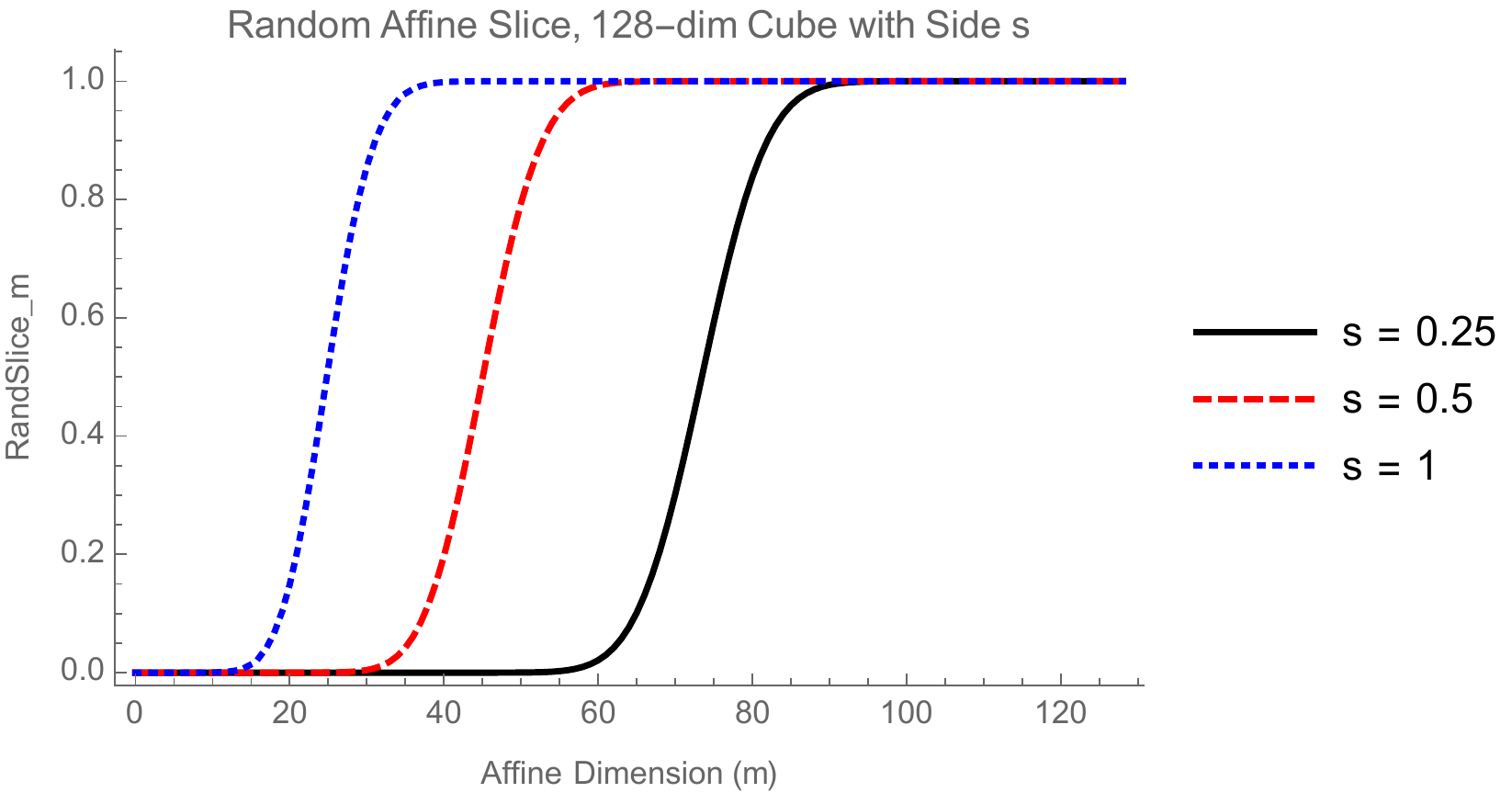}
\end{subfigure} \\
\caption{\textbf{Phase Transitions for Moving Flats.} \textsc{[left]}
The function $\mathrm{RandProj}_m$, defined in~\eqref{eqn:randproj-intro},
compares the total rotation volume of a random $m$-dimensional projection of a scaled cube $s\set{Q}^n$
with the total rotation volume of the scaled cube.
The transition moves \emph{right} with increasing scale $s$.
\textsc{[right]} The function $\mathrm{RandSlice}_m$, defined in~\eqref{eqn:randslice-intro},
compares the total rigid motion volume of a random $m$-dimensional
affine slice of a scaled cube $s \set{Q}^n$
with the total rigid motion volume of the scaled cube.
The transition moves \emph{left} with increasing scale $s$.
\textsc{[top]} $\R^{32}$ and \textsc{[bottom]} $\R^{128}$.
Note that the relative transition width becomes narrower as the ambient dimension increases.}
\label{fig:randproj-randslice}
\end{center}
\end{figure}

\subsubsection{Projection onto a Random Subspace}

The first question concerns the interaction
between a set and a random subspace:

\begin{quotation} \textcolor{dkblue}{
\textbf{Suppose that we project a convex body onto a random
linear subspace of a given dimension.  How big is the
image relative to the original set?}}
\end{quotation}

\noindent
See Figure~\ref{fig:random-projection} for a schematic.
The integral geometry literature contains a detailed
answer to this question.  The result is usually
written in terms of the intrinsic volumes
(Fact~\ref{fact:projection-intvol}), but
we have discovered that the rotation
volumes lead to a simpler statement:

\begin{fact}[Projection Formula] \label{fact:proj-formula-intro}
Let $\set{K} \subset \R^n$ be a nonempty convex body.
For each subspace dimension $m = 0, 1, 2, \dots, n$
and each index $i = 0, 1, 2, \dots, m$,
\begin{equation} \label{eqn:proj-formula-intro}
\int_{\Gr(m,n)} \rotv_i^{m}(\set{K}|\set{L}) \, {\nu}_m(\diff{\set{L}})
	= \rotv_{n-m+i}^{n}(\set{K}).
\end{equation}
The superscript indicates the ambient dimension
of the space in which the rotation volumes are calculated.
The Grassmannian $\Gr(m,n)$ of $m$-dimensional subspaces in $\R^n$
is equipped with its rotation-invariant probability measure $\nu_m$;
see Appendix~\ref{sec:invariant-grass}.
\end{fact}

Among other things, the projection formula allows us
to study the total rotation volume of a random projection
onto an $m$-dimensional subspace.  To do so,
we define the average
\begin{equation} \label{eqn:randproj-intro}
\mathrm{RandProj}_m(\set{K}) := \int_{\Gr(m,n)} \frac{\rotwills^m(\set{K}|\set{L})}{\rotwills^n(\set{K})} \,\nu_m(\diff{\set{L}})
	= \sum_{i = 0}^m \frac{\rotv_{n-i}^n(\set{K})}{\rotwills^n(\set{K})}.
\end{equation}
The second relation follows when we sum~\eqref{eqn:proj-formula-intro}
over indices $i = 0, 1,2,\dots, m$ and divide by the total rotation volume
$\rotwills^n(\set{K})$.

Figure~\ref{fig:randproj-randslice} illustrates the behavior
of the function $m \mapsto \mathrm{RandProj}_m(s\set{Q}^n)$
for some scaled cubes $s \set{Q}^n$.
Evidently, the function jumps from zero to one as the subspace
dimension $m$ increases, and the relative width of the jump
becomes narrower as the ambient dimension $n$ grows.
We have developed a complete mathematical
explanation for this empirical observation.

In view of~\eqref{eqn:randproj-intro}, the function $\mathrm{RandProj}_m$
can be written using the rotation volume random variable~\eqref{eqn:wvol-rv-def}:
$$
\mathrm{RandProj}_m(\set{K}) = \Prob{ \rotI_{\set{K}} \leq m }.
$$
Second, invoke Theorem~\ref{thm:rotvol-conc},
on the concentration of rotation volumes,
to see that $m \mapsto \Prob{ \smash{\rotI_{\set{K}} \leq m} }$
makes a transition from zero to one around the value
$m = \rotdelta(\set{K})$. Here is a precise statement.

\begin{bigthm}[Random Projections: Phase Transition]\label{thm:randproj-intro}
Consider a nonempty convex body $\set{K} \subset \R^n$
with central rotation volume $\rotdelta(\set{K})$.
For each subspace dimension $m \in \{0,1,2, \dots ,n\}$, form the average \begin{equation*}\mathrm{RandProj}_m(\set{K}) := \int_{\Gr(m,n)} \frac{\rotwills^m(\set{K}|\set{L})}{\rotwills^n(\set{K})} \,\nu_m(\diff{\set{L}})
	\in [0, 1].
\end{equation*}
For a proportion $\alpha \in (0,1)$, set the transition width
$$
t_{\star}(\alpha) := \big[ 2 \rotdelta(\set{K}) \log(1/\alpha) \big]^{1/2} + \tfrac{2}{3} \log(1/\alpha).
$$
Then $\mathrm{RandProj}_m(\set{K})$ undergoes a phase transition at $m = \rotdelta(\set{K})$
with width controlled by $t_{\star}(\alpha)$: \begin{equation*}
\begin{aligned}
m &\leq \rotdelta(\set{K}) - t_{\star}(\alpha)
&&\quad\text{implies}\quad &\mathrm{RandProj}_m(\set{K}) &\leq \alpha; \\
m &\geq \rotdelta(\set{K}) + t_{\star}(\alpha)
&&\quad\text{implies}\quad &\mathrm{RandProj}_m(\set{K}) &\geq 1 - \alpha.
\end{aligned}
\end{equation*}
\end{bigthm}

\noindent
We have already given the majority of the proof;
see Section~\ref{sec:randproj} for the remaining details.

Here is an interpretation.
Random projection of a convex body onto a subspace either obliterates or preserves its total rotation volume,
depending on whether the dimension $m$ of the subspace is smaller or larger than $\rotdelta(\set{K})$.
Thus, the central rotation volume $\rotdelta(\set{K})$ measures the ``dimension'' of the set $\set{K}$.

The transition between the two cases takes place over a very narrow range of subspace dimensions.
For example, when $\alpha = 1/\econst^{2}$, the transition width is around $2\rotdelta(\set{K})^{1/2} + 1$,
which is usually much smaller than the location $\rotdelta(\set{K})$ of the transition.
Keeping in mind that the central rotation volume $\rotdelta(\set{K}) \in [0, n]$, we see that this transition takes place
over at most $2\sqrt{n} + 1$ dimensions, which should be compared with the range $n$
of possible subspace dimensions.

\begin{figure}[t]
\centering
\includegraphics[width=0.48\textwidth]{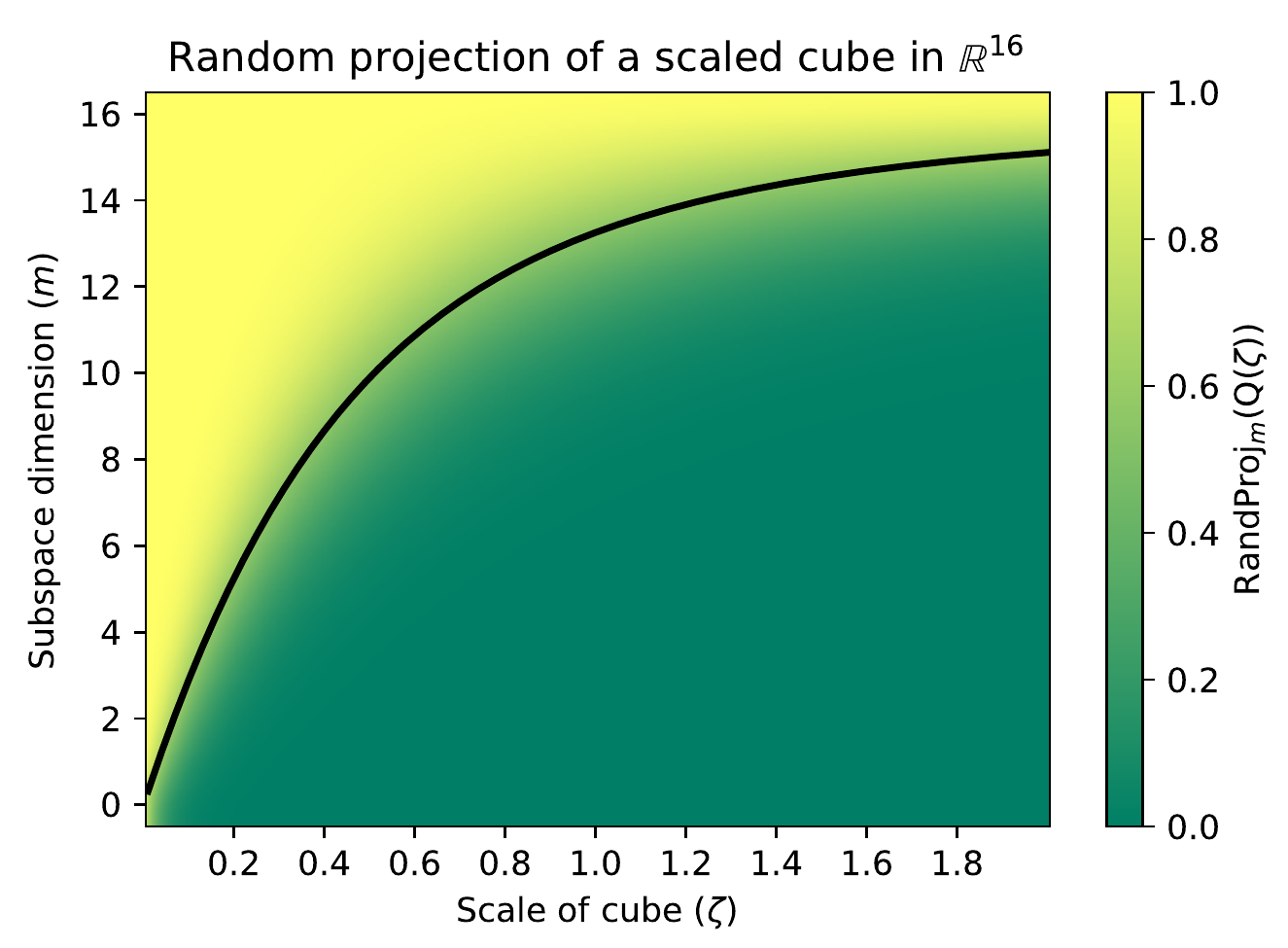} \hspace{1pc}
\includegraphics[width=0.48\textwidth]{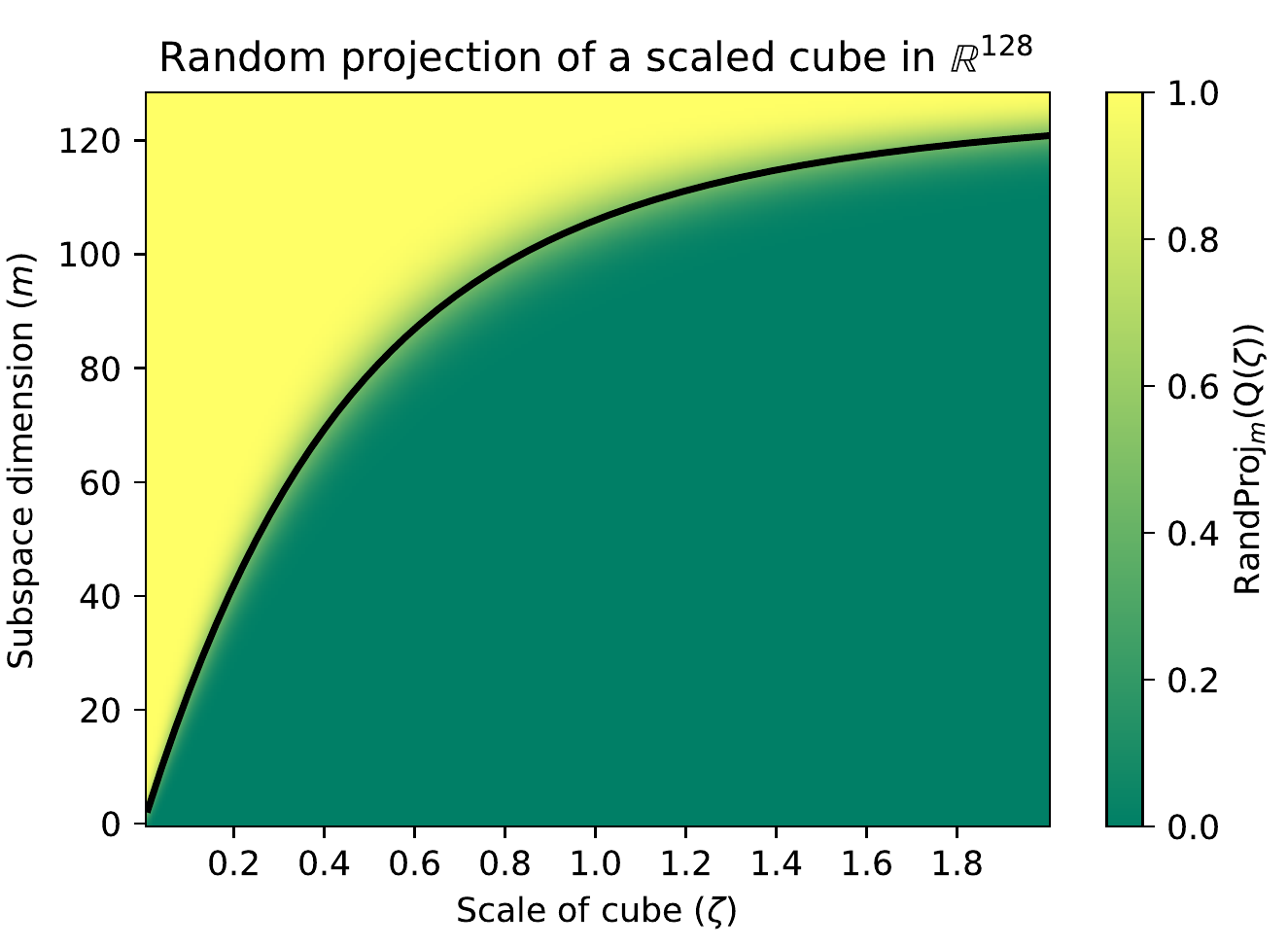}
\caption{\textbf{Phase Transition for Projections of a Cube.}
For a scaled cube $\set{Q}(\zeta) = \zeta \sqrt{2n / \pi} \, \set{Q}^n$, these plots display $\mathrm{RandProj}_m(\mathsf{Q}(\zeta))$
as a function of the scaling parameter $\zeta$ and the dimension $m$
of the projection.  The black curve traces the asymptotic expression~\eqref{eqn:central-scaled-cube}
for the central rotation volume $\rotdelta(\set{Q}(\zeta)) / n$.
The function $\mathrm{RandProj}_m$ is defined in~\eqref{eqn:randproj-intro}.
The plots are similar, except that the  transition region narrows
when the ambient dimension increases from $n = 16$ to $n = 128$.} \label{fig:randproj-phase-map}
\end{figure} 

\begin{example}[Scaled Cube: Random Projection]
Consider the family $\set{Q}(\zeta;n) = \zeta \sqrt{2n/\pi} \, \set{Q}^n$ of scaled cubes
that were introduced in Example~\ref{ex:scaled-cube-delta}.
When $n$ is large, Theorem~\ref{thm:randproj-intro}
and the asymptotic formula~\eqref{eqn:central-scaled-cube} show that
\begin{equation*}
\begin{aligned}
\frac{m}{n} &\lessapprox \frac{2}{1+\sqrt{1+\zeta^{-2}}} &&\quad\text{implies}\quad &\mathrm{RandProj}_m(\set{Q}(\zeta;n)) &\approx 0; \\ \frac{m}{n} &\gtrapprox  \frac{2}{1+\sqrt{1+\zeta^{-2}}} &&\quad\text{implies}\quad &\mathrm{RandProj}_m(\set{Q}(\zeta;n)) &\approx 1. \end{aligned}
\end{equation*}
To confirm this result,
Figure~\ref{fig:randproj-phase-map} displays the numerical value $\mathrm{RandProj}_m(\set{Q}(\zeta; n))$
as a function of $\zeta$ and $m$ for two choices of dimension $n$.
The asymptotic phase transition is superimposed.  We remark that each projection of the cube $\set{Q}^n$
is a zonotope with at most $n$ facets.
\end{example}

\begin{remark}[Related Work]
Milman~\cite{Mil90:Note-Low} established the the \emph{diameter} of a random projection
of a norm ball undergoes a change in behavior as the dimension $m$
of the subspace increases.  When $m$ is much smaller than the mean width of the body,
random projections are isomorphic to Euclidean balls whose radius is the mean width.
When $m$ is much larger than the mean width of the body, the diameter of the random projection
is comparable to $\sqrt{m/n}$ times the diameter of the body.  See~\cite[Sec.~5.7.1]{artstein2015asymptotic}
or~\cite[Sec.~7.7, Sec.~9.2.2]{Ver18:High-Dimensional-Probability} for more details.
\end{remark}

\begin{figure}[t]
\begin{center}
\includegraphics[width=0.7\textwidth]{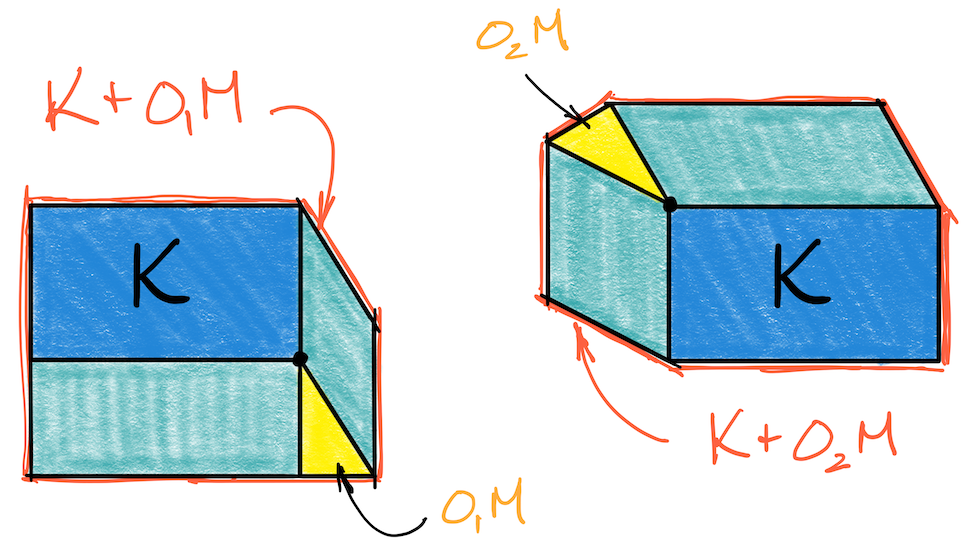}
\caption{\textbf{Rotation Means.}  This picture illustrates the
Minkowski sum of a convex body $\set{K}$ and two rotations
$\mtx{O}_1 \set{M}$ and $\mtx{O}_2 \set{M}$ of a convex body
$\set{M}$.  The rotation mean value formula, Fact~\ref{fact:rotmean},
expresses the expected area of the sum, averaged over all rotations,
in terms of the perimeters and areas of $\set{K}$ and $\set{M}$.}
\label{fig:rotation-mean}
\end{center}
\end{figure}

\subsubsection{Rotation Means}

Next, we consider a problem that involves the interaction
between two different convex bodies:

\begin{quotation} \textcolor{dkblue}{
\textbf{Suppose that we add a randomly rotated convex body
to a fixed convex body.  How big is the sum
relative to the size of the original sets?}}
\end{quotation}

\noindent
See Figure~\ref{fig:rotation-mean} for an illustration.
One version of this question admits an exact answer,
which is contained in a classic result called the
rotation mean value formula (Fact~\ref{fact:rotmean}).
We have discovered that rotation means exhibit
a sharp phase transition.

\begin{bigthm}[Rotation Means: Phase Transition] \label{thm:rotmean-intro}
Consider two nonempty convex bodies $\set{K}, \set{M} \subset \R^n$,
and define the sum $\rotDelta(\set{K}, \set{M}) := \rotdelta(\set{K}) + \rotdelta(\set{M})$ of the central rotation volumes. Form the average \begin{equation} \label{eqn:rotmean-intro}
\mathrm{RotMean}(\set{K}, \set{M}) := \int_{\SO(n)} \frac{\rotwills(\set{K} + \mtx{O} \set{M})}{\rotwills(\set{K}) \, \rotwills(\set{M})} \, \nu(\diff{\mtx{O}})
	\in [0, 1].
\end{equation}
The special orthogonal group $\SO(n)$ is equipped with its invariant probability measure $\nu$;
see Section~\ref{sec:invariant-SO}.  For a proportion $\alpha \in (0,1)$,
set the transition width
$$
t_{\star}(\alpha) := \big[ 2 \rotDelta(\set{K},\set{M}) \log(1/\alpha) \big]^{1/2} + \tfrac{2}{3} \log(1/\alpha).
$$ Then $\mathrm{RotMean}(\set{K}, \set{M})$ undergoes a phase transition at $\rotDelta(\set{K}, \set{M}) = n$
with width controlled by $t_{\star}(\alpha)$:
\begin{equation*}
\begin{aligned}
\rotDelta(\set{K},\set{M}) &\leq n - t_{\star}(\alpha)
&&\quad\text{implies}\quad &\mathrm{RotMean}(\set{K},\set{M}) &\geq 1 - \alpha; \\
\rotDelta(\set{K}, \set{M}) &\geq n + t_{\star}(\alpha)
&&\quad\text{implies}\quad &\mathrm{RotMean}(\set{K},\set{M}) &\leq \alpha.
\end{aligned}
\end{equation*}
\end{bigthm}

\noindent
The proof of Theorem~\ref{thm:rotmean-intro} appears in Section~\ref{sec:rotmean}.

This result states that the total rotation volume of the sum
$\set{K} + \mtx{O} \set{M}$, averaged over rotations $\mtx{O}$,
does not exceed the product of the total rotation volumes of the two sets.
When the total ``dimension'' $\rotDelta(\set{K}, \set{M})$ of the two sets is smaller
than the ambient dimension $n$, the average total rotation volume
of the sum is nearly as large as possible.
The situation changes when the total ``dimension'' of the two sets
is larger than the ambient dimension, in which case the average total
rotation volume of the sum is far smaller than the upper bound.
The width of the transition region between the
two regimes is not more than $2\sqrt{n}$, which
is negligible since $\rotDelta(\set{K}, \set{M}) \in [0, 2n]$.

We can make an analogy with linear algebra.  For two subspaces in
general position, the dimension of the sum is equal to the total
dimension of the subspaces when the total dimension is smaller
than the ambient dimension.  Otherwise, the dimension of the sum
equals the ambient dimension.

\begin{figure}[t]
\begin{center}
\includegraphics[width=0.48\textwidth]{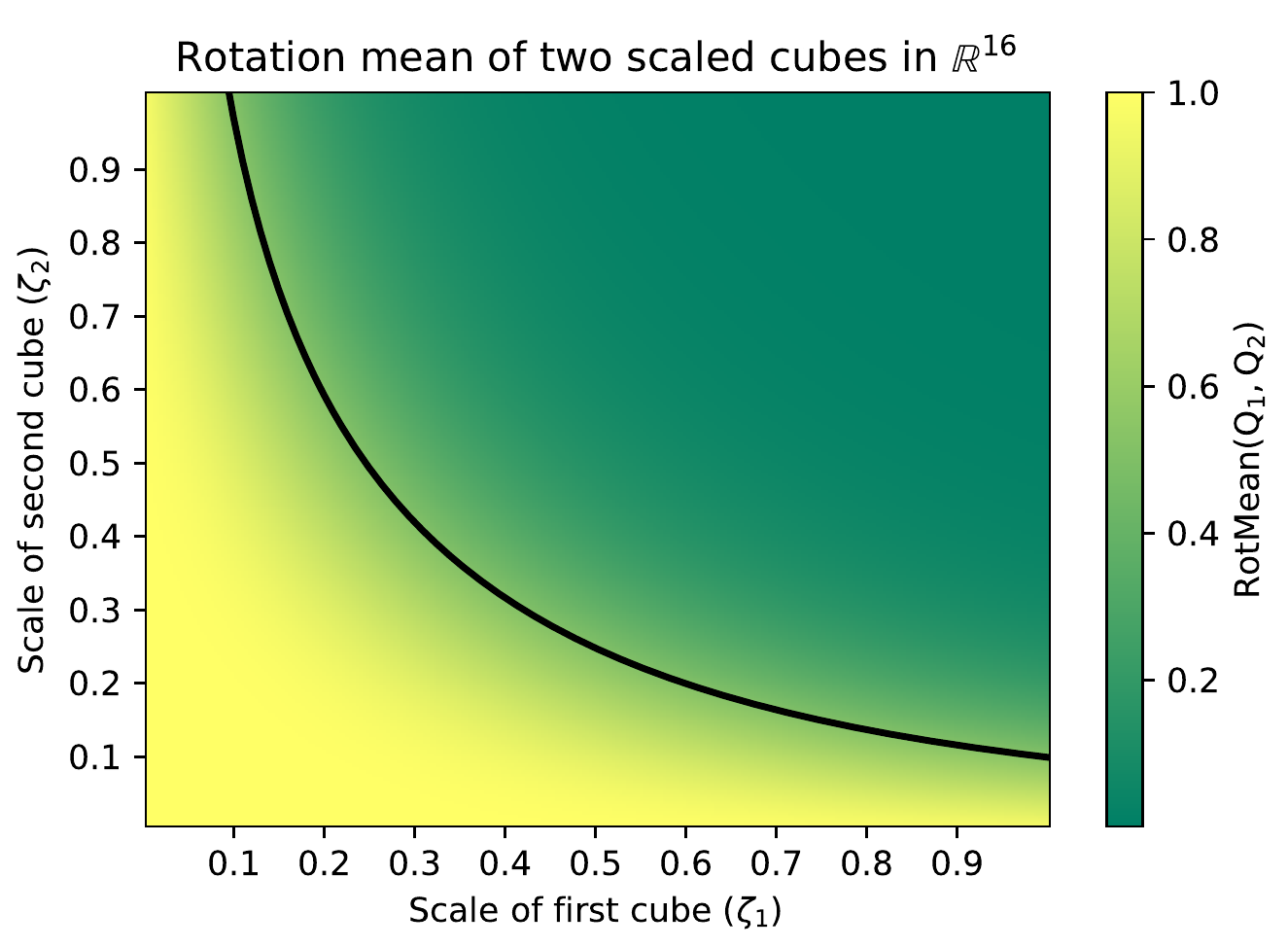} \includegraphics[width=0.48\textwidth]{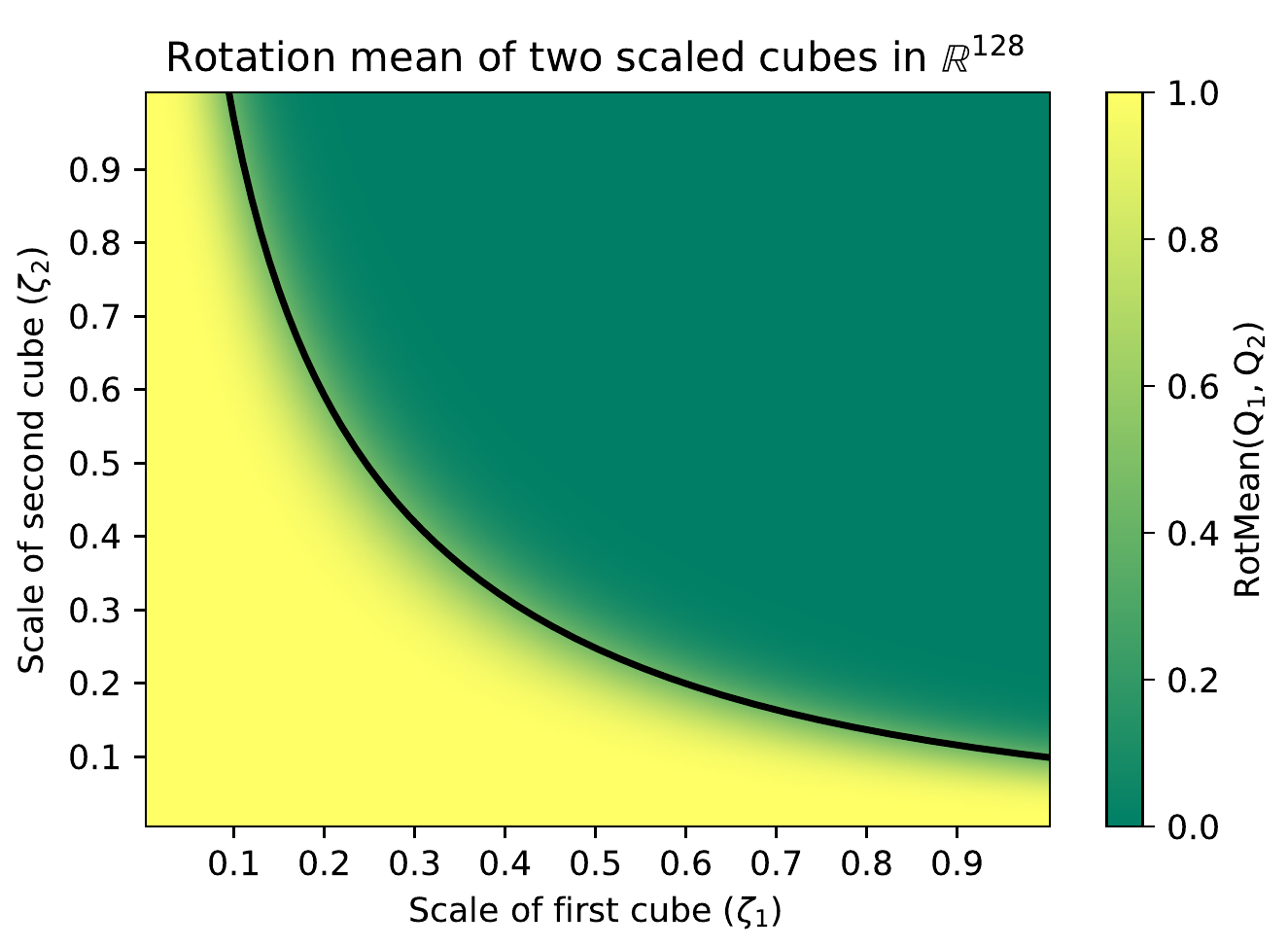}
\caption{\textbf{Phase Transition for Rotation Mean of Two Cubes.}
For scaled cubes $\set{Q}(\zeta) = \zeta \sqrt{2n / \pi} \, \set{Q}^n$ with $n \in \{16,128\}$,
this plot displays the value of $\mathrm{RotMean}(\set{Q}(\zeta_1), \set{Q}(\zeta_2))$
as a function of the scaling parameters $\zeta_1$ and $\zeta_2$.
The black curve traces the asymptotic phase transition~\eqref{eqn:rotmean-asymp}.
The function $\mathrm{RotMean}$ is defined in~\eqref{eqn:rotmean-intro}.} \label{fig:rotvol-phase-map}
\end{center}
\end{figure} 

\begin{example}[Scaled Cubes: Rotation Mean]
Consider the family $\set{Q}(\zeta; n) = \zeta \sqrt{2n/\pi} \, \set{Q}^n$ of scaled cubes.
When $n$ is large, Theorem~\ref{thm:rotmean-intro} and the asymptotic formula~\eqref{eqn:central-scaled-cube}
demonstrate that the function $(\zeta_1, \zeta_2) \mapsto \mathrm{RotMean}(\set{Q}(\zeta_1; n), \set{Q}(\zeta_2; n))$
undergoes a phase transition along the curve
\begin{equation} \label{eqn:rotmean-asymp}
\frac{2}{1 + \sqrt{1 + {\zeta_1}^{-2}}} + \frac{2}{1 + \sqrt{1 + {\zeta_2}^{-2}}} = 1.
\end{equation}
See Figure~\ref{fig:rotvol-phase-map} for a numerical illustration.
\end{example}

\begin{remark}[Related Work]
The global form of Dvoretsky's theorem, due to Bourgain et al.~\cite{BLM88:Minkowski-Sums},
states that a rotation mean $\mtx{O}_1 \set{K} + \dots + \mtx{O}_r \set{K}$
is isomorphic to a Euclidean ball when $r$ is sufficiently large.
Here, the convex body $\set{K}$ is a norm ball, and
the random rotations $\mtx{O}_i$ are independent and uniform on $\SO(n)$.
A bound for the number $r$ of summands can be obtained from
the geometry of the set $\set{K}$.
See~\cite[Sec.~5.6]{artstein2015asymptotic}.
\end{remark}

\subsection{Rigid Motions}

Next, we discuss two problems where we integrate over all
Euclidean proper rigid motions (i.e., translations and rotations).
The first involves random affine slices of a convex body.
The second involves the intersection of randomly positioned convex bodies.
These models are trickier to interpret than the models
involving rotations because the integrals cannot be
regarded as averages.

\begin{figure}[t!]
\begin{center}
\includegraphics[width=0.6\textwidth]{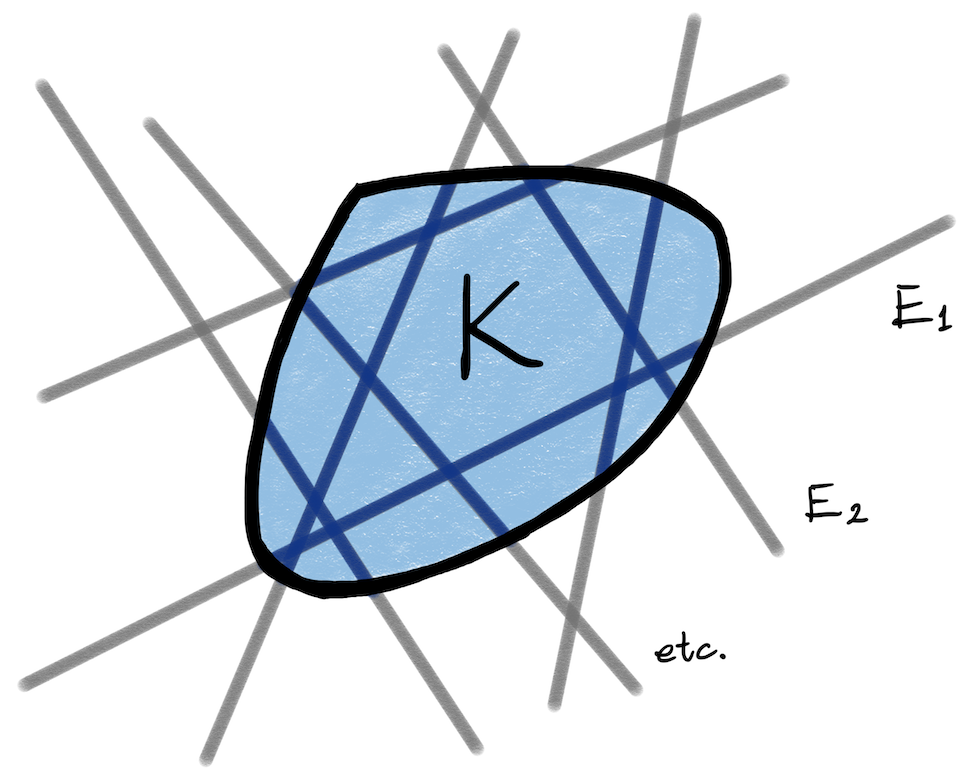}
\caption{\textbf{Intersection with Affine Spaces.}
This picture exhibits slices of a convex body $\set{K}$
by one-dimensional affine spaces $\set{E}_1$, $\set{E}_2$,
etc.  The slicing formula, Fact~\ref{fact:slicing}, asserts
that an appropriate integral of the length of all such affine slices is proportional to the area of the set $\set{K}$.
Furthermore, the measure of the set of affine lines that hit $\set{K}$
is proportional to the perimeter of $\set{K}$.} \label{fig:crofton}
\end{center}
\end{figure}

\subsubsection{Intersection with Random Affine Spaces}

The first problem is dual to the problem of projecting
onto a random subspace:

\begin{quotation} \textcolor{dkblue}{
\textbf{Suppose we slice a convex body with a random
affine space of a given dimension.  How does the total
size of the slices compare with the size of the original body?}}
\end{quotation}

\noindent
Figure~\ref{fig:crofton} contains a schematic.
This problem has an exact solution in terms of intrinsic volumes, a result
known as Crofton's formula (Fact~\ref{fact:slicing}).
We have discovered a sharp phase transition here too.

\begin{bigthm}[Random Slices: Phase Transition] \label{thm:crofton-intro}
Consider a nonempty convex body $\set{K} \subset \R^n$
with central rigid motion volume $\rmdelta(\set{K})$
and variance proxy $\rmsigma^2(\set{K})$, as defined in~Theorem~\ref{thm:rmvol-intro}.
For each affine space dimension $m \in \{0,1,2, \dots, n\}$, form the integral
\begin{equation} \label{eqn:randslice-intro}
\mathrm{RandSlice}_m(\set{K}) := \int_{\Af(m,n)} \frac{\rmwills(\set{K} \cap \set{E})}{\rmwills(\set{K})} \, \mu_m(\diff{\set{E}}) \in [0,1].
\end{equation}
The set $\Af(m,n)$ of $m$-dimensional affine spaces in $\R^n$ is equipped
with a canonical rigid-motion-invariant measure $\mu_m$; see Section~\ref{sec:invariant-Af}.
For a proportion $\alpha \in (0,1)$, set
the transition width
$$
t_{\star}(\alpha) := \big[ 4 \rmsigma^2(\set{K}) \log(1/\alpha) \big]^{1/2} + \tfrac{4}{3} \log(1/\alpha).
$$ Then $m \mapsto \mathrm{RandSlice}_m(\set{K})$ undergoes a phase transition at $m = \rmdelta(\set{K})$
with width controlled by $t_{\star}(\alpha)$:
\begin{equation*}
\begin{aligned}
m &\leq \rmdelta(\set{K}) - t_{\star}(\alpha)
&&\quad\text{implies}\quad &\mathrm{RandSlice}_m(\set{K}) &\leq \alpha; \\
m &\geq \rmdelta(\set{K}) + t_{\star}(\alpha)
&&\quad\text{implies}\quad &\mathrm{RandSlice}_m(\set{K}) &\geq 1 - \alpha.
\end{aligned}
\end{equation*}
\end{bigthm}

\noindent
See Section~\ref{sec:crofton} for the proof of Theorem~\ref{thm:crofton-intro}.
Figure~\ref{fig:randproj-randslice} contains a numerical demonstration.

A distinctive feature of the affine slicing model is that the total
rigid motion volume $\rmwills(\set{K} \cap \set{E})$ of an affine
slice \emph{never} exceeds $\rmwills(\set{K})$; this point follows
from monotonicity of volumes.
The theorem states that the total rigid motion volume of a random slice is either a negligible proportion
of the maximum or an overwhelming proportion, depending on whether the dimension
$m$ of the affine space is smaller or larger than the central rigid motion volume $\rmdelta(\set{K})$.
The transition occurs over a narrow change in the subspace dimension,
on the order of $[ \rmdelta(\set{K}) \wedge ( (n+1) - \rmdelta(\set{K}) ) ]^{1/2}$.

Therefore, we can interpret the central rigid motion volume $\rmdelta(\set{K})$
as a measure of the ``codimension'' of the body $\set{K}$.  (Recall that the central rigid
motion volume \emph{decreases} as the set $\set{K}$ gets larger!)

\begin{figure}[t]
\centering
\includegraphics[width=0.48\textwidth]{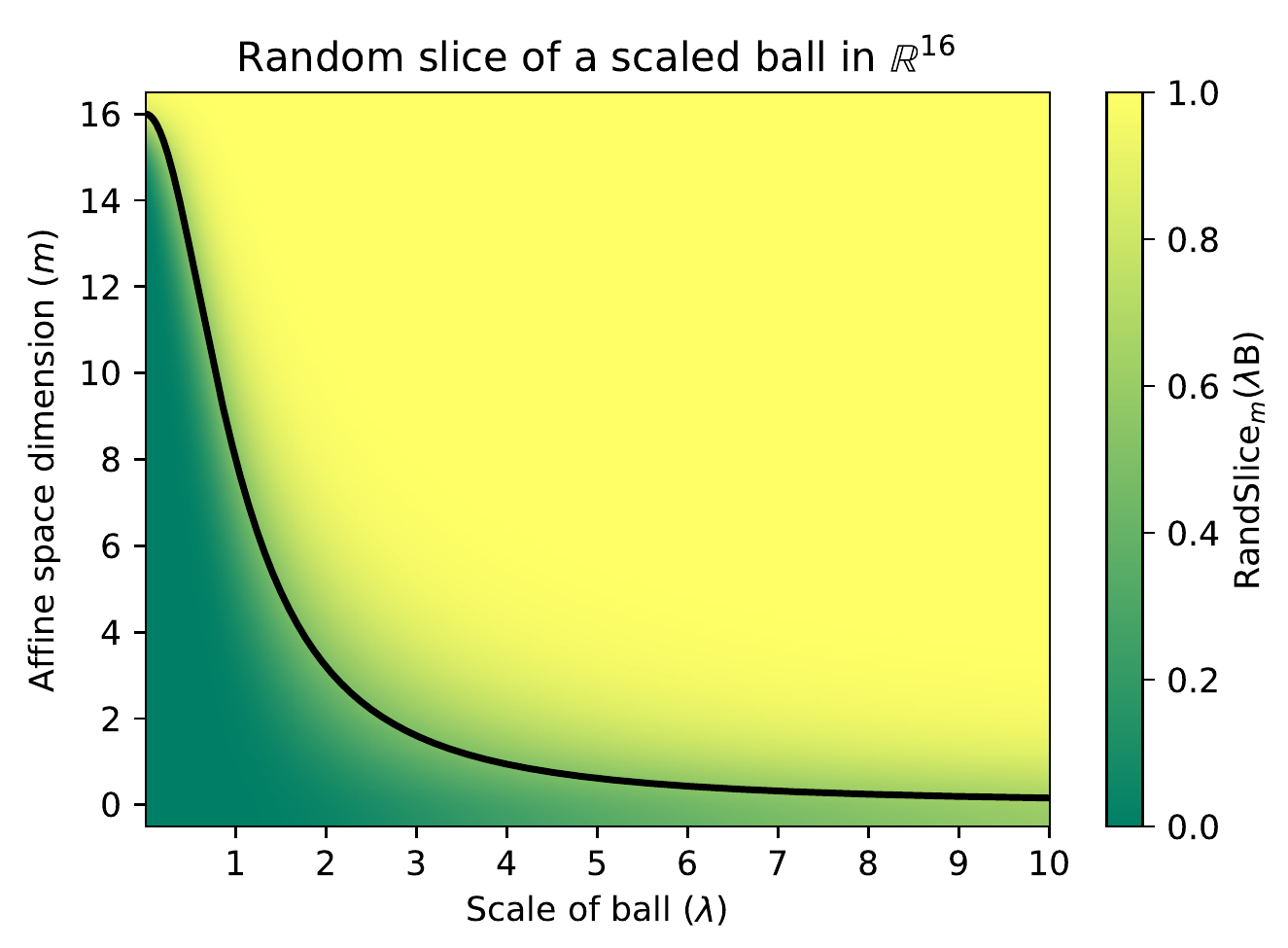} \hspace{1pc}
\includegraphics[width=0.48\textwidth]{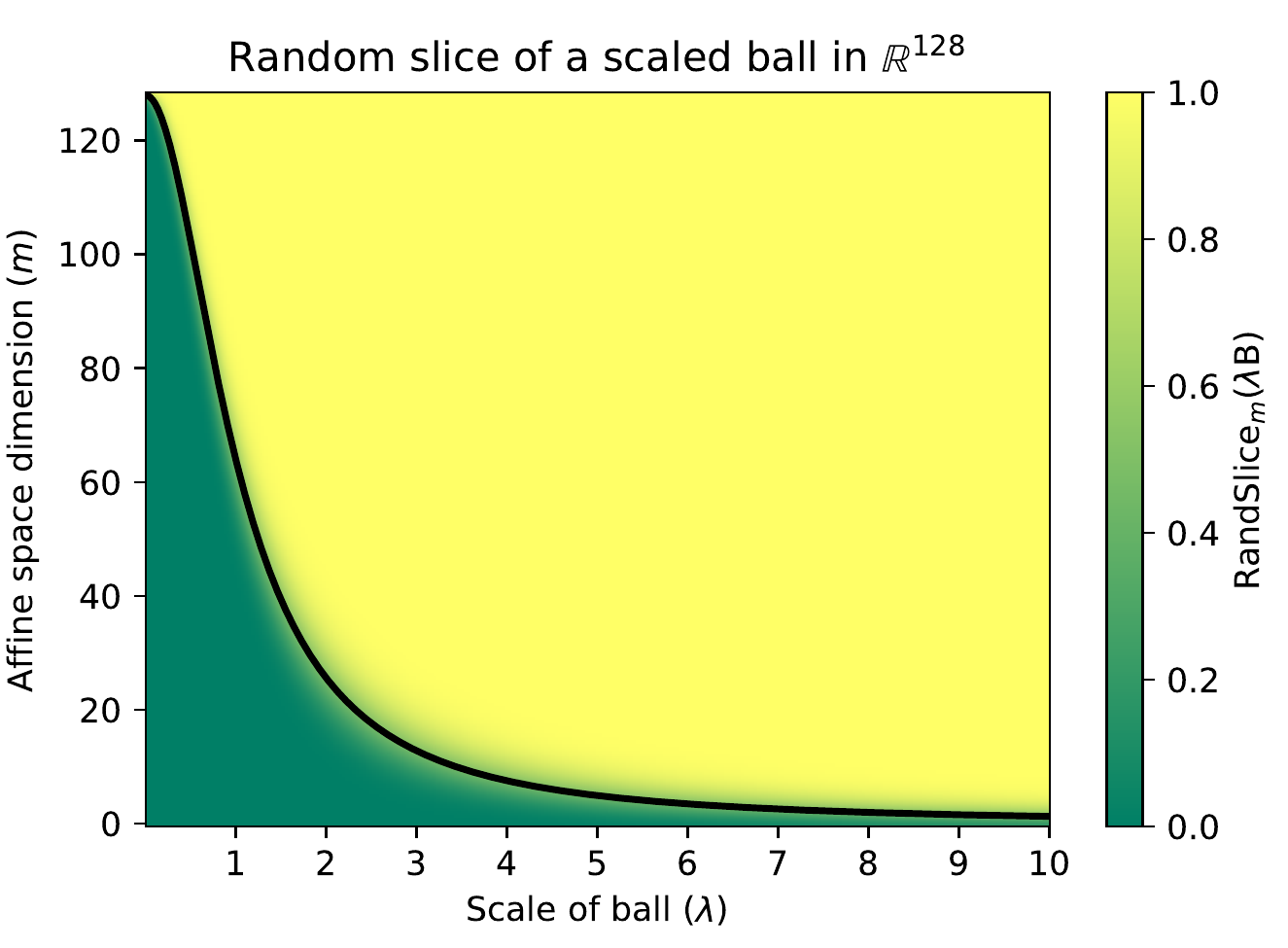}
\caption{\textbf{Phase Transition for Slices of a Ball.}
For a scaled ball $\lambda\ball{n}$, these plots display $\mathrm{RandSlice}_m(\lambda \ball{n})$
as a function of the scaling parameter $\lambda$ and the dimension $m$
of the slice.  The black curve is the asymptotic expression~\eqref{eqn:central-scaled-ball}
for the central rigid motion volume $\rmdelta(\lambda \ball{n}) / n$.
The function $\mathrm{RandSlice}_m$ is defined in~\eqref{eqn:randslice-intro}.}
\label{fig:crofton-phase-map}
\end{figure} 

\begin{example}[Scaled Ball: Random Slice]
Consider the family $\lambda\ball{n}$ of scaled Euclidean balls from Example~\ref{ex:scaled-ball-delta}.
When $n$ is large, Theorem~\ref{thm:crofton-intro}
and the asymptotic formula~\eqref{eqn:central-scaled-ball} show that
\begin{equation*}
\begin{aligned}
\frac{m}{n} &\lessapprox \frac{1}{1 + \lambda^2} &&\quad\text{implies}\quad &\mathrm{RandSlice}_m(\lambda \ball{n}) &\approx 0; \\ \frac{m}{n} &\gtrapprox  \frac{1}{1 + \lambda^2} &&\quad\text{implies}\quad &\mathrm{RandSlice}_m(\lambda \ball{n}) &\approx 1. \end{aligned}
\end{equation*}
Figure~\ref{fig:crofton-phase-map} displays the numerical value $\mathrm{RandSlice}_m(\lambda\ball{n})$
as a function of $\lambda$ and $m$ for two choices of dimension $n$.
The asymptotic phase transition is superimposed. 
\end{example}

\begin{remark}[Related Work]
Dvoretsky's theorem states that every norm ball in $\R^n$
has an affine slice, with dimension on the order of $\log n$,
that is isomorphic to a Euclidean ball.
Milman's proof~\cite{Mil71:New-Proof} shows that a random slice
has the desired property with high probability.
Research in asymptotic convex geometry has elucidated how the structure of the body
affects the dimension of Euclidean slices.
In particular, there is a phase transition when the codimension $m$ of the affine space
passes the squared Gaussian width of the set.
See~\cite[Chaps.~5, 9]{artstein2015asymptotic}
or~\cite[Sec.~9.4]{Ver18:High-Dimensional-Probability} for details.
\end{remark}

\subsubsection{The Kinematic Formula}

The last problem is dual to the question about rotation means:

\begin{quotation} \textcolor{dkblue}{
\textbf{Suppose that  we intersect a fixed convex body with a randomly positioned convex body.
How does the total size of the intersections compare with the size of the original sets?}}
\end{quotation}

\noindent
See Figure~\ref{fig:kinematic} for an illustration.
The celebrated kinematic formula (Fact~\ref{fact:kinematic}) gives the exact solution to this problem. The kinematic formula also exhibits a sharp phase transition.

\begin{figure}[t!]
\begin{center}
\includegraphics[width=0.9\textwidth]{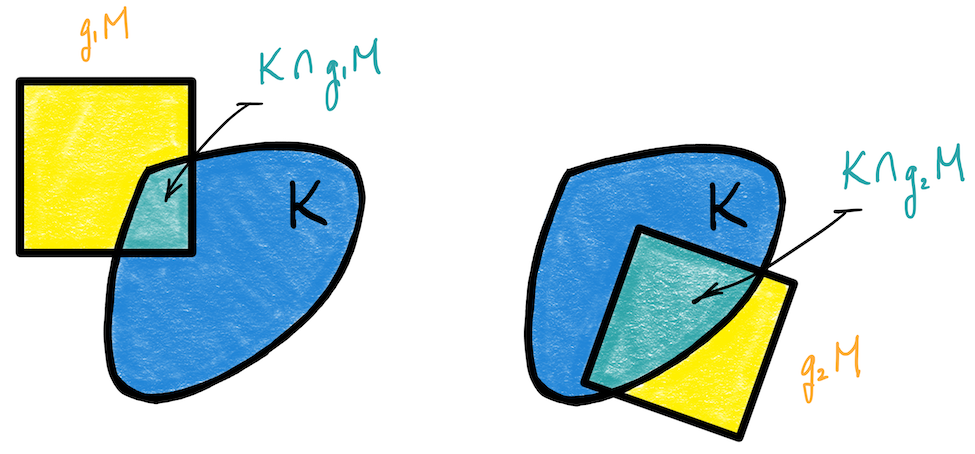}
\caption{\textbf{Kinematic Formula.}  This diagram shows
the intersection of a convex body $\set{K}$ with each of two rigid
transformations $g_1 \set{M}$ and $g_2 \set{M}$ of a convex body $\set{M}$.
The kinematic formula, Fact~\ref{fact:kinematic},
states that an appropriate integral of the area of
all such intersections is proportional to the product of the
area of $\set{K}$ and the area of $\set{M}$.
Moreover, the measure of the set of rigid motions that bring
$\set{M}$ into contact with $\set{K}$ can be expressed
in terms of the area and perimeter of the two sets.} \label{fig:kinematic}
\end{center}
\end{figure}

\begin{bigthm}[Kinematic Formula: Phase Transition] \label{thm:kinematic-intro}
Consider two nonempty convex bodies $\set{K}, \set{M} \subset \R^n$,
and define the sum $\rmDelta(\set{K},\set{M}) := \rmdelta(\set{K}) + \rmdelta(\set{M})$ of the central rigid motion volumes. Form the integral \begin{equation}\label{eqn:kinematic-intro}
\mathrm{Kinematic}(\set{K}, \set{M}) := \int_{\RM(n)} \frac{\rmwills(\set{K} \cap g \set{M})}{\rmwills(\set{K}) \, \rmwills(\set{M})} \, \mu(\diff g)
	\in [0, 1].
\end{equation}
The group $\RM(n)$ of proper rigid motions on $\R^n$ is equipped with a canonical
motion-invariant measure $\mu$; see Section~\ref{sec:invariant-RM}.
For a proportion $\alpha \in (0,1)$, set
the transition width
$$
t_{\star}(\alpha) := \big[ 4 \rmDelta(\set{K},\set{M}) \log(1/\alpha) \big]^{1/2} + \tfrac{4}{3} \log(1/\alpha).
$$ Then $\mathrm{Kinematic}(\set{K}, \set{M})$ undergoes a phase transition at $\rmDelta(\set{K},\set{M}) = n$
with width controlled by $t_{\star}(\alpha)$:
\begin{equation*}
\begin{aligned}
\rmDelta(\set{K},\set{M}) &\leq n - t_{\star}(\alpha)
&&\quad\text{implies}\quad & \mathrm{Kinematic}(\set{K},\set{M}) &\geq 1 - \alpha; \\
\rmDelta(\set{K},\set{M}) &\geq n + t_{\star}(\alpha)
&&\quad\text{implies}\quad & \mathrm{Kinematic}(\set{K},\set{M}) &\leq \alpha.
\end{aligned}
\end{equation*}
\end{bigthm}

\noindent
Section~\ref{sec:kinematic} contains the proof of Theorem~\ref{thm:kinematic-intro}.

Theorem~\ref{thm:kinematic-intro} asserts that the total rigid motion volume
of an intersection $\set{K} \cap g \set{M}$, integrated over proper motions $g$,
does not exceed the product of the total rigid motion volumes of the two bodies.
When the total ``codimension'' $\rmDelta(\set{K}, \set{M})$ of the two sets
is smaller than the ambient dimension $n$, the integral of the total rotation
volume is nearly as large as possible.  In the complementary case, it is
negligible.  The transition takes place as the total ``codimension''
changes by about $\sqrt{n}$, or less.

\begin{example}[Scaled Balls: Kinematics]
Consider the family $\lambda \ball{n}$ of scaled Euclidean balls.
For large $n$, Theorem~\ref{thm:kinematic-intro} and the asymptotic formula~\eqref{eqn:central-scaled-ball}
demonstrate that the function $(\lambda_1,\lambda_2) \mapsto \mathrm{Kinematic}(\lambda_1\ball{n}, \lambda_2\ball{n})$
undergoes a phase transition at the curve
\begin{equation} \label{eqn:kinematic-asymp}
\frac{1}{1 + \lambda_1^2}+\frac{1}{1 + \lambda_2^2} = 1
\quad\text{or equivalently}\quad \lambda_1\lambda_2=1.
\end{equation}
See Figure~\ref{fig:kinematic-phase-map} for a numerical illustration.
\end{example}

\subsection{Proof Strategy}

All four of the phase transition results follow from considerations
that are similar to the argument behind Theorem~\ref{thm:randproj-intro},
the result on random projections.  First, we use the appropriate weighted
intrinsic volumes to rewrite a classical integral geometry formula as an ordinary
convolution.  This step allows us to express the value of the geometric integral as a probability
involving weighted intrinsic volume random variables.  Last, we invoke
concentration of the weighted intrinsic volumes to see that the probability
undergoes a phase transition.  The proofs appear in Sections~\ref{sec:moving-flats}
and~\ref{sec:moving-bodies}.

\begin{figure}[t]
\begin{center}
\includegraphics[width=0.48\textwidth]{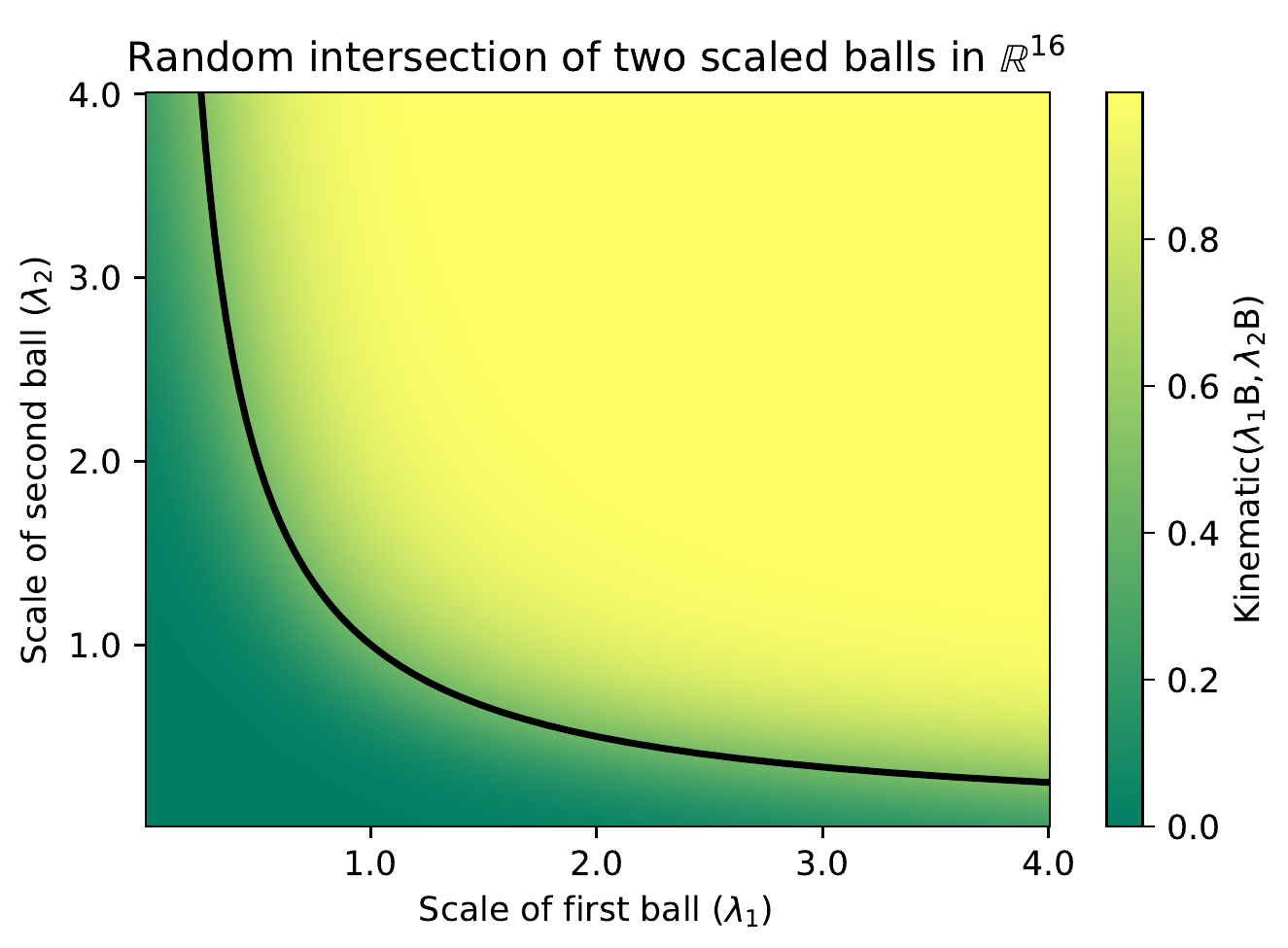} \includegraphics[width=0.48\textwidth]{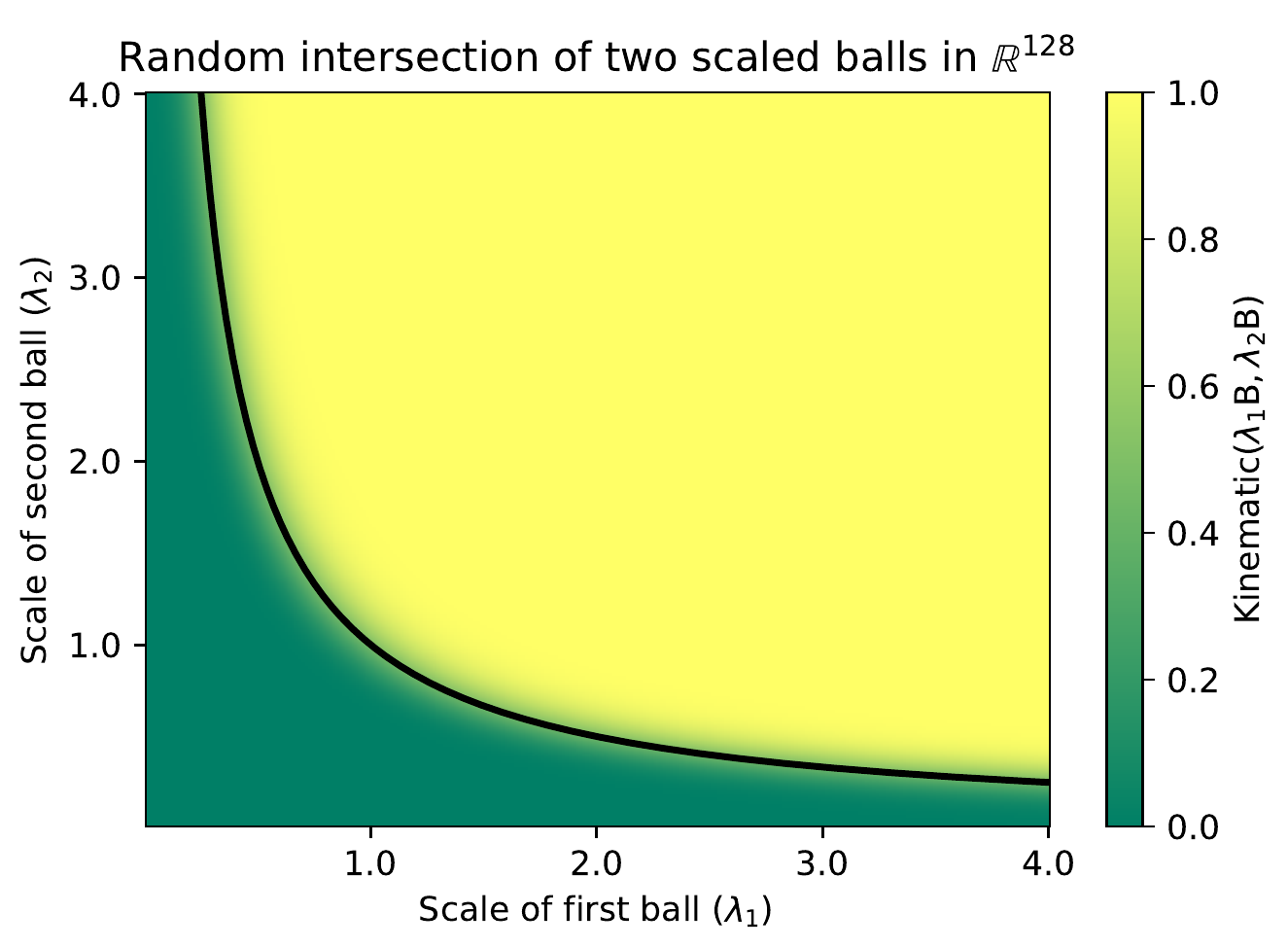}
\caption{\textbf{Phase Transition for Intersection of Two Moving Balls.}
For two scaled balls, this plot displays $\mathrm{Kinematic}(\lambda_1 \ball{n}, \lambda_2\ball{n})$
as a function of the scaling parameters $\lambda_1$ and $\lambda_2$.
The black curve is the asymptotic phase transition~\eqref{eqn:kinematic-asymp}.
The function $\mathrm{Kinematic}$ is defined in~\eqref{eqn:kinematic-intro}.} \label{fig:kinematic-phase-map}
\end{center}
\end{figure}

\section{Concentration via Generating Functions} \label{sec:genfun}

We use the entropy method to study the fluctuations of the weighted intrinsic volume random
variables about their mean.  This approach depends on relationships between
the cumulant generating function and its derivatives.
The modern role of these ideas in probability can be traced to the work of
Ledoux~\cite{Led96:Talagrands-Deviation}, but the approach has deeper roots
in statistical mechanics and information theory.  Our presentation is
based on Maurer's article~\cite{Mau12:Thermodynamics-Concentration}.

The development in this section does not depend on properties of
the intrinsic volumes, so it is more transparent to state the
basic results for an arbitrary bounded random variable.  We include detailed
cross-references, so many readers will prefer to continue to the
novel material in the next section.

\subsection{Moment Generating Functions}\label{sub:mgf-cgf}
We begin with the central objects of study: the moment generating function of a random variable
and the logarithm of the moment generating function.

\begin{definition}[Moment Generating Functions]
Let $Y$ be a bounded, real-valued random variable.
The \emph{moment generating function} (mgf) of $Y$
is defined as
\begin{equation} \label{eqn:mgf}
m_{Y}(\theta) := \Expect \econst^{\theta Y}
\quad\text{for $\theta \in \R$.}
\end{equation}
The \emph{cumulant generating function} (cgf) of $Y$
is the logarithm of the mgf:
\begin{equation} \label{eqn:cgf}
\xi_{Y}(\theta) := \log m_Y(\theta) = \log \Expect \econst^{\theta Y}
\quad\text{for $\theta \in \R$.}
\end{equation}
\end{definition}

We can link the cgf $\xi_Y$ to bounds on the tails of the random variable $Y$
using a classic result, often called the Cram\'er--Chernoff method.

\begin{fact}[Tails and Cgfs] \label{fact:cgf-tail}
Let $Y$ be a bounded random variable with cgf $\xi_Y$.  For all $t \geq 0$,
\begin{align*}
\log \Prob{ Y \geq +t }
	&\leq \inf_{\theta > 0}\ \left[ \xi_Y(\theta) - \theta t \right]; \\
\log \Prob{ Y \leq -t }
	&\leq \inf_{\theta < 0}\ \left[ \xi_Y(\theta) + \theta t \right].
\end{align*}\end{fact}

\noindent
The proof is easy. Just apply Markov's inequality to the random variable $\exp(\theta Y)$.
For example, see~\cite[Sec.~2.2]{boucheron2013concentration}. 

Let us record a simple but important translation identity for the cgf:
\begin{equation}\label{eqn:cgf-shift}
 \xi_{Y+c}(\theta) = \xi_Y(\theta)+c\theta
 \quad\text{for each $c \in \R$.}
\end{equation}
This result allows us to focus on zero-mean random variables.

Cgfs are particularly useful for studying a sum of independent random variables.
Indeed, if $Y$ and $Z$ are independent, then
\begin{equation} \label{eqn:cgf-indep}
\xi_{Y+Z}(\theta) = \xi_Y(\theta) + \xi_Z(\theta)
\quad\text{for all $\theta \in \R$.}
\end{equation}
This identity is easy consequence of the definition~\eqref{eqn:cgf}
and independence.

\subsection{The Gibbs Measure} 

To extract concentration inequalities from Fact~\ref{fact:cgf-tail},
it suffices to obtain estimates for the cgf $\xi_Y$.
A powerful method for controlling the cgf $\xi_Y$ is to bound its derivatives.  The derivatives of $\xi_Y$ can be expressed compactly in terms of
the moments of a probability distribution related to $Y$.

Let $\varrho_Y$ be the probability measure on the real line that
gives the distribution of the bounded real random variable $Y$.
For a parameter $\theta \in \R$,
we can define another probability distribution via
$$
\diff{\varrho_Y^{\theta}}(y)
	:= \frac{\econst^{\theta y}\idiff{\varrho_Y(y)}}{m_Y(\theta)}
	\quad\text{for $y \in \R$.}
$$
In the thermodynamics literature, the distribution $\varrho_Y^{\theta}$ is called
the \emph{Gibbs measure} at inverse temperature $\theta$, or the \term{canonical ensemble}.
Applied probabilists often call this object an \emph{exponential tilting} of the distribution.
Note that we allow the inverse temperature $\theta$ to take nonpositive values.

Since the random variable $Y$ is bounded, we can compute derivatives of
the generating functions $m_Y$ and $\xi_Y$ by passing the derivative
through the expectation.
The first derivative of the cgf satisfies
\begin{equation} \label{eqn:thermal-expect}
\xi_Y'(\theta) 
	= \frac{m_Y'(\theta)}{m_Y(\theta)}
	= \frac{\Expect[ Y \econst^{\theta Y} ]}{\Expect \econst^{\theta Y} }.
\end{equation}
The quantity $\xi_Y'(\theta)$ coincides with the mean
of the Gibbs measure $\varrho_Y^{\theta}$.  It is also called the
\term{thermal mean} of $Y$ at inverse temperature $\theta$.
Turning to the second derivative, we find that
\begin{equation} \label{eqn:thermal-var}
\xi_Y''(\theta)
	= \frac{m_Y''(\theta)}{m_Y(\theta)} - \left( \frac{m_Y'(\theta)}{m_Y(\theta)} \right)^2 \\
	= \frac{\Expect[ Y^2 \econst^{\theta Y} ]}{\Expect \econst^{\theta Y}}
	- \left( \frac{\Expect[ Y \econst^{\theta Y} ]}{\Expect \econst^{\theta Y} } \right)^2.
\end{equation}
The quantity $\xi_Y''(\theta)$ coincides with the variance
of the Gibbs measure $\varrho_Y^{\theta}$.  It is also known
as the \term{thermal variance} of $Y$ at inverse temperature $\theta$.

The derivatives of the cgf $\xi_Y$ at zero have special significance.  Indeed,
\begin{equation} \label{eqn:cgf-zero}
\xi_Y(0) = 0; \qquad
\xi_Y'(0) = \Expect Y; \qquad
\xi_Y''(0) = \Var[Y].
\end{equation}
These relations follow instantly by specializing the
formulas~\eqref{eqn:cgf},~\eqref{eqn:thermal-expect}, and~\eqref{eqn:thermal-var}.
The interpretation of $\xi_Y'$ as a mean has further consequences.
In particular,
\begin{equation} \label{eqn:thermal-mean-extreme}
\inf Y \leq \xi_Y'(\theta) \leq \sup Y
\quad\text{for all $\theta \in \R$.}
\end{equation}
The interpretation of $\xi_Y''$ as a variance also has ramifications.
Since variances are nonnegative,
$\xi_Y''(\theta) \geq 0$, which implies that $\xi_Y$ is a convex function. Moreover, the thermal variance is invariant under shifts of the random variable:
\begin{equation} \label{eqn:thermal-var-shift}
\xi_{Y + c}''(\theta) = \xi_{Y}''(\theta)
\quad\text{for all $c \in \R$ and all $\theta \in \R$.}
\end{equation}
This point follows from~\eqref{eqn:cgf-shift}. 

\subsection{Differential Inequalities}

As mentioned, a natural mechanism for controlling the cgf $\xi_Y$
is to bound the thermal mean $\xi_Y'$ over an interval,
which limits the growth of the cgf.
The thermal mean, in turn, can be calculated from the thermal variance via integration:
\begin{equation}\label{eqn:thermal-mean}
\xi_Y'(\theta)-\Expect Y \overset{\eqref{eqn:cgf-zero}}{=} \xi_Y'(\theta)-\xi_Y'(0) = \int_0^\theta \xi_Y''(s) \idiff{s}
	= - \int_{\theta}^0 \xi_Y''(s) \idiff{s}
	\quad\text{for $\theta \in \R$.}
\end{equation}
In many situations, we can bound the thermal variance $\xi_Y''$ in terms
of the thermal mean $\xi_Y'$. For instance, suppose that 
\begin{equation*}
  \xi_Y''(\theta) \leq a\cdot \xi_Y'(\theta)
  \quad\text{for some $a > 0$.}
\end{equation*}
Then \eqref{eqn:thermal-mean}
leads to an inequality relating the thermal mean $\xi_Y'$ and the cgf $\xi_Y$.

To solve this kind of differential inequality, we rely on a 1901 theorem
of Petrovitch~\cite{Pet01:Sur-Maniere}.  Compare this statement with
the more familiar result that Gr{\"o}nwall derived in 1919.

\begin{fact}[Petrovitch] \label{fact:petrovitch}
Let $g : \R \to \R$ be the solution to the ordinary differential equation
$$
\begin{cases}
g'(t) = u( t; g(t) ) & \text{for $t \in \R$}; \\
g(0) = 0,
\end{cases}
$$
where $u : \R^2 \to \R$ is any function.
For a function $f : \R \to \R$ with the boundary condition $f(0) = 0$,
the differential inequality
\begin{align*}
f'(t) \leq u( t; f(t) )
	\quad\text{on $t > 0$}
	\quad\text{implies}\quad
	f(t) &\leq g(t)
	\quad\text{on $t \geq 0$}; \\
f'(t) \geq u( t; f(t) )
	\quad\text{on $t < 0$}
	\quad\text{implies}\quad
	f(t) &\leq g(t)
	\quad\text{on $t \leq 0$}.
\end{align*}\end{fact}

In other words, we can solve a differential inequality by means of the
ordinary differential equation that arises when we replace the inequality
with an equality.  This step is analogous with the Herbst argument
in the entropy method~\cite[Chap.~6]{boucheron2013concentration}.

\subsection{Poisson Cgf Bound}

We illustrate this strategy by showing how a simple differential
inequality leads to a Poisson-type cgf bound.

\begin{lemma}[Poisson Cgf Bound] \label{lem:psi'topsi}
Suppose that $Y$ is a zero-mean random variable.  When $v \geq 0$,
\begin{align}
&\xi_Y'(\theta) \leq a \xi_Y(\theta) + v \theta \quad\text{on $\theta > 0$}
\quad\text{implies}\quad
	\xi_Y(\theta) \leq v \cdot \frac{\econst^{+a\theta} - a\theta - 1}{a^2} \quad\text{on $\theta \geq 0$};  \label{eqn:mgfY+} \\
&\xi_Y'(\theta) \geq a \xi_Y(\theta) + v \theta \quad \text{on $\theta < 0$}
\quad\text{implies}\quad
	\xi_Y(\theta) \leq v \cdot \frac{\econst^{-a\theta} + a\theta - 1}{a^2} \quad\text{on $\theta \leq 0$}. \label{eqn:mgfY-}
\end{align}
\end{lemma}

\begin{proof}
Consider the first differential inequality for the cgf $\xi_Y$.
Since $\xi_Y(0) = \Expect Y  = 0$, the associated ordinary differential equation is
$$
g(0) = 0
\quad\text{and}\quad
g'(\theta) = a g(\theta) + v \theta 
\quad\text{on $\theta > 0$.}
$$
By the method of integrating factors, we quickly arrive at the solution:
$$
g(\theta) = v \cdot \frac{\econst^{a \theta} - a \theta - 1}{a^2}
\quad\text{on $\theta \geq 0$.}
$$
Fact~\ref{fact:petrovitch} now implies the bound~\eqref{eqn:mgfY+}.
The bound~\eqref{eqn:mgfY-} follows from the parallel argument.
\end{proof}

\subsection{Poisson Concentration}

The cgf bounds that appear in Lemma~\ref{lem:psi'topsi} have
the same form as the cgf bounds that lead to the Bennett
inequality for a sum of independent, bounded real random
variables~\cite[Sec.~2.7]{boucheron2013concentration}.
We arrive at the following tail inequalities.

\begin{fact}[Poisson Concentration] \label{fact:poisson}
Let $Y$ be a (zero-mean) random variable.  For $v \geq 0$ and $t \geq 0$,
\begin{align}
\xi_Y(\theta) \leq v \cdot \frac{\econst^{+a\theta} - a\theta - 1}{a^2} \quad\text{on $\theta \geq 0$}
\quad\text{implies}\quad
\log \Prob{ Y \geq +t }
	&\leq \frac{-v}{a^2} \psi\left(\frac{+at}{v}\right); \label{eqn:ProbY+} \\
\xi_Y(\theta) \leq v \cdot \frac{\econst^{-a\theta} + a\theta - 1}{a^2} \quad\text{on $\theta \leq 0$}
\quad\text{implies}\quad
\log \Prob{ Y \leq -t }
	&\leq \frac{-v}{a^2} \psi\left(\frac{-at}{v}\right). \label{eqn:ProbY-}
\end{align}
The function $\psi(u) := (1 + u) \log(1 + u) - u$ for all $u \in \R$,
with the convention $\psi(u) = + \infty$ for $u < -1$.
\end{fact}

\begin{proof}
Suppose that~\eqref{eqn:mgfY+} is in force.  Then
Fact~\ref{fact:cgf-tail} implies the tail bound~\eqref{eqn:ProbY+}:
$$
\log \Prob{ Y \geq t }
	\leq \inf_{\theta > 0}\ \left[ \xi_Y(\theta ) - \theta t \right]
	\leq \inf_{\theta > 0}\ \left[ v \cdot \frac{\econst^{a \theta} - a \theta - 1}{a^2} - \theta t \right]
	= \frac{v}{a^2} \psi\left( \frac{at}{v} \right).
$$
The infimum is attained at $\theta = a^{-1} \log(1 + at/v)$.
The tail bound~\eqref{eqn:ProbY-} follows from the parallel argument.
\end{proof}

The tail function $\psi$ satisfies the inequality $\psi(u) \geq (u^2/2)/(1+u/3)$ for all $u \in \R$.
Therefore, each of the bounds~\eqref{eqn:ProbY+} and~\eqref{eqn:ProbY-} implies
weaker, but more interpretable, results.  For all $t \geq 0$,
\begin{equation} \label{eqn:bernstein} \log \Prob{ Y \geq + t} \leq \frac{-t^2/2}{v + at / 3}
\quad\text{and}\quad
\log \Prob{ Y \leq - t} \leq \frac{-t^2/2}{(v - at / 3)_+}.
\end{equation}
The tail bounds in~\eqref{eqn:bernstein} are called Bernstein inequalities;
it is also common to combine them into a single formula.

\section{From Thermal Variance Bounds to Concentration}

Pursuing the ideas in Section~\ref{sec:genfun},
we can now derive concentration inequalities for
the weighted intrinsic volume random variables
as a consequence of bounds on the thermal variance.
For the moment, we will take the thermal variance
bounds for granted.  The subsequent sections
develop the machinery required to control
the thermal variance of this type of random
variable.

\subsection{Rotation Volumes}

To begin, we present detailed concentration results for
rotation volume random variables, introduced
in Section~\ref{sec:wvol}.

\begin{theorem}[Rotation Volumes: Variance and Cgf] \label{thm:rotvol-conc}
Let $\set{K} \subset \R^n$ be a nonempty convex body with
rotation volume random variable $\rotI_{\set{K}}$
and central rotation volume $\rotdelta(\set{K}) = \Expect \rotI_{\set{K}}$.
The variance of $\rotI_{\set{K}}$ satisfies
$$
\Var[ \rotI_{\set{K}} ] \leq \rotdelta(\set{K}).
$$
The cgf of $\rotI_{\set{K}}$ satisfies
$$
\xi_{\rotI_{\set{K}}}(\theta) \leq \rotdelta(\set{K}) \cdot (\econst^{\theta} - \theta - 1)
\quad\text{for all $\theta \in \R$.}
$$
\end{theorem}

\noindent
The proof of Theorem~\ref{thm:rotvol-conc} begins in Section~\ref{sec:rotvol-from-tv}
and is completed in Section~\ref{sec:rotvol-conc}

Combining Theorem~\ref{thm:rotvol-conc} with Fact~\ref{fact:poisson},
we immediately obtain probability bounds for the rotation volume
random variable $\rotI_{\set{K}}$ of a convex body $\set{K}$.
\begin{align*}
\log \Prob{ \rotI_{\set{K}} \geq \rotdelta(\set{K}) + t }
	&\leq  -\rotdelta(\set{K})\cdot \psi(+t/\rotdelta(\set{K})) && \quad\text{for $0 \leq t$;} \\
\log \Prob{ \rotI_{\set{K}} \leq \rotdelta(\set{K}) - t }
	&\leq -\rotdelta(\set{K})\cdot \psi(-t/\rotdelta(\set{K})) && \quad\text{for $0 \leq t \leq \rotdelta(\set{K})$.}
\end{align*}
The tail function $\psi$ is defined in Fact~\ref{fact:poisson}.
The Poisson-type bound implies a weaker Bernstein inequality of the form~\eqref{eqn:bernstein}:
\begin{equation} \label{eqn:rotvol-bernstein}
\log \Prob{ \frac{\rotI_{\set{K}} - \rotdelta(\set{K})}{t} \geq 1 }
	\leq \frac{-t^2/2}{\rotdelta(\set{K}) + \abs{t}/3}
	\quad\text{for all $t \neq 0$.}
\end{equation}
The unusual formulation~\eqref{eqn:rotvol-bernstein} captures both the lower and upper tail in a single expression.
Theorem~\ref{thm:rotvol-intro}, in the introduction,
rephrases the latter inequality in a more standard way.
It leads to the phase transition for random projections, Theorem~\ref{thm:randproj-intro}.

As a further consequence, we can deduce probability inequalities for sums.
Let $\rotI_{\set{K}}$ and $\rotI_{\set{M}}$ be \emph{independent} rotation
volume random variables, associated with two convex bodies $\set{K}, \set{M} \subset \R^n$.
Define the sum $\rotDelta(\set{K},\set{M}) := \rotdelta(\set{K}) + \rotdelta(\set{M})$
of central rotation volumes.  Then
\begin{equation} \label{eqn:rotvol-sum}
\log \Prob{ \frac{(\rotI_{\set{K}} + \rotI_{\set{M}}) - \rotDelta(\set{K},\set{M}) }{t} \geq 1 }
	\leq \frac{-t^2/2}{\rotDelta(\set{K},\set{M}) + \abs{t}/3}
	\quad\text{for $t \neq 0$.}
\end{equation}
This result follows from the fact~\eqref{eqn:cgf-indep}
that cgfs are additive, from Theorem~\ref{thm:rotvol-conc},
and from Fact~\ref{fact:poisson}.
It plays a key role in the proof of Theorem~\ref{thm:rotmean-intro},
the phase transition for rotation means.

\subsubsection{Reduction to Thermal Variance Bound}
\label{sec:rotvol-from-tv}

Theorem~\ref{thm:rotvol-conc} follows from a bound on the
thermal variance of the rotation volume random variable.

\begin{proposition}[Rotation Volumes: Thermal Variance] \label{prop:rotvol-tvar}
Let $\set{K} \subset \R^n$ be a nonempty convex body.  The thermal variance
of the rotation volume random variable $\rotI_{\set{K}}$ satisfies
$$
\xi_{\rotI_{\set{K}}}''(\theta) \leq \xi_{\rotI_{\set{K}}}'(\theta)
\quad\text{for all $\theta \in \R$.}
$$
\end{proposition}

\noindent
The proof of Proposition~\ref{prop:rotvol-tvar} is the object
of Section~\ref{sec:rotvol-conc}.  Right now, we can use it to derive
concentration for the rotation volumes.

\begin{proof}[Proof of Theorem~\ref{thm:rotvol-conc} from Proposition~\ref{prop:rotvol-tvar}]
To obtain the variance bound, note that
$$
\Var[ \rotI_{\set{K}} ] \stackrel{\eqref{eqn:cgf-zero}}{=} \xi_{\rotI_{\set{K}}}''(0) \leq \xi_{\rotI_{\set{K}}}'(0)
	\stackrel{\eqref{eqn:cgf-zero}}{=} \Expect \rotI_{\set{K}}
	= \rotdelta(\set{K}).
$$
The inequality is Proposition~\ref{prop:rotvol-tvar}.

To derive the cgf bound, we pass to the zero-mean random variable
\begin{equation*}
Y := \rotI_{\set{K}} - \Expect \rotI_{\set{K}} = \rotI_{\set{K}} - \rotdelta(\set{K}).
\end{equation*}
The cgfs of $\rotI_{\set{K}}$ and $Y$ are related as
\begin{equation*} \label{eqn:cgf-I-Y}
\xi_{\rotI_{\set{K}}}(\theta) = \xi_{Y + \rotdelta(\set{K})}(\theta) \overset{\eqref{eqn:cgf-shift}}{=} \xi_Y(\theta) + \rotdelta(\set{K}) \,\theta
\quad\text{for all $\theta \in \R$.}
\end{equation*}
Assume that $\theta>0$. Using Proposition~\ref{prop:rotvol-tvar} again, we find that
\begin{equation*}
  \xi'_Y(\theta)
  \stackrel{\eqref{eqn:thermal-mean}}{=} \int_0^\theta \xi''_Y(s) \idiff{s}
  \stackrel{\eqref{eqn:thermal-var-shift}}{=} \int_0^\theta \xi''_{\rotI_{\set{K}}}(s) \idiff{s}
  \leq \int_0^{\theta} \xi'_{\rotI_{\set{K}}}(s) \idiff{s}
\stackrel{\eqref{eqn:cgf-zero}}{=} \xi_Y(\theta) + \rotdelta(\set{K}) \, \theta.
\end{equation*}
For $\theta < 0$, the inequality is reversed.  An application of Lemma~\ref{lem:psi'topsi}
with $a = 1$ and $v = \rotdelta(\set{K})$ furnishes the result.
\end{proof}

\subsection{Rigid Motion Volumes}

We continue with a refined concentration result for the rigid motion volume
random variable, introduced in Section~\ref{sec:wvol-conc-intro}.

\begin{theorem}[Rigid Motion Volumes: Variance and Cgf] \label{thm:rmvol-conc}
Let $\set{K} \subset \R^n$ be a nonempty convex body with
rigid motion random variable $\rmI_{\set{K}}$ and
central rigid motion volume $\rmdelta(\set{K}) := \Expect \rmI_{\set{K}}$.
Define the complement
$
\rmdelta_{\circ}(\set{K}) := (n+1) - \rmdelta(\set{K}).
$
The variance of $\rmI_{\set{K}}$ satisfies
$$
\Var[ \rmI_{\set{K}} ] \leq \frac{2 \rmdelta(\set{K}) \rmdelta_{\circ}(\set{K})}{n+1} =: \bar{v}(\set{K}).
$$
The cgf of $\rmI_{\set{K}}$ satisfies
\begin{align*}
\xi_{\rmI_{\set{K}}}(\theta) &\leq \bar{v}(\set{K}) \cdot \frac{\econst^{\beta_{\circ} \theta} - \beta_{\circ} \theta - 1}{\beta_{\circ}^2}
	\quad\text{for $\theta \geq 0$}
	&\text{where}&&
	\beta_{\circ} := \frac{2\rmdelta_{\circ}(\set{K})}{n+1} < 2; \\
\xi_{\rmI_{\set{K}}}(\theta) &\leq \bar{v}(\set{K}) \cdot \frac{\econst^{\beta \theta} - \beta \theta - 1}{\beta^2}
	\quad\text{for $\theta \leq 0$}
	&\text{where}&&
	\beta := \frac{2\rmdelta(\set{K})}{n+1} < 2. 
\end{align*}
\end{theorem}

\noindent
The proof of Theorem~\ref{thm:rmvol-conc} appears in Section~\ref{sec:rmvol-from-tv}
and Section~\ref{sec:rmvol-conc}.

Together, Theorem~\ref{thm:rmvol-conc} and Fact~\ref{fact:poisson}
deliver concentration inequalities for the rigid motion volumes.
For all $t \geq 0$,
\begin{align*}
\log \Prob{ \rmI_{\set{K}} \geq \rmdelta(\set{K}) + t }
	&\leq \frac{-(n+1)}{2\rmdelta_{\circ}(\set{K})} \cdot \rmdelta(\set{K}) \cdot \psi( t/\rmdelta(\set{K}) ); \\
\log \Prob{ \rmI_{\set{K}} \leq \rmdelta(\set{K}) - t }
	&\leq \frac{-(n+1)}{2\rmdelta(\set{K})} \cdot \rmdelta_{\circ}(\set{K}) \cdot \psi( t /\rmdelta_{\circ}(\set{K}) ). 
\end{align*}
For $t \neq 0$, the rigid motion random variable satisfies a Bernstein inequality~\eqref{eqn:bernstein}
of the form \begin{align}
\log \Prob{ \frac{\rmI_{\set{K}} - \rmdelta(\set{K})}{t} \geq 1 }
\leq \frac{-t^2/4}{(\rmdelta(\set{K}) \wedge \rmdelta_{\circ}(\set{K})) + \abs{t}/3}.  \label{eqn:rmvol-bernstein}
\end{align}
Theorem~\ref{thm:rmvol-intro}, in the introduction, contains a slightly weaker variant of
the last display. The inequality~\eqref{eqn:rmvol-bernstein} contributes to the phase transition for random slices 
that appears in Theorem~\ref{thm:crofton-intro}.

We can also obtain bounds for sums of rigid motion volume random variables.
Let $\rmI_{\set{K}}$ and $\rmI_{\set{M}}$ be \emph{independent} rigid motion
volume random variables associated with convex bodies $\set{K}, \set{M} \subset \R^n$.
Define
\begin{equation*}
\rmDelta(\set{K}, \set{M}) := \rmdelta(\set{K}) + \rmdelta(\set{M})
\quad\text{and}\quad
\bar{v}(\set{K}, \set{M}) := (\rmdelta(\set{K}) \wedge \rmdelta_{\circ}(\set{K})) + (\rmdelta(\set{M}) \wedge \rmdelta_{\circ}(\set{M})).
\end{equation*}
Using the additivity~\eqref{eqn:cgf-indep} of cgfs and making some simple bounds
in Theorem~\ref{thm:rmvol-conc}, we find that 
\begin{equation} \label{eqn:rmvol-sum}
\log \Prob{ \frac{(\rmI_{\set{K}} + \rmI_{\set{M}}) - \rmDelta(\set{K}, \set{M}) }{t} \geq 1 }
	\leq \frac{-t^2/4}{\bar{v}(\set{K}, \set{M}) + \abs{t} / 3}.
\end{equation}
This result yields the phase transition for the kinematic formula,
reported in Theorem~\ref{thm:kinematic-intro}.

\subsubsection{Reduction to Thermal Variance Bound}
\label{sec:rmvol-from-tv}

Theorem~\ref{thm:rmvol-conc} follows from a thermal variance bound.

\begin{proposition}[Rigid Motion Volumes: Thermal Variance] \label{prop:rmvol-tvar}
Let $\set{K} \subset \R^n$ be a nonempty convex body.  The thermal
variance of the rigid motion random variable $\rmI_{\set{K}}$ satisfies
$$
\xi_{\rmI_{\set{K}}}''(\theta) \leq \frac{2}{n+1} \cdot \xi_{\rmI_{\set{K}}}'(\theta) \cdot \big[ (n+1) - \xi_{\rmI_{\set{K}}}'(\theta) \big]
\quad\text{for all $\theta \in \R$.}
$$
\end{proposition}

\noindent
We establish Proposition~\ref{prop:rmvol-tvar} below in Section~\ref{sec:rmvol-conc}.
Here, we show how the bound leads to concentration of the rigid motion volumes.

\begin{proof}[Proof of Theorem~\ref{thm:rmvol-conc} from Proposition~\ref{prop:rmvol-tvar}]
To derive the variance bound, apply Proposition~\ref{prop:rmvol-tvar} to obtain
$$
\Var[ \rmI_{\set{K}} ] \stackrel{\eqref{eqn:cgf-zero}}{=} \xi_{\rmI_{\set{K}}}''(0)
	\leq \frac{2}{n+1} \cdot \xi_{\rmI_{\set{K}}}'(0) \cdot \big[ (n+1) - \xi_{\rmI_{\set{K}}}'(0) \big]
	\stackrel{\eqref{eqn:cgf-zero}}{=} \frac{2}{n+1} \cdot \rmdelta(\set{K}) \cdot \rmdelta_{\circ}(\set{K}).
$$
We have used the definition $\rmdelta(\set{K}) = \Expect \rmI_{\set{K}}$ twice.

For the concentration result, we pass to the zero-mean random variable
$Y:=\rmI_{\set{K}}-\rmdelta(\set{K})$.  First, assume that $\theta>0$.
Integrating the bound from Proposition~\ref{prop:rmvol-tvar}, we find that
\begin{align*}
\xi_Y'(\theta) &\stackrel{\eqref{eqn:thermal-mean}}{\leq} \frac{2}{n+1} \int_0^\theta
	\big[ (n + 1) - \xi_{\rmI_{\set{K}}}'(s) \big] \cdot \xi_{\rmI_{\set{K}}}'(s) \idiff{s}\\
	& \stackrel{\phantom{\eqref{eqn:thermal-mean}}}{\leq} \frac{2}{n+1} \big[ (n + 1) - \xi_{\rmI_{\set{K}}}'(0) \big]
		\int_{0}^\theta \xi_{\rmI_{\set{K}}}'(s) \idiff{s}
	\stackrel{\eqref{eqn:cgf-zero}}{=} \frac{2 \rmdelta_\polar(\set{K})}{n+1}
		\cdot \xi_{\rmI_{\set{K}}}(\theta).
\end{align*}
The second inequality depends on the property that
the derivative $\xi_{\rmI_{\set{K}}}'$ is increasing because
the cgf $\xi_{\rmI_{\set{K}}}$ is convex.  We have also used
fact~\eqref{eqn:thermal-mean-extreme} to infer that the thermal
mean is nonnegative: $\xi_{\rmI_{\set{K}}}'(s) \geq \inf \rmI_{\set{K}} \geq 0$.
Now, invoke~\eqref{eqn:cgf-shift} to pass to the cgf $\xi_Y$ of the zero-mean variable $Y$:
\begin{equation} \label{eqn:diff-ineq-pos}
\xi_Y'(\theta)
	\leq \frac{2 \rmdelta_{\circ}(\set{K})}{n+1}
	\big[ \xi_Y(\theta) + \rmdelta(\set{K})\, \theta \big]
	= \beta_\circ \xi_Y(\theta) + \bar{v}(\set{K}) \, \theta
	\quad\text{for $\theta > 0$.}
\end{equation}
We can invoke the first case of Lemma~\ref{lem:psi'topsi} with $a = \beta_{\circ}$
and $v = \bar{v}(\set{K})$ to obtain the cgf bound for $\theta \geq 0$.

If $\theta< 0$, a straightforward variant of the same argument implies that
\begin{equation} \label{eqn:diff-ineq-neg}
\xi_Y'(\theta)
	\geq \frac{2\rmdelta(\set{K})}{n+1} \left[ \rmdelta_\polar(\set{K})\,\theta  - \xi_Y(\theta) \right]
	= - \beta \xi_Y(\theta) + \bar{v}(\set{K}) \, \theta 
	\quad\text{for $\theta < 0$.}
\end{equation}
The cgf bound for $\theta \leq 0$ follows from the second case of
Lemma~\ref{lem:psi'topsi} with $a=-\beta$ and with $v = \bar{v}(\set{K})$.
\end{proof}

\subsection{Intrinsic Volumes}

Using the same methodology, we can prove a concentration
inequality for the intrinsic volume random variable that improves
significantly over the results in our prior work~\cite{LMNPT20:Concentration-Euclidean}.
These results are not immediately relevant, so we have
postponed them to Appendix~\ref{app:intvol}.

\section{From Distance Integrals to Generating Functions}
\label{sec:distance-integral}

As we have seen, the concentration results for weighted intrinsic volumes
depend on estimates for the thermal variance of the associated random variable.
To obtain these bounds, the first step is to find an alternative expression for
the mgf of the random variable.  The key observation is that the Steiner formula~\eqref{eqn:steiner-intro}
allows us to pass from the discrete weighted intrinsic volume random variable, taking values in
$\{0, 1, \dots, n\}$, to a continuous random variable, taking values in $\R^n$.
This section contains the background for this argument.
The next two sections execute the approach for the
two sequences of weighted intrinsic volumes.

\subsection{The Distance to a Convex Body}

To develop this approach, we first introduce some functions related to
the Euclidean distance between a point and a convex body.

\begin{definition}[Distance to a Convex Body] \label{def:dist}
The \term{distance} to the nonempty convex body $\set{K} \subset \R^n$
is the function
$$
\dist_{\set{K}}(\vct{x}) := \min \{ \norm{ \vct{y} - \vct{x}} : \vct{y} \in \set{K} \}
\quad\text{for $\vct{x} \in \R^n$.}
$$
The \term{projection} onto the convex body $\set{K}$ is the (unique) point where the distance is realized:
$$
\proj_{\set{K}}(\vct{x}) := \arg \min \{ \norm{ \vct{y} - \vct{x}} : \vct{y} \in \set{K} \}
\quad\text{for $\vct{x} \in \R^n$.}
$$
The \term{normal vector} to the convex body, induced by a point $\vct{x} \in \R^n$, is
$$
\vct{n}_{\set{K}}(\vct{x}) := \frac{\vct{x} - \proj_{\set{K}}(\vct{x})}{\|\vct{x}-\proj_{\set{K}}(\vct{x})\|}  = \frac{\vct{x} - \proj_{\set{K}}(\vct{x})}{\dist_{\set{K}}(\vct{x})}
	\quad\text{for $\vct{x} \in \R^n\backslash \set{K}$.}
$$
If $\vct{x}\in \set{K}$, we set $\vct{n}_{\set{K}}=\zerovct$. 
\end{definition}

\subsection{The Steiner Formula as a Distance Integral}\label{sub:steiner-dist}

We can reinterpret Steiner's formula~\eqref{eqn:steiner-intro} as a statement that
integrals of the distance to a convex body can be evaluated in terms of the intrinsic volumes
of the convex body.  This perspective dates at least as far back as
Hadwiger's influential paper~\cite{Had75:Willssche}.

\begin{fact}[Generalized Steiner Formula] \label{fact:distance-integral}
For every function $f : \R_+ \to \R$ where the integrals are finite,
\begin{equation} \label{eqn:gen-steiner}
\int_{\R^n} f(\dist_{\set{K}}(\vct{x})) \idiff{\vct{x}}
	= f(0) \cdot \intvol_n(\set{K}) + \sum_{i=0}^{n-1} \left( \int_0^\infty f(r) \cdot r^{n-i-1} \idiff{r} \right) \omega_{n-i} \cdot \intvol_i(\set{K}).
\end{equation}
\end{fact}

\noindent
See~\cite[Prop.~2.4]{LMNPT20:Concentration-Euclidean} for a short proof of this result.

Following~\cite{LMNPT20:Concentration-Euclidean}, we observe that the right-hand
side of the formula~\eqref{eqn:gen-steiner} is a linear function of the sequence of
intrinsic volumes of $\set{K}$.  By a careful choice of the function $f$,
we can express any moment of the intrinsic volume sequence.  \emph{A fortiori},
we can also express any moment of the rotation volumes or the rigid motion volumes.
The value of this approach is that we can treat the distance integrals that arise
using methods from geometric functional analysis.  In particular,
we will exploit variance inequalities for log-concave and concave measures to
obtain bounds for distance integrals and, thereby, for weighted intrinsic volumes.

\begin{remark}[Related Work]
A rudimentary form of this argument first appeared in our
paper~\cite{ALMT14:Living-Edge} on phase transitions in conic geometry.
Later, the papers~\cite{MT14:Steiner-Formulas} and~\cite{GNP17:Gaussian-Phase}
demonstrated the full power of this approach for treating the conic intrinsic
volumes.  Our recent work~\cite{LMNPT20:Concentration-Euclidean} contains an
initial attempt to execute a similar method for Euclidean intrinsic volumes.
In this paper, we have found a seamless implementation of the idea.
\end{remark}

\subsection{Properties of the Distance Function}

The distance function enjoys a number of elegant properties,
which we record for later reference.

\begin{fact}[Properties of the Distance Function] \label{fact:dist}
The maps appearing in Definition~\ref{def:dist} satisfy
the following properties.
\begin{enumerate}
\item\label{dist:1} The function $\dist_{\set{K}}$ and its square $\dist_{\set{K}}^2$ are convex;
\item\label{dist:2}	The function $\dist_{\set{K}}^2$ is everywhere differentiable, and $\grad \dist_{\set{K}}^2(\vct{x}) = 2 \dist_{\set{K}}(\vct{x}) \, \vct{n}_{\set{K}}(\vct{x})$;
\item\label{dist:3} The Hessian of the squared distance satisfies $(\Hess \dist^2_{\set{K}}(\vct{x})) \, \vct{n}_{\set{K}}(\vct{x}) = 2\vct{n}_{\set{K}}(\vct{x})$.
\end{enumerate}
\end{fact}

\begin{proof}[Proof Sketch]
The convexity of $\dist_{\set{K}}(\vct{x})$ follows
because it is the minimum of a jointly convex function of $(\vct{x}, \vct{y})$
with respect to the variable $\vct{y}$, and the convexity of $\dist_{\set{K}}^2$
is a standard consequence of the convexity and nonnegativity of $\dist_{\set{K}}$.
For Claim~\eqref{dist:2}, the differentiability and the gradient of the squared distance function, see~\cite[Thm.~2.26(b)]{RW98:Variational-Analysis}. Claim~\eqref{dist:3} follows by computing the gradient on both sides of the identity $4\dist^2_{\set{K}}(\vct{x}) = \|\grad \dist_{\set{K}}^2(\vct{x})\|^2$ and equating terms.
\end{proof}

\section{Rotation Volumes: Thermal Variance Bound}
\label{sec:rotvol-conc}

In this section, we establish Proposition~\ref{prop:rotvol-tvar}, the thermal
variance bound for the rotation volumes.
The basic idea is to invoke the
generalized Steiner formula (Fact~\ref{fact:distance-integral})
to rewrite the mgf of the rotation volume
random variable in terms of a distance integral.  After some further
manipulations, we can bound the distance integral using a functional
inequality for log-concave probability measures.

We begin with the rotation volumes
because the ideas shine through most brightly.
The other cases follow the same pattern of
argument, but the mass of detail becomes denser.

\subsection{Setup}

Fix a nonempty convex body $\set{K} \subset \R^n$. This is the only convex body that appears in this section, so we will ruthlessly suppress it from the notation.
For each $i = 0, 1, 2, \dots, n$, let $\intvol_i$
denote the $i$th intrinsic volume of $\set{K}$.
The rotation volumes $\rotv_i$ and the total rotation volume $\rotwills$ are given by
\begin{equation} \label{eqn:rotvol-def-pf}
\rotv_i := \frac{\omega_{n+1}}{\omega_{i+1}} \intvol_{n-i}
\quad\text{and}\quad
\rotwills := \sum_{i=0}^n \rotv_i.
\end{equation}
The rotation volume random variable $\rotI$ follows the distribution
\begin{equation} \label{eqn:rotvol-rv-def}
\Prob{ \rotI = n - i } = \rotv_{i} / \rotwills
\quad\text{for $i = 0, 1, 2, \dots, n$.}
\end{equation}
The central rotation volume $\rotdelta := \Expect \rotI$.
We also abbreviate $\dist := \dist_{\set{K}}$.  

\subsection{The Distance Integral}
\label{sec:rotvol-distint}

First, we express the exponential moments of the sequence
of rotation volumes as a distance integral.

\begin{proposition}[Rotation Volumes: Distance Integral] \label{prop:rotvol-distint}
For each $\theta \in \R$, define the convex potential
\begin{equation} \label{eqn:rotvol-potent}
\rotJ(\vct{x}) := 2 \pi \econst^{\theta} \dist(\vct{x})
\quad\text{for $\vct{x} \in \R^n$.}
\end{equation}
For any function $h : \R_+ \to \R$ where the expectations on the right-hand side are finite,
\begin{equation} \label{eqn:rotvol-distint}
\frac{\omega_{n+1} \econst^{n\theta}}{2} \int_{\R^n} h(\rotJ(\vct{x})) \cdot \econst^{ - \rotJ(\vct{x}) } \idiff{\vct{x}}
	= \sum_{i=0}^n \Expect[ h(G_{i}) ] \, \econst^{(n-i) \theta} \cdot \rotv_i. \end{equation}
The random variable $G_i \sim \textsc{gamma}(i, 1)$ follows the gamma distribution with shape parameter $i$
and scale parameter $1$.  By convention, $G_0 = 0$.
\end{proposition}

\begin{proof}
To verify that the potential $\rotJ$ defined in~\eqref{eqn:rotvol-potent}
is convex, simply invoke Fact~\ref{fact:dist}\eqref{dist:1},
which states that the distance function is convex.

To obtain the moment identity, instantiate Fact~\ref{fact:distance-integral} with the function
$f(r) = h(2\pi \econst^{\theta} r) \, \econst^{-2\pi \econst^\theta r}$. We ascertain that
\begin{align*}
\int_{\R^n} h(\rotJ(\vct{x})) \cdot \econst^{-\rotJ(\vct{x})} \idiff{\vct{x}}
&= h(0) \cdot \intvol_n + \sum_{i=0}^{n-1} \left(\int_0^\infty h(2\pi \econst^{\theta} r) \cdot \econst^{-2\pi\econst^{\theta}r} r^{n-i-1}  \idiff{r} \right)
	\omega_{n-i} \cdot \intvol_i \\
	&= \frac{2}{\omega_{n+1}} \left[ h(0) \cdot \rotv_0 + \sum_{i=1}^n \left( \int_0^\infty h(2\pi \econst^{\theta} r) \cdot \econst^{-2\pi\econst^{\theta}r} r^{i-1} \idiff{r} \right)
	\frac{\omega_{i} \omega_{i+1}}{2} \cdot \rotv_i \right].
\end{align*}We have used the definition~\eqref{eqn:rotvol-def-pf} of $\rotv_i$ and the fact~\eqref{eqn:ball-sphere} that $\omega_1 = 2$.
To handle the integrals, make the change of variables $2\pi\econst^{\theta} r \mapsto s$.
We recognize an expectation with respect to the gamma distribution:
$$
\int_0^\infty h(2\pi \econst^{\theta} r) \cdot r^{i-1} \econst^{-2\pi\econst^{\theta}r} \idiff{r}
	= \frac{1}{(2\pi \econst^{\theta})^{i}} \int_{0}^\infty h(s) \cdot s^{i-1} \econst^{-s} \idiff{s}
	= \frac{\Gamma(i)}{(2\pi \econst^{\theta})^{i}} \Expect[ h(G_i) ].
$$
Altogether,
$$
\frac{\omega_{n+1}}{2} \int_{\R^n} h(\rotJ(\vct{x})) \cdot \econst^{-\rotJ(\vct{x})} \idiff{\vct{x}}
	= \sum_{i=0}^n \frac{\Gamma(i) \, \omega_i \omega_{i+1}}{2 (2\pi)^i} \cdot \Expect[ h(G_i) ] \, \econst^{-i \theta}\cdot \rotv_i.
$$
Owing to the formula~\eqref{eqn:ball-sphere} for the surface area and the Legendre
duplication formula, each of the fractions on the right-hand side equals one.
Multiply each side by $\econst^{n\theta}$ to complete the argument.
\end{proof}

As an immediate corollary, we obtain the metric representations of the total rotation volume
and the central rotation volume, reported in Proposition~\ref{prop:rotvol-metric}.

\begin{proof}[Proof of Proposition~\ref{prop:rotvol-metric}]
Apply Proposition~\ref{prop:rotvol-distint} with $h(s) = 1$ and then with $h(s) = s$, using the fact that $\Expect G_i=i$.
Select the parameter $\theta = 0$.
\end{proof}

\subsection{A Family of Log-Concave Measures}

Proposition~\ref{prop:rotvol-distint} allows us to obtain an alternative
expression for the mgf $m_{\rotI}$ of the rotation volume random variable.  Indeed,
we can choose $h(s) = 1$ to obtain
\begin{equation} \label{eqn:rotvol-mgf}
m_{\rotI}(\theta) \stackrel{\eqref{eqn:mgf}}{=} \Expect \econst^{\theta \rotI}
	\stackrel{\eqref{eqn:rotvol-rv-def}}{=} \frac{1}{\rotwills}\sum_{i=0}^{n} \econst^{(n-i)\theta} \cdot \rotv_i 
	= \frac{\omega_{n+1} \econst^{n\theta}}{2 \rotwills}
	\int_{\R^n} \econst^{-\rotJ(\vct{x})} \idiff{\vct{x}}.
\end{equation}
For each $\theta \in \R$, construct the probability measure $\rotmu_{\theta}$
with the log-concave density
\begin{equation*} \label{eqn:rotvol-measure}
\frac{\diff\rotmu_{\theta}(\vct{x})}{\diff{\vct{x}}} = \frac{\omega_{n+1}}{2 \rotwills} \cdot \frac{ \econst^{n\theta}}{ m_{\rotI}(\theta)} \cdot \econst^{-\rotJ(\vct{x})} \quad\text{for $\vct{x} \in \R^n$.}
\end{equation*}
Owing to~\eqref{eqn:rotvol-mgf}, the measure $\rotmu_{\theta}$ has total mass one.
The density is log-concave because $\rotJ$ is convex.

\begin{corollary}[Rotation Volumes: Probabilistic Formulation] \label{cor:rotvol-lc}
Instate the notation from Proposition~\ref{prop:rotvol-distint}.
Draw a random vector $\vct{z} \sim \rotmu_{\theta}$.  Then
$$
\Expect[ h(\rotJ(\vct{z})) ]
	=
\frac{1}{m_{\rotI}(\theta)} \sum_{i=0}^n \Expect[ h(G_{i}) ]\, \econst^{(n-i) \theta} \cdot \Prob{ \rotI = n - i }.
$$
\end{corollary}

\begin{proof}
This result is a straightforward reinterpretation of Proposition~\ref{prop:rotvol-distint}.
\end{proof}

\subsection{A Thermal Variance Identity}

With Corollary~\ref{cor:rotvol-lc} at hand, we quickly obtain
an alternative formula for $\xi_{\rotI}''(\theta)$, the thermal variance~\eqref{eqn:thermal-var} of the rotation volume random variable $\rotI$.
These results are phrased in terms of
a variance with respect to the measure $\rotmu_{\theta}$.

\begin{lemma}[Rotation Volumes: Thermal Variance Identity] \label{lem:rotvol-tmv}
Draw a random vector $\vct{z} \sim \rotmu_{\theta}$. The thermal mean and variance of the rotation volume random variable satisfy
\begin{align} 
\xi_{\rotI}'(\theta) &= \Expect[ n - \rotJ(\vct{z}) ]; \notag \\
\xi_{\rotI}''(\theta) &= \Var[ \rotJ(\vct{z}) ] - (n - \xi_{\rotI}'(\theta)). \label{eqn:rotvol-tmv}
\end{align}
\end{lemma}

\begin{proof}
The gamma random variables $G_i$ satisfy the identities
$$
\Expect[ n - G_{i} ] = n - i
\quad\text{and}\quad
\Expect[ (n - G_{i})^2 - G_{i} ] = (n-i)^2.
$$
First, we apply Corollary~\ref{cor:rotvol-lc} with $h(s) = n - s$
to arrive at
\begin{equation} \label{eqn:rotvol-tm-pf}
\xi_{\rotI}'(\theta) \overset{\eqref{eqn:thermal-expect}}{=} \frac{m_{\rotI}'(\theta)}{m_{\rotI}(\theta)}
	\stackrel{\eqref{eqn:thermal-expect}}{=} \frac{1}{m_{\rotI}(\theta)} \sum_{i=0}^n (n - i) \, \econst^{(n-i)\theta} \cdot \Prob{ \rotI = n - i }
	= \Expect[ n - \rotJ(\vct{z}) ].
\end{equation}
Next, we apply Corollary~\ref{cor:rotvol-lc} with $h(s) = (n-s)^2 - s$.
This step yields
\begin{align*}
\xi_{\rotI}''(\theta) \stackrel{\eqref{eqn:thermal-var}}{=} \frac{m_{\rotI}''(\theta)}{m_{\rotI}(\theta)} - (\xi_{\rotI}'(\theta))^2 
	&\stackrel{\eqref{eqn:thermal-var}}{=} \left[ \frac{1}{m_{\rotI}(\theta)} \sum_{i=0}^n (n-i)^2 \, \econst^{(n-i) \theta} \cdot \Prob{ \rotI = n - i } \right] - (\xi_{\rotI}'(\theta))^2 \\
	&\stackrel{\eqref{eqn:rotvol-tm-pf}}{=} \Expect[ (n - \rotJ(\vct{z}))^2 - \rotJ(\vct{z}) ] - (\Expect[n - \rotJ(\vct{z})])^2 \\
	&\stackrel{\phantom{\eqref{eqn:rotvol-tm-pf}}}{=} \Var[ n - \rotJ(\vct{z}) ] + \Expect[ n - \rotJ(\vct{z}) ] - n \\
	&\stackrel{\eqref{eqn:rotvol-tm-pf}}{=} \Var[ \rotJ(\vct{z}) ] + (\xi_{\rotI}'(\theta) - n).
\end{align*}
We used the invariance properties of the variance. This is the advertised result.
\end{proof}

\subsection{Variance of Information}

We have now converted the problem of bounding the thermal variance
of the rotation volume random variable into a problem about the
variance of a function of a log-concave random variable.
To control this variance, we require a recent result on
the information content of a log-concave random variable.

\begin{fact}[Nguyen, Wang] \label{fact:nguyen-wang}
Let $J : \R^n \to \R$ be a convex potential, and 
consider the log-concave probability measure $\mu$ on $\R^n$
whose density is proportional to $\econst^{-J}$.  Then
\begin{equation}\label{eqn:nguyen-wang}
 \Var_{\vct{z} \sim \mu}[J(\vct{z})] \leq n.
\end{equation}
\end{fact}

Fact~\ref{fact:nguyen-wang} was obtained independently by Nguyen~\cite{Ngu13:Inegalites-Fonctionelles}
and by Wang~\cite{wang2014heat} in their doctoral theses.  For short proofs of this result,
we refer to~\cite[Thm.~2.3]{FMW16:Optimal-Concentration} and~\cite[Rem.~4.2]{BGG18:Dimensional-Improvements}.
The variance bound~\eqref{eqn:nguyen-wang} holds with equality when the potential $J$ is 1-homogeneous,
but we believe that improvements remain possible.

\subsection{Proof of Proposition~\ref{prop:rotvol-tvar}}

We must bound the thermal variance $\xi_{\rotI}''(\theta)$ of the rotation volume random variable.
To do so, we combine Lemma~\ref{lem:rotvol-tmv} and Fact~\ref{fact:nguyen-wang}:
$$
\xi_{\rotI}''(\theta)
	\overset{\eqref{eqn:rotvol-tmv}}{=} \Var[ \rotJ(\vct{z}) ] - (n - \xi_{\rotI}'(\theta))
	\overset{\eqref{eqn:nguyen-wang}}{\leq} n - (n - \xi_{\rotI}'(\theta))
	= \xi_{\rotI}'(\theta).
$$
This completes the proof of Proposition~\ref{prop:rotvol-tvar}
and, thus, the proof of Theorem~\ref{thm:rotvol-conc}.

\section{Rigid Motion Volumes: Thermal Variance Bound}
\label{sec:rmvol-conc}

This section contains the proof of Proposition~\ref{prop:rmvol-tvar},
following the approach from Section~\ref{sec:rotvol-conc}.  This case
is significantly more complicated than the case of rotation volumes.
Nevertheless, the overall strategy remains the same.

\subsection{Setup}

Fix a nonempty convex body $\set{K} \subset \R^n$, which will not appear
in the notation.  For each $i = 0, 1, 2, \dots, n$, let $\intvol_i$ be the
$i$th intrinsic volume of $\set{K}$.  Introduce the rigid motion volumes $\rmv_i$
and the total rigid motion volume $\rmwills$:
\begin{equation} \label{eqn:rmvol-def-pf}
\rmv_i := \frac{\omega_{n+1}}{\omega_{i+1}} \intvol_i
\quad\text{and}\quad
\rmwills := \sum_{i=0}^n \rmv_i.
\end{equation}
The rigid motion volume random variable $\rmI$ has the distribution
\begin{equation} \label{eqn:rmvol-rv-pf}
\Prob{ \rmI = n-i} = \rmv_i / \rmwills
\quad\text{for $i = 0, 1, 2, \dots, n$.}
\end{equation}
Write $\rmdelta := \Expect \rmI $ for the expectation.
We also abbreviate $\dist := \dist_{\set{K}}$.

\subsection{The Distance Integral}
\label{sec:rmvol-distint}

We begin with a distance integral that generates the exponential
moments of the rigid motion volumes.

\begin{proposition}[Rigid Motion Volumes: Distance Integral] \label{prop:rmvol-potent}
For each $\theta \in \R$, define the convex potential
\begin{equation} \label{eqn:rmvol-potent}
\rmJ(\vct{x}) := \big[ 1 + \econst^{-2\theta} \dist^2(\vct{x}) \big]^{1/2}
\quad\text{for $\vct{x} \in \R^n$.}
\end{equation}
For any function $h : [0,1] \to \R$ where the expectations on the right-hand side are finite,
$$
\int_{\R^n} h\big( 1 - \rmJ(\vct{x})^{-2} \big) \cdot \rmJ(\vct{x})^{-(n+1)} \idiff{\vct{x}}
	= \sum_{i=0}^n \Expect[ h(B_{n-i}) ] \, \econst^{(n-i) \theta} \cdot \rmv_i.
$$
The random variable $B_{n-i} \sim \textsc{beta}((n-i)/2, (i+1)/2)$ follows a beta distribution.
By convention $B_0 = 0$.
\end{proposition}

\begin{proof}
The potential $\rmJ$ defined in~\eqref{eqn:rmvol-potent} is convex
because of Fact~\ref{fact:dist}\eqref{dist:1} and elementary convexity arguments.
To obtain the identity, we invoke Fact~\ref{fact:distance-integral} with the function
$$
f(r) = \frac{h\big(1 - (1+\econst^{-2\theta} r^2 )^{-1} \big)}{(1+\econst^{-2\theta} r^2)^{(n+1)/2}}
	= \frac{h\big( \econst^{-2\theta} r^2 / (1 + \econst^{-2\theta} r^2) \big)}{(1+\econst^{-2\theta} r^2)^{(n+1)/2}}.
$$
This yields a formula for the moments of the rigid motion volume sequence:
\begin{align*}
\int_{\R^n} &h\big(1 - \rmJ(\vct{x})^{-2} \big) \cdot \rmJ(\vct{x})^{-(n+1)} \idiff{\vct{x}} \\
	&= h(0) \cdot \intvol_n + \sum_{i=0}^{n-1} \left( \int_0^\infty h\left(\frac{\econst^{-2\theta} r^2}{1 + \econst^{-2\theta} r^2}\right) \cdot \frac{r^{n-i-1}}{(1 + \econst^{-2\theta} r^2)^{(n+1)/2}} \idiff{r} \right) \omega_{n-i} \cdot \intvol_i \\
	&= h(0) \cdot \rmv_n + \sum_{i=0}^{n-1} \left( \frac{\econst^{(n-i)\theta}}{2} \int_0^1 h(s) \cdot s^{(n-i)/2 - 1} (1 - s)^{(i+1)/2 - 1} \idiff{s} \right) \frac{ \omega_{n-i}\omega_{i+1}}{ \omega_{n+1} }\cdot \rmv_i
\end{align*}We have used the definition~\eqref{eqn:rmvol-def-pf} of $\rmv_i$
and made the change of variables $\econst^{-2\theta} r^2 / (1+ \econst^{-2\theta} r^2) \mapsto s$.
Using the formula~\eqref{eqn:ball-sphere} for the surface area of a sphere, we see that the
weight in each term coincides with a beta function:
$$
\frac{\omega_{n-i} \omega_{i+1}}{2 \omega_{n+1}}
	= \frac{\Gamma((n+1)/2)}{\Gamma((n-i)/2) \, \Gamma((i+1)/2)}
	= \frac{1}{\mathrm{B}((n-i)/2, (i+1)/2)}.
$$
Now, we recognize that the integral and the surface area constants combine to produce
the expectation $\Expect[ h(B_{n-i}) ]$ with respect to the beta random variable $B_{n-i}$.
Last, combine the displays. \end{proof}

As a consequence, we arrive at metric representations for the total rigid motion volume
and the central rigid motion volume, stated in Proposition~\ref{prop:rmvol-metric}.

\begin{proof}[Proof of Proposition~\ref{prop:rmvol-metric}]
Apply Proposition~\ref{prop:rmvol-potent} with $h(s) = 1$ and then with $h(s) = (n+1) s$,
using the fact that $\Expect[B_{n-i}]=(n-i)/(n+1)$. 
Select the parameter $\theta = 0$.
\end{proof}

\subsection{A Family of Concave Measures}

Using Proposition~\ref{prop:rmvol-potent}, we can express the mgf $m_{\rmI}$
of the rigid motion volume random variable as a distance integral.
Choosing $h(s) = 1$, we find that
\begin{equation} \label{eqn:rmvol-mgf}
m_{\rmI}(\theta) \stackrel{\eqref{eqn:mgf}}{=} \Expect \econst^{\theta \rmI}
	\stackrel{\eqref{eqn:rmvol-rv-pf}}{=} \frac{1}{\rmwills} \sum_{i=0}^n \econst^{(n-i)\theta} \cdot \rmv_i
	= \int_{\R^n} \rmJ(\vct{x})^{-(n+1)} \idiff{\vct{x}}.
\end{equation}
For each $\theta \in \R$, construct the concave probability measure $\rmmu_{\theta}$
with the density
$$
\frac{\diff\rmmu_{\theta}(\vct{x})}{\diff{\vct{x}}}
	= \frac{1}{\rmwills} \cdot \frac{1}{m_{\rmI}(\theta)} \cdot \rmJ(\vct{x})^{-(n+1)}
	\quad\text{for $\vct{x} \in \R^n$.}
$$
The identity~\eqref{eqn:rmvol-mgf} ensures that $\rmmu_{\theta}$ is normalized.

\begin{corollary}[Rigid Motion Volumes: Probabilistic Formulation] \label{cor:rmvol-lc}
Instate the notation of Proposition~\ref{prop:rmvol-potent}.  Draw a random vector
$\vct{z} \sim \rmmu_{\theta}$.  Then
$$
\Expect\big[ h\big(1 - \rmJ(\vct{z})^{-2} \big) \big]
	= \frac{1}{m_{\rmI}(\theta)} \sum_{i=0}^n \Expect[ h(B_{n-i}) ] \, \econst^{(n-i)\theta} \cdot \Prob{ \rmI = n - i }.
$$
\end{corollary}

\begin{proof}
This is just a restatement of Proposition~\ref{prop:rmvol-potent}.
\end{proof}

\subsection{A Thermal Variance Identity}

Using Corollary~\ref{cor:rmvol-lc}, we can evaluate the thermal
mean $\xi_{\rmI}'$ and the thermal variance $\xi_{\rmI}''$
of the rigid motion volume random variable $\rmI$ in terms
of integrals against the concave measure $\rmJ$.

\begin{lemma}[Rigid Motion Volumes: Thermal Variance Identity] \label{lem:rmvol-tmv}
Draw a random vector $\vct{z} \sim \rmmu_{\theta}$.  Then
\begin{align}
\xi_{\rmI}'(\theta) &= (n+1) \cdot \Expect[ 1 - \rmJ(\vct{z})^{-2} ]; \label{eqn:rmvol-tm} \\ \xi_{\rmI}''(\theta) &= (n+1)(n+3) \cdot \Var[ \rmJ(\vct{z})^{-2} ] 
	- \frac{2 \xi_{\rmI}'(\theta) \cdot [(n+1) - \xi_{\rmI}'(\theta)]}{n+1}. \label{eqn:rmvol-tv}
\end{align}
\end{lemma}

\begin{proof}
The beta random variable $B_{n-i}$ satisfies
$$
(n+1) \, \Expect[ B_{n-i} ] = n-i
\quad\text{and}\quad
(n+1) \Expect[ (n+3) B_{n-i}^2 - 2 B_{n-i} ] = (n-i)^2.
$$
Invoke Corollary~\ref{cor:rmvol-lc} with $h(s) = (n+1) s$ to obtain
\begin{equation} \label{eqn:rmvol-tm-pf}
\xi_{\rmI}'(\theta) \stackrel{\eqref{eqn:thermal-expect}}{=}
	\frac{1}{m_{\rmI}(\theta)} \sum_{i=0}^n (n-i) \,\econst^{(n-i)\theta} \cdot \Prob{ \rmI = n-i }
	= (n+1) \cdot \Expect [ 1 - \rmJ(\vct{z})^{-2} ].
\end{equation}
Apply Corollary~\ref{cor:rmvol-lc} with $h(s) = (n+1)((n+3) s^2 - 2s )$ to obtain
\begin{align*}
\xi_{\rmI}''(\theta) &\stackrel{\eqref{eqn:thermal-var}}{=}
	\left[ \frac{1}{m_{\rmI}(\theta)} \sum_{i=0}^n (n-i)^2 \,\econst^{(n-i)\theta} \cdot \Prob{ \rmI = n-i } \right]
		- (\xi_{\rmI}'(\theta))^2 \\
	&\stackrel{\eqref{eqn:rmvol-tm-pf}}{=} (n+1) \Expect \big[ (n+3) \big(1 - \rmJ(\vct{z})^{-2} \big)^2 - 2 \big(1-\rmJ(\vct{z})^{-2}\big) \big]
		- (n+1)^2 \big(\Expect[ 1 - \rmJ(\vct{z})^{-2} ] \big)^2 \\
	&\stackrel{\phantom{\eqref{eqn:rmvol-tm-pf}}}{=} (n+1)(n+3) \Var[ 1 - \rmJ(\vct{z})^{-2} ]
		- 2 (n+1)\left( \Expect[ 1 - \rmJ(\vct{z})^{-2} ] - (\Expect[ 1 - \rmJ(\vct{z})^{-2} ])^2 \right) \\
	&\stackrel{\eqref{eqn:rmvol-tm-pf}}{=} (n+1)(n+3) \Var[ \rmJ(\vct{z})^{-2} ]
		- \frac{2 \xi_{\rmI}'(\theta) \cdot[(n+1) - \xi_{\rmI}'(\theta)]}{n+1}.
\end{align*}
This is the required result.
\end{proof}

\subsection{A Variance Bound}

To compute exponential moments of the rigid motion volume random variable $\rmI$,
we can use variance inequalities for concave measures.  We will establish the following estimate.

\begin{lemma}[Rigid Motion Volumes: Variance Bound] \label{lem:rmvol-varbd}
Instate the notation of Lemma~\ref{lem:rmvol-tmv}.  For a random variable $\vct{z} \sim \rmmu_{\theta}$,
\begin{equation} \label{eqn:rmvol-varbd}
\Var[ \rmJ(\vct{z})^{-2} ]
	\leq \frac{4}{n+4}  \Expect[ \rmJ(\vct{z})^{-2}] \cdot \Expect[ 1 - \rmJ(\vct{z})^{-2} ]. 
\end{equation}
\end{lemma}

To prove this result, we will use the following statement.

\begin{fact}[Nguyen] \label{fact:nguyen}
Let $J : \R^n \to \R_{++}$ be a twice differentiable and strongly convex potential.
Consider the concave probability measure $\mu$ on $\R^n$ whose density is proportional
to $J^{-(n+1)}$.  For any differentiable function $f : \R^n \to \R$,
$$
\Var_{\vct{z} \sim \mu}[ f(\vct{z}) ]	
	\leq \frac{1}{n} \int_{\R^n} \ip{ (\Hess J(\vct{x}))^{-1} \grad f(\vct{x}) }{ \grad f(\vct{x}) } J(\vct{x}) \idiff{\mu}(\vct{x}).
$$
\end{fact}

Fact~\ref{fact:nguyen} extends the Brascamp--Lieb variance inequality
to a special class of concave measures.  Nguyen~\cite{Ngu14:Dimensional-Variance}
proved this result using H{\"o}rmander's $L_2$ method.  It improves on a
bound that Bobkov \& Ledoux~\cite{BL09:Weighted-Poincare-Type}
derived from the Borell--Brascamp--Lieb variance inequality~\cite{Bor75:Convex-Set,BL76:Extensions-Brunn-Minkowski}.

\begin{proof}[Proof of Lemma~\ref{lem:rmvol-varbd}]
Fix the parameter $\theta \in \R$. To apply Nguyen's variance inequality, we need to perturb the potential $\rmJ$
so that it is strongly convex.  For each $\eps > 0$, set
\begin{equation} \label{eqn:igJ-perturb}
J^{\eps}(\vct{x}) := \rmJ(\vct{x}) + \eps \normsq{\vct{x}} / 2,
\quad\text{recalling that}\quad
\rmJ(\vct{x}) = \big[ 1 + \econst^{-2\theta} \dist^2(\vct{x}) \big]^{1/2}. \end{equation}
Next, define the function $f$ and compute its gradient:
\begin{equation} \label{eqn:igf-grad}
f(\vct{x}) = \rmJ(\vct{x})^{-2}
\quad\text{and}\quad
\grad f(\vct{x}) = \frac{-2}{\rmJ(\vct{x})^3} \grad \rmJ(\vct{x}).
\end{equation}
Let us continue to the main argument.

Our duty is to evaluate the quadratic form induced by the inverse Hessian
of the perturbed potential $J^{\eps}$ at the vector $\grad f$.  
First, use the definition~\eqref{eqn:rmvol-potent}
of the potential $\rmJ$ and Fact~\ref{fact:dist} to calculate that
\begin{equation} \label{eqn:igJ-gradnorm}
\norm{ \smash{\grad \rmJ(\vct{x})} }^2 = \econst^{-2\theta} \big( 1 - \rmJ(\vct{x})^{-2} \big).
\end{equation}
Differentiate~\eqref{eqn:igJ-gradnorm} again:
\begin{equation} \label{eqn:igJ-hessian}
(\Hess \rmJ(\vct{x})) \, \grad \rmJ(\vct{x})
	= \frac{1}{2} \grad \norm{ \smash{\grad \rmJ(\vct{x})} }^2
	= \frac{1}{2} \grad \big[ \econst^{-2\theta} (1 - \rmJ(\vct{x})^{-2}) \big]
= \frac{-\econst^{-2\theta}}{2} \grad f(\vct{x}).
\end{equation}
We have used~\eqref{eqn:igf-grad} in the last step.  Consequently,
the perturbed potential~\eqref{eqn:igJ-perturb} satisfies
$$
(\Hess J^{\eps}(\vct{x})) \, \grad \rmJ(\vct{x})
	= \frac{-\econst^{-2\theta}}{2} \grad f(\vct{x}) + \eps \grad \rmJ(\vct{x})
	= \frac{-1}{2} \big[ \econst^{-2\theta} + \eps \rmJ(\vct{x})^3 \big] \grad f(\vct{x}).
$$
Multiply through by the inverse Hessian and rearrange:
$$
(\Hess J^{\eps}(\vct{x}))^{-1} \grad f(\vct{x})
	= \frac{-2}{\econst^{-2\theta} + \eps \rmJ(\vct{x})^3} \grad \rmJ(\vct{x}).
$$
Take the inner product with $\grad f$ and multiply by $J^{\eps}$:
\begin{align*}
&\ip{ (\Hess J^{\eps}(\vct{x}))^{-1} \grad f(\vct{x}) }{ \grad f(\vct{x}) } J^{\eps}(\vct{x}) \\
&\qquad= \frac{4}{\econst^{-2\theta} + \eps \rmJ(\vct{x})^3} \rmJ(\vct{x})^{-3} \norm{ \smash{ \grad\rmJ(\vct{x})}}^2 J^{\eps}(\vct{x}) \\
&\qquad= \frac{4 \econst^{-2\theta}}{\econst^{-2\theta} + \eps \rmJ(\vct{x})^3}
	\big[ \rmJ(\vct{x})^{-2} \big(1 - \rmJ(\vct{x})^{-2}\big) + \eps \norm{\vct{x}}^2 \big(1 - \rmJ(\vct{x})^{-2}\big) / 2 \big].
\end{align*}
We have used~\eqref{eqn:igf-grad} to reach the second line.
The third line requires~\eqref{eqn:igJ-gradnorm} and
the definition~\eqref{eqn:igJ-perturb} of the perturbed potential $J^{\eps}$.
Taking the limit as $\eps \downarrow 0$,
$$
\ip{ (\Hess \rmJ(\vct{x}))^{-1} \grad f(\vct{x}) }{ \grad f(\vct{x} } \rmJ(\vct{x})
	= 4 \rmJ(\vct{x})^{-2} \big(1 - \rmJ(\vct{x})^{-2}\big)
$$
This computation gives us the integrand in Nguyen's variance inequality.

Fact~\ref{fact:nguyen} applies to the strongly convex potential $J^{\eps}$ and the function $f$.
By dominated convergence, the inequality also remains valid in the limit as $\eps \downarrow 0$.
Thus, for $\vct{z} \sim \rmmu_{\theta}$,
$$
\Var[ \rmJ(\vct{z})^{-2} ]
	= \Var[ f(\vct{z}) ]
	\leq \frac{4}{n} \Expect\big[ \rmJ(\vct{z})^{-2} \big( 1 - \rmJ(\vct{z})^{-2} \big) \big].
$$
The rest is simple algebra.  Rewrite the expectation on the right-hand side as
$$
\Expect\big[ \rmJ(\vct{z})^{-2} \big( 1 - \rmJ(\vct{z})^{-2} \big) \big]
	= \Expect[ \rmJ(\vct{z})^{-2} ] \cdot \Expect[  1 - \rmJ(\vct{z})^{-2} ] - \Var[ \rmJ(\vct{z})^{-2} ].
$$
Some manipulations yield the bound
$$
\Var[ \rmJ(\vct{z})^{-2} ]
	\leq \frac{4}{n+4} \Expect[ \rmJ(\vct{z})^{-2} ] \cdot \Expect[  1 - \rmJ(\vct{z})^{-2} ].
$$
This is the advertised result.
\end{proof}

\subsection{Proof of Proposition~\ref{prop:rmvol-tvar}}

We may now complete the proof of the thermal variance bound, Proposition~\ref{prop:rmvol-tvar}.
First, Lemmas~\ref{lem:rmvol-tmv} and~\ref{lem:rmvol-varbd} imply that
\begin{align} \label{eqn:rmvol-tvar-pf}
(n+1)(n+3) \cdot \Var[\rmJ(\vct{z})^{-2}]
	&\stackrel{\eqref{eqn:rmvol-varbd}}{\leq} 4(n+1) \cdot \Expect[ \rmJ(\vct{z})^{-2}] \cdot \Expect[ 1 - \rmJ(\vct{z})^{-2} ] \\
	&\stackrel{\eqref{eqn:rmvol-tm}}{=} \frac{4 \xi_{\rmI}'(\theta) \cdot[(n+1) - \xi_{\rmI}'(\theta)]}{n+1}.
\end{align}
Applying Lemma~\ref{lem:rmvol-tmv} again,
\begin{align*}
\xi_{\rmI}''(\theta)
	&\stackrel{\eqref{eqn:rmvol-tv}}{=} (n+1)(n+3) \cdot \Var[\rmJ(\vct{z})^{-2}] - \frac{2\xi_{\rmI}'(\theta)((n+1) - \xi_{\rmI}'(\theta))}{n+1} \\
	&\stackrel{\eqref{eqn:rmvol-tvar-pf}}{\leq} \frac{2 \xi_{\rmI}'(\theta) \cdot [(n+1) - \xi_{\rmI}'(\theta)]}{n+1}. 
\end{align*}
We have finished the argument.

\section{Moving Flats: Phase Transitions}
\label{sec:moving-flats}

We are now prepared to demonstrate the existence of phase transitions
in classic integral geometry problems.  The first step in our program
is to rewrite integral geometry formulas in terms
of weighted intrinsic volumes (Appendix~\ref{sec:formulas}).
Then we can apply concentration of the
weighted intrinsic volumes to obtain new insights on integral geometry.

This section treats two questions involving moving flats.
There are strong formal parallels between these two results,
reflecting the duality between projections and slices.
In Section~\ref{sec:moving-bodies}, we turn to problems involving moving convex bodies.

\subsection{Random Projections}
\label{sec:randproj}

First, we take up the question about projecting a convex body
onto a random subspace; see Figure~\ref{fig:random-projection}
for a schematic.

\subsubsection{The Projection Formula}

The projection formula describes the average rotation volumes
of a projection of a convex body onto a random subspace of
a given dimension.

\begin{fact}[Projection Formula] \label{fact:projection-formula}
Consider a nonempty convex body $\set{K} \subset \R^n$.
For each subspace dimension $m =0,1,2,\dots, n$,
and each index $i = 0,1,2,\dots, m$,
\begin{equation} \label{eqn:proj-pf}
\int_{\Gr(m, n)} \rotv_i^m(\set{K}|\set{L}) \, {\nu}_m(\diff{\set{L}})
	= \rotv_{n-m + i}^n(\set{K})
	\quad\text{for $i = 0, 1, 2, \dots, m$.}
\end{equation}
The superscripts indicate the ambient dimension in which the rotation volumes are computed, viz.
$$
\rotv_i^m(\set{M}) := \frac{\omega_{m+1}}{\omega_{i+1}} \, \intvol_{m - i}(\set{M}),
\quad\text{while}\quad
\rotv_i^n(\set{M}) := \frac{\omega_{n+1}}{\omega_{i+1}} \, \intvol_{n - i}(\set{M})
$$
The Grassmannian $\Gr(m, n)$ is equipped with its invariant probability measure
$\nu_m$; see Appendix~\ref{sec:invariant-grass}.
\end{fact}

Fact~\ref{fact:projection-formula} follows from Fact~\ref{fact:projection-intvol},
the formulation in terms of intrinsic volumes, 
and Lemma~\ref{lem:structure}.  Roughly speaking,
it shows that random projection acts as a shift
operator on the sequence of rotation volumes.  The case $i = 0$ corresponds to
Kubota's formula.  The case $i = 0$ and $m = n - 1$ is equivalent to Cauchy's
surface area formula.

We can give a probabilistic interpretation of the random
projection formula in terms of the rotation volume random variable.
Define
\begin{equation} \label{eqn:proj-wills-pf}
\mathrm{RandProj}_m(\set{K})
	:= \int_{\Gr(m, n)} \frac{\rotwills^m(\set{K}|\set{L})}{\rotwills^n(\set{K})} \, {\nu}_m(\diff{\set{L}})
	= \sum_{i=0}^m \frac{\rotv_{n-i}^n(\set{K})}{\rotwills^n(\set{K})}
	= \Prob{ \rotI_{\set{K}} \leq m }.
\end{equation}
The superscript on the total rotation volumes $\rotwills^m$ and $\rotwills^n$
reflects the ambient dimension.
The second relation follows when we sum~\eqref{eqn:proj-pf} over indices $i = 0, 1, 2, \dots, m$
and then divide by the total rotation volume $\rotwills^n(\set{K})$.
In the last step, we have identified a probability involving $\rotI_{\set{K}}$,
the rotation volume random variable~\eqref{eqn:wvol-rv-def}.
The exact formula~\eqref{eqn:proj-wills-pf} also implies that $\mathrm{RandProj}_m(\set{K})$
is bounded between zero and one.

\subsubsection{The Approximate Projection Formula}

In this section, we complete the proof of Theorem~\ref{thm:randproj-intro},
the phase transition for random projections.  To do so, we reparameterize the dimension $m$
of the random subspace as $m = \rotdelta(\set{K}) + t$ for $t \in \R$.
The Bernstein inequality~\eqref{eqn:rotvol-bernstein} for
rotation volume random variable states that,
for $t \neq 0$,
$$
\Prob{ \frac{\rotI_{\set{K}} - \rotdelta(\set{K})}{t} \geq 1 } \leq \exp \left( \frac{-t^2/2}{\rotdelta(\set{K}) + \abs{t}/3} \right)
	=: p(t).
$$
For a proportion $\alpha \in (0, 1)$, invert the tail probability function to see that
$$
p(t) \leq \alpha
\quad\text{if and only if}\quad
\abs{t} \geq \big[ 2  \rotdelta(\set{K}) \log(1/\alpha) + \tfrac{1}{9} \log^2(1/\alpha) \big]^{1/2} + \tfrac{1}{3} \log(1/\alpha). 
$$
Using the subadditivity of the square root, we can relax this to a simpler bound.
Define the transition width
$$
t_{\star}(\alpha) := \big[ 2\rotdelta(\set{K}) \log(1/\alpha)\big]^{1/2} + \tfrac{2}{3} \log(1/\alpha).
$$
We have shown that
$$
\abs{t} \geq t_{\star}(\alpha)
\quad\text{implies}\quad
\Prob{ \frac{\rotI_{\set{K}} - \rotdelta(\set{K})}{t} \geq 1 }
	\leq \alpha.
$$
Use the relation $m = \rotdelta(\set{K}) + t$ and rearrange to obtain the
pair of implications
\begin{align*}
m \leq \rotdelta(\set{K}) - t_{\star}(\alpha)
\quad\text{implies}\quad
\Prob{ \rotI_{\set{K}} \leq m } &\leq \alpha; \quad\text{and} \\ m \geq \rotdelta(\set{K}) + t_{\star}(\alpha)
\quad\text{implies}\quad
\Prob{ \rotI_{\set{K}} \geq m } &\leq \alpha
\quad\text{implies}\quad
\Prob{ \rotI_{\set{K}} \leq m } \geq 1 - \alpha.
\end{align*}
Last, invoke the formula~\eqref{eqn:proj-wills-pf} to replace
the probability, $\Prob{ \rotI_{\set{K}} \leq m }$,
by the integral geometric quantity, $\mathrm{RandProj}_m(\set{K})$.
This completes the proof of Theorem~\ref{thm:randproj-intro}.

\subsection{Random Slices} \label{sec:crofton}

Next, we consider the dual problem of slicing a convex body
by a random affine space.  An illustration appears
in Figure~\ref{fig:crofton}.

\subsubsection{The Slicing Formula}

The slicing formula, which is attributed to Crofton,
describes the total content of the slices
of a convex body by affine spaces of a given dimension.

\begin{fact}[Slicing Formula] \label{fact:slicing}
Consider a nonempty convex body $\set{K} \subset \R^n$.
For each affine space dimension $m =0,1,2,\dots,n$
and each index $i = 0,1,2,\dots, m$,
\begin{equation} \label{eqn:slicing-pf}
\int_{\Af(m,n)} \rmv_i(\set{K} \cap \set{E}) \, \mu_m(\diff{\set{E}})
	= \rmv_{n-m+i}(\set{K}).
\end{equation}
The rigid motion volumes are calculated with respect
to the ambient dimension $n$.
The affine Grassmannian $\Af(m,n)$ is equipped with its
invariant measure $\mu_m$; see Appendix~\ref{sec:invariant}.
\end{fact}

Fact~\ref{fact:slicing} follows from Fact~\ref{fact:slicing-intvol} and Lemma~\ref{lem:structure}.  Our formulation shows that random slicing acts as a
shift operator on the sequence of rigid motion volumes;
note the formal similarity between~\eqref{eqn:slicing-pf}
and the analogous result~\eqref{eqn:proj-pf} for random projections.
Of particular interest is the case $i = 0$, which expresses
the measure of the set of affine spaces that hit $\set{K}$.

As above, the formula~\eqref{eqn:slicing-pf} has a probabilistic interpretation.
Define
\begin{equation} \label{eqn:crofton-wills-pf}
\mathrm{RandSlice}_m(\set{K}) := \int_{\Af(m,n)} \frac{\rmwills(\set{K} \cap \set{E})}{\rmwills(\set{K})} \, \mu_m(\diff{\set{E}})
	= \sum_{i=0}^m \frac{\rmv_{n-i}(\set{K})}{\rmwills(\set{K})}
	= \Prob{ \rmI_{\set{K}} \leq m }.
\end{equation}
To reach the second relation, sum~\eqref{eqn:slicing-pf} over $i = 0,1,2,\dots, m$,
and divide by the total rigid motion volume $\rmwills(\set{K})$.
We have recognized a probability involving $\rmI_{\set{K}}$,
the rigid motion random variable~\eqref{eqn:wvol-rv-def}.

In particular, the formula~\eqref{eqn:crofton-wills-pf} demonstrates
that $\mathrm{RandSlice}_m(\set{K})$ is bounded between zero and one.
A striking feature of the slicing problem is that a much stronger condition
holds:
$$
0 \leq \frac{\rmwills(\set{K} \cap \set{E})}{\rmwills(\set{K})} \leq 1
\quad\text{for all $\set{E} \in \Af(m,n)$.}
$$
This relation is a straightforward consequence of the fact that
intrinsic volumes are monotone increasing with respect to set inclusion.

\subsubsection{The Approximate Slicing Formula}

Let us continue with the proof of Theorem~\ref{thm:crofton-intro}.
The argument is almost identical with the proof of Theorem~\ref{thm:randproj-intro},
so we will leave out some details.
Rewrite the dimension $m$ of the affine
space as $m = \rmdelta(\set{K}) + t$ for $t \in \R$.  The Bernstein inequality~\eqref{eqn:rmvol-bernstein}
for the rigid motion volume random variable states that, for $t \neq 0$,
$$
\Prob{ \frac{\rmI_{\set{K}} - \rmdelta(\set{K})}{t} \geq 1 }
	\leq \exp\left( \frac{-t^2/4}{\rmsigma^2(\set{K}) + \abs{t}/3} \right)
	\quad\text{where}\quad
	\rmsigma^2(\set{K}) := \rmdelta(\set{K}) \wedge \big((n+1) - \rmdelta(\set{K})\big).
$$
For a proportion $\alpha \in (0,1)$, the transition width is
$$
t_{\star}(\alpha) := \big[ 4 \rmsigma^2(\set{K}) \log(1/\alpha) \big]^{1/2} + \tfrac{4}{3} \log(1/\alpha).
$$
Repeating the arguments from Section~\ref{sec:randproj},
\begin{equation*}
\begin{aligned}
m &\leq \rmdelta(\set{K}) - t_{\star}(\alpha)
&\quad\text{implies}\quad&&
\mathrm{RandSlice}_m(\set{K}) &\leq \alpha; \\
m &\geq \rmdelta(\set{K}) + t_{\star}(\alpha)
&\quad\text{implies}\quad&&
\mathrm{RandSlice}_m(\set{K}) &\geq 1 - \alpha.
\end{aligned}
\end{equation*}
Collect the results to obtain Theorem~\ref{thm:crofton-intro}.

\section{Moving Bodies: Phase Transitions}
\label{sec:moving-bodies}

In this section, we consider problems involving two moving bodies.
The approach is similar to the argument in Section~\ref{sec:moving-flats}.
The main technical difference is that the integral geometry formulas
involve weighted intrinsic volumes of both convex bodies.

\subsection{Rotation Means}
\label{sec:rotmean}

First, we consider what happens when we add a randomly rotated convex body
to a fixed convex body.  See Figure~\ref{fig:rotation-mean} for a picture.

\subsubsection{The Rotation Mean Formula}

The rotation mean formula gives an exact account of the rotation
volumes of the Minkowski sum of a fixed convex body and a randomly
rotated convex body.

\begin{fact}[Rotation Mean Value Formula] \label{fact:rotmean}
Consider two nonempty convex bodies $\set{K}, \set{M} \subset \R^n$.
Then
\begin{equation} \label{eqn:rotmean-pf}
\int_{\SO(n)} \rotv_i(\set{K} + \mtx{O} \set{M}) \, \nu(\diff{\mtx{O}})
	= \sum_{j \geq i} \rotv_j(\set{K}) \, \rotv_{n + i-j}(\set{M})
	\quad\text{for $i = 0,1,2,\dots, n$.}
\end{equation}
The rotation volumes are computed with respect to the ambient dimension $n$.
The special orthogonal group $\SO(n)$ is equipped with its invariant probability
measure $\nu$; see Appendix~\ref{sec:invariant-SO}.
\end{fact}

Fact~\ref{fact:rotmean} is attributed to Hadwiger.
Our formulation follows from Fact~\ref{fact:rotmean-intvol} and Lemma~\ref{lem:structure}.
It shows that a rotation mean corresponds with a convolution of the rotation volumes.
The special case $\set{M} = \lambda \ball{n}$ is equivalent to the
Steiner formula~\eqref{eqn:steiner-intro}.

We can easily obtain a probabilistic interpretation of the rotation
mean value formula.  Define
\begin{equation} \label{eqn:rotmean-wills-pf}
\mathrm{RotMean}(\set{K}, \set{M}) :=
	\int_{\SO(n)} \frac{\rotwills(\set{K} + \mtx{O} \set{M})}{\rotwills(\set{K}) \, \rotwills(\set{M})} \, \nu(\diff{\mtx{O}})
	= \sum_{i + j \geq n} \frac{\rotv_i(\set{K}) \, \rotv_j(\set{M})}{\rotwills(\set{K}) \, \rotwills(\set{M})}
	= \Prob{ \rotI_{\set{K}} + \rotI_{\set{M}} \geq n }.
\end{equation}
To obtain the second relation, sum~\eqref{eqn:rotmean-pf} over $i = 0, 1,2, \dots, n$,
and divide by the product $\rotwills(\set{K}) \, \rotwills(\set{M})$ of the total rotation volumes.  
In this expression, $\rotI_{\set{K}}$ and $\rotI_{\set{M}}$ are \emph{independent}
rotation volume random variables~\eqref{eqn:wvol-rv-def} associated with
the convex bodies $\set{K}$ and $\set{M}$.
As a consequence of~\eqref{eqn:rotmean-wills-pf}, we immediately recognize that
the integral $\mathrm{RotMean}(\set{K}, \set{M})$ is bounded between zero and one.

\subsubsection{The Approximate Rotation Mean Formula}

We may now establish Theorem~\ref{thm:rotmean-intro}.
Define the total rotation volume
$$
\rotDelta(\set{K}, \set{M}) := \rotdelta(\set{K}) + \rotdelta(\set{M})
\in [0, 2n].
$$
Parameterize the phase transition by writing $\rotDelta(\set{K}, \set{M}) = n + t$
for $t \in \R$.  The Bernstein inequality~\eqref{eqn:rotvol-sum} for the sum of rotation mean
random variables states that
$$
\Prob{ \frac{\rotI_{\set{K}} + \rotI_{\set{M}} - \rotDelta(\set{K},\set{M})}{t} \geq 1 }
	\leq \exp\left( \frac{-t^2/2}{\rotDelta(\set{K},\set{M}) + \abs{t} / 3} \right)
	\quad\text{for $t \neq 0$.}
$$
For a proportion $\alpha \in (0,1)$, set the transition width
$$
t_{\star}(\alpha) := \big[ 2 \rotDelta(\set{K},\set{M}) \log(1/\alpha) \big]^{1/2} + \tfrac{2}{3} \log(1/\alpha).
$$
We find that
$$
\abs{t} \geq t_{\star}(\alpha)
\quad\text{implies}\quad
\Prob{ \frac{\rotI_{\set{K}} + \rotI_{\set{M}} - \rotDelta(\set{K},\set{M})}{t} \geq 1 }
	\leq \alpha.
$$
Rearrange this inequality to determine that
\begin{equation}
\begin{aligned}
\rotDelta(\set{K},\set{M}) &\geq n + t_{\star}(\alpha)
&\quad\text{implies}\quad&&
\mathrm{RotMean}(\set{K}, \set{M})
	&\leq \alpha; \\
\rotDelta(\set{K},\set{M}) &\leq n - t_{\star}(\alpha)
&\quad\text{implies}\quad&&
\mathrm{RotMean}(\set{K}, \set{M})
	&\geq 1 - \alpha.	
\end{aligned}
\end{equation}
We have established Theorem~\ref{thm:rotmean-intro}.

\subsection{The Kinematic Formula}
\label{sec:kinematic}

Last, we turn to the kinematic formula, which describes how two
moving convex bodies intersect;  
see the illustration in Figure~\ref{fig:kinematic}.
This setting is dual to the
problem of rotation means.  

\subsubsection{Exact Kinematics}

The kinematic formula expresses the rigid motion volumes
of an intersection of a fixed convex body and a randomly
transformed convex body.

\begin{fact}[Kinematic Formula] \label{fact:kinematic}
Consider two nonempty bodies $\set{K}, \set{M} \subset \R^n$.
For each $i = 0,1,2,\dots,n$,
\begin{equation} \label{eqn:kinematic-pf}
\int_{\RM(n)} \rmv_i(\set{K} \cap g \set{M}) \, \mu(\diff{g})
	= \sum_{j \geq i} \rmv_j(\set{K}) \, \rmv_{n+i-j}(\set{M})
	\quad\text{for $i = 0, 1, 2, \dots, n$.}
\end{equation}
The rigid motion volumes are calculated with respect
to the ambient dimension $n$.
The group $\RM(n)$ of proper rigid motions on $\R^n$ is equipped with
its invariant measure $\mu$; see Section~\ref{sec:invariant}.
\end{fact}

Fact~\ref{fact:kinematic} goes back to work of Blaschke, Santal{\'o},
Chern and others.  The statement here follows from Fact~\ref{fact:kinematic-intvol}
and Lemma~\ref{lem:structure}.
Our formulation states that random intersection acts as convolution
on sequences of rigid motion volumes; note the formal parallel
between~\eqref{eqn:kinematic-pf} and~\eqref{eqn:rotmean-pf}, the expression for rotation means.
The special case $i = 0$ is called the \term{principal kinematic formula},
and it describes the measure of the set of rigid motions that bring $\set{M}$
into contact with $\set{K}$.

To obtain a probabilistic formulation, define
\begin{equation} \label{eqn:kinematic-wills-pf}
\mathrm{Kinematic}(\set{K}, \set{M}) 
	:= \int_{\RM(n)} \frac{\rmwills(\set{K} \cap g \set{M})}{\rmwills(\set{K}) \, \rmwills(\set{M})} \, \mu(\diff{g})
	= \sum_{i + j \geq n} \frac{\rmv_i(\set{K}) \, \rmv_j(\set{K})}{\rmwills(\set{K}) \, \rmwills(\set{M})}
	= \Prob{ \rmI_{\set{K}} + \rmI_{\set{M}} \geq n }.
\end{equation}
The second relation follows when we sum~\eqref{eqn:kinematic-pf} over $i = 0,1,2,\dots,n$
and divide by the product $\rmwills(\set{K}) \, \rmwills(\set{M})$ of the total rigid motion volumes.
In this formula, $\rmI_{\set{K}}$ and $\rmI_{\set{M}}$ are \emph{independent} rigid motion volume random
variables associated with the sets $\set{K}$ and $\set{M}$.
The formula~\eqref{eqn:kinematic-wills-pf} also demonstrates
that $\mathrm{Kinematic}(\set{K}, \set{M})$ is bounded between zero and one.

\subsubsection{Approximate Kinematics}

Finally, we establish Theorem~\ref{thm:kinematic-intro}.  Introduce the total
rigid motion volume
$$
\rmDelta(\set{K},\set{M}) := \rmdelta(\set{K}) + \rmdelta(\set{M}) \in [0, 2n].
$$
Define the variance proxy
$$
\bar{v}(\set{K},\set{M}) := \big[ \rmdelta(\set{K}) \wedge ((n+1) - \rmdelta(\set{K})) \big] + \big[ \rmdelta(\set{M}) \wedge ((n+1) - \rmdelta(\set{M})) \big]
	\in [0, n+1].
$$	
We parameterize the phase transition as $\rmDelta(\set{K},\set{M}) = n + t$ for $t \in \R$.
The Bernstein inequality~\eqref{eqn:rmvol-sum} for the sum
of rigid motion volume random variables ensures that
$$
\Prob{ \frac{\rmI_{\set{K}} + \rmI_{\set{M}}-\rmDelta(\set{K},\set{M})}{t} \geq 1 }
	\leq \exp\left( \frac{-t^2/4}{\bar{v}(\set{K},\set{M}) + \abs{t}/3} \right)
	\quad\text{for $t \neq 0$.}
$$
For a proportion $\alpha \in (0,1)$, we define the transition width
$$
t_{\star}(\alpha) := \big[ 4\bar{v}(\set{K},\set{M}) \log(1/\alpha) \big]^{1/2} + \tfrac{4}{3} \log(1/\alpha).
$$
With this notation,
$$
\abs{t} \geq t_{\star}(\alpha)
\quad\text{implies}\quad
\Prob{ \frac{\rmI_{\set{K}} + \rmI_{\set{M}}-\rmDelta(\set{K},\set{M})}{t} \geq 1 } \leq \alpha.
$$
This inequality implies that
\begin{equation*}
\begin{aligned}
\rmDelta(\set{K},\set{M}) &\geq n + t_{\star}(\alpha)
&\quad\text{implies}\quad&&
\mathrm{Kinematic}(\set{K}, \set{M}) &\leq \alpha; \\
\rmDelta(\set{K},\set{M}) &\leq n - t_{\star}(\alpha)
&\quad\text{implies}\quad&&
\mathrm{Kinematic}(\set{K}, \set{M}) &\geq 1 - \alpha.
\end{aligned}
\end{equation*}
This result is stronger than Theorem~\ref{thm:kinematic-intro},
which follows from the loose bound
$\bar{v}(\set{K},\set{M}) \leq \rmDelta(\set{K},\set{M})$.

\subsection{Iteration}

Both the rotation mean formula, Fact~\ref{fact:rotmean},
and the kinematic formula, Fact~\ref{fact:kinematic},
can be iterated to handle combinations of more than
two convex bodies.  For example, see~\cite[Thm.~5.1.5]{SW08:Stochastic-Integral}.
It is straightforward to extend our methods to obtain
phase transitions for combinations of many moving
convex bodies.  We omit further discussion.

\section{Conclusions and Outlook}

To recapitulate: We have discovered that several major problems
in integral geometry exhibit sharp phase transitions.
To explain these phenomena, we cast
the integral geometry formulas in terms of weighted intrinsic volumes.
These reformulations reveal that each result
has a natural expression as a probability involving
weighted intrinsic volume random variables.
The phase transition arises as a consequence of the
fact that these random variables concentrate
around their expectations.

The main technical contribution of the paper
is to prove that weighted intrinsic volumes concentrate.
This argument first rewrites the exponential moments of the
sequence of weighted intrinsic volumes in terms of
a distance integral.  The distance integral
can be interpreted as the variance of a function
with respect to a continuous concave measure.
We bound the distance integral using a variance
inequality for concave measures.
These variance inequalities emerged
from research on functional extensions of
the Brunn--Minkowski inequality, so geometry
is central to our whole program.

Our approach can be generalized in several different directions.
On the probabilistic side, it is likely that weighted
intrinsic volumes satisfy a central limit theorem,
analogous to the one derived for conic intrinsic
volumes in~\cite{GNP17:Gaussian-Phase}.  On the
geometric side, the intrinsic volume sequence
is a special case of the sequence of mixed volumes
of two convex bodies.
We anticipate that our techniques can be used
to prove that the sequence of mixed volumes
also concentrates, and we plan to investigate
consequences of this phenomenon in
geometry, combinatorics, and algebra.

\appendix

\section{Invariant Measures}
\label{sec:invariant}

Integral geometry problems involve integration over spaces of geometric
objects, equipped with invariant measures.  To formulate these questions correctly,
we must be quite explicit about the construction of the measures.
Many mysteries in geometric probability, such as Bertrand's paradox~\cite[pp.~5--6]{Ber89:Calcul-Probabilites},
can be traced to a confusion about the random model.  The material in
this section is summarized from~\cite[Chaps.~5, 13]{SW08:Stochastic-Integral}.

\subsection{Notational Collision}

We have chosen to stick with the standard notation from~\cite{SW08:Stochastic-Integral}
for invariant measures.  Although this notation collides with the symbols we previously used
for concave measures, we do not use these constructions simultaneously, so there
is no risk of confusion.

\subsection{The Special Orthogonal Group}
\label{sec:invariant-SO}

Recall that the special orthogonal group $\SO(n)$ consists
of all rotations acting on $\R^n$.  We identify the group with its
representation as the family of all $n \times n$ orthogonal matrices
with determinant one, acting on itself by matrix multiplication.
Since the special orthogonal group is compact (in the relative topology),
it admits a unique invariant probability measure, denoted by $\nu$. That is,
$$
\nu( \mtx{O} \set{S} ) = \nu( \set{S} )
\quad\text{for all $\mtx{O} \in \SO(n)$ and Borel $\set{S} \subseteq \SO(n)$.}
$$
Furthermore, the measure is scaled so that $\nu(\SO(n)) = 1$.

\subsection{The Grassmannian}
\label{sec:invariant-grass}

The (real) Grassmannian $\Gr(m, n)$ is the collection of all $m$-dimensional subspaces
in the Euclidean space $\R^n$.  The special orthogonal group $\SO(n)$ acts on $\Gr(m,n)$ by rotation.

Fix a reference subspace $\set{L}_{\star} \in \Gr(m,n)$; the particular choice has no downstream effect.  
We can construct a unique rotation-invariant probability measure $\nu_m$ on $\Gr(m, n)$
as the push-forward of the measure $\nu$ on $\SO(n)$ via the map $\mtx{O} \mapsto \mtx{O} \set{L}_{\star}$.
The measure satisfies
$$
\nu_m( \mtx{O} \set{S} ) = \nu_m( \set{S} )
\quad\text{for all $\mtx{O} \in \SO(n)$ and Borel $\set{S} \subseteq \Gr(m,n)$.}
$$
The total measure of the Grassmannian satisfies $\nu_m(\Gr(m,n)) = 1$.
In fact, the invariance property extends to the full orthogonal group
because negation acts as the identity map on the Grassmannian.

\subsection{The Group of Proper Rigid Motions}
\label{sec:invariant-RM}

The special Euclidean group $\RM(n)$ consists of proper
rigid motions on $\R^n$.
More precisely, a proper rigid motion $g$ acts on $\R^n$ first by applying a
rotation $\mtx{O}$ and then applying a translation $\tau_{\vct{x}}$ by a vector $\vct{x}$.
The group product is composition.

Introduce the notation $\mathrm{Leb}_m$ for the Lebesgue measure on
an $m$-dimensional subspace of $\R^n$.  
We construct an invariant measure $\mu$ on $\RM(n)$
as the push-forward of the measure $\nu \times \mathrm{Leb}_n$
on the product space $\SO_n \times \R^n$ under the map $(\mtx{O}, \vct{x}) \mapsto
\tau_{\vct{x}} \circ \mtx{O}$.  In other words,
$$
\mu( g \set{S} ) = \mu( \set{S} )
\quad\text{for all $g \in \RM(n)$ and Borel $\set{S} \subseteq \RM(n)$.}
$$
Unraveling the definitions, we learn that the measure $\mu$
satisfies the normalization
\begin{equation} \label{eqn:RM-normalization}
\mu \{ g \in \RM(n) : g \vct{0}_n \in \ball{n} \} = \kappa_n.
\end{equation}
In other words, the measure of the set of proper rigid motions that
keeps the origin $\vct{0}_n$ within the Euclidean unit ball equals the
Lebesgue measure of the ball.  In~\eqref{eqn:RM-normalization},
we can replace the origin by any other point without changing the construction.
Subject to the normalization~\eqref{eqn:RM-normalization},
the measure $\mu$ is the unique rigid-motion-invariant measure on $\RM(n)$.

\subsection{The Affine Grassmannian}
\label{sec:invariant-Af}

The affine Grassmannian $\Af(m, n)$ consists of all affine subspaces
of $\R^n$ with dimension $m$.  
The group $\RM(n)$ of rigid motions acts on $\Af(m, n)$ in the obvious way.

Fix a reference affine space $\set{E}_{\star} \in \Af(m,n)$.
We can construct a unique measure $\mu_m$ on $\Af(m, n)$ that is invariant
under proper rigid motions as the push-forward of the measure $\mu$
on $\RM(n)$ via the map $g \mapsto g \set{E}_\star$.
The invariance property of the measure $\mu_m$ means that
$$
\mu_m( g \set{S} ) = \mu_m( \set{S} )
\quad\text{for all $g \in \RM(n)$ and Borel $\set{S} \subseteq \Af(m,n)$.}
$$
In fact, the measure is invariant over the full group of rigid motions,
$\mathrm{E}(n)$, because negation is an involution on $\Af(m,n)$.
The normalization is
\begin{equation} \label{eqn:Af-normalization}
\mu_m \{\set{E} \in \Af(m,n) : \ball{n} \cap \set{E} \neq \emptyset\} = \kappa_{n-m}.
\end{equation}
In other words, the measure of the set of $m$-dimensional affine spaces that hit
the Euclidean unit ball equals the volume of the $(n-m)$-dimensional ball.

\section{Elements of Integral Geometry}
\label{sec:elements}

This section summarizes the key properties
of the intrinsic volumes.  Then we outline Hadwiger's approach
to integral geometry, which highlights why intrinsic volumes
are so significant.  As a concrete example of this methodology,
we derive a particular case of Crofton's formula.

\subsection{Properties of (Weighted) Intrinsic Volumes}
\label{sec:intvol-properties}

We begin with an overview of the properties of the
intrinsic volume functionals.
Let $\set{K}, \set{M} \subset \R^n$ be convex bodies.
For each index $i = 0,1,2, \dots, n$, the intrinsic volume $\intvol_i$ is:

\begin{enumerate}
\item	\textbf{Nonnegative:} $\intvol_i(\set{K}) \geq 0$.

\item	\textbf{Monotone:} $\set{M} \subseteq \set{K}$ implies $\intvol_i(\set{M}) \leq \intvol_i(\set{K})$.

\item	\textbf{Homogeneous:}\label{eqn:homogeneous}
$\intvol_i(\lambda \set{K}) = \lambda^i \, \intvol_i(\set{K})$ for each $\lambda \geq 0$.

\item	\textbf{Invariant:}  
$\intvol_i(g \set{K}) = \intvol_i(\set{K})$ for each \emph{rigid motion} $g$.
That is, $g$ acts by rotation, reflection, and translation.

\item	\textbf{Intrinsic:} $\intvol_i(\set{K}) = \intvol_i(\set{K} \times \{ \vct{0}_r \})$
for each natural number $r$.

\item	\textbf{A Valuation:}  $\intvol_i(\emptyset) = 0$.  If $\set{M} \cup \set{K}$ is also a convex body, then
$$
\intvol_i( \set{M} \cap \set{K} ) + \intvol_i( \set{M} \cup \set{K} )
	= \intvol_i(\set{M}) + \intvol_i(\set{K}).
$$
This is a restricted version of the additivity property satisfied by a measure.

\item	\textbf{Continuous:}  If $\set{K}_m \to \set{K}$ in the Hausdorff metric, then
$\intvol_i(\set{K}_m) \to \intvol_i(\set{K})$.
\end{enumerate}

It is clear from the definitions~\eqref{eqn:rotvol-def-intro} and~\eqref{eqn:rmvol-def-intro}
that the rotation volumes, $\rotv_i$,
and the rigid motion volumes, $\rmv_i$,
are not intrinsic, but they enjoy the other six properties on the list.
Note that, because of the indexing, the rotation volume $\rotv_i$
on $\R^n$ is homogeneous of degree $n - i$.
Similarly, the total rotation volume, $\rotwills$, 
and the total rigid motion volume, $\rmwills$, are neither homogeneous nor intrinsic,
but they share the other five properties of the intrinsic volumes.

\subsection{Hadwiger's Approach to Integral Geometry}

There is a remarkable and illuminating path to the main
formulas of integral geometry that proceeds
via a deep theorem of Hadwiger~\cite{Had51:Funktionalsatzes,Had52:Additive-Funktionale,Had57:Vorlesungen}.

\begin{fact}[Hadwiger's Functional Theorem] \label{fact:hadwiger}
Suppose that $F$ is a rigid-motion-invariant, continuous
valuation on the space of convex bodies in $\R^n$.
Then $F$ is a linear combination of the intrinsic
volumes $\intvol_0, \intvol_1, \dots, \intvol_n$.
\end{fact}

\noindent
See~\cite[Thm.~14.4.6]{SW08:Stochastic-Integral}
for a proof of this result.  The critical step is
to verify that the $n$th intrinsic volume, $\intvol_n$,
is the only rigid-motion-invariant, continuous valuation on $\R^n$
that assigns the value one to the cube $\set{Q}^n$
and zero to every convex body with dimension
strictly less than $n$.  This is analogous to
the fact that the Lebesgue measure, $\mathrm{Leb}_n$,
is the unique translation-invariant measure on
Borel sets in $\R^n$.  The technical impediment
to proving Fact~\ref{fact:hadwiger} is that convex bodies
form a smaller class than Borel sets,
so they provide fewer constraints on the valuation.

A canonical method for obtaining an invariant, continuous valuation $F$
is to integrate a continuous valuation $\phi$ over a space of geometric
objects, such as the group $\mathrm{E}(n)$ of all rigid motions.  Then
Fact~\ref{fact:hadwiger} demonstrates that
$$
F(\set{K}) := \int_{\mathrm{E}(n)} \phi(g \set{K}) \, \mu(\diff{g})
	= \sum_{i=0}^n a_i \, \intvol_i(\set{K})
	\quad\text{for $a_i \in \R$.}
$$
To determine the coefficients $a_i$, we typically evaluate the
functional $F$ on simple convex bodies, such as Euclidean balls.
This methodology can be used to obtain
many results in integral geometry.

Our approach to integral geometry replaces the intrinsic
volumes with either the rotation volumes or the rigid motion
volumes, depending on the transformation group in the problem.
In the context of Hadwiger's functional theorem, this simply amounts to a change
of basis for the linear space of invariant, continuous valuations.

\subsection{Example: Crofton's Formula}

To illustrate how Hadwiger's functional theorem is deployed,
let us derive Crofton's slicing formula.
For a convex body $\set{K} \subset \R^n$ and an affine space
$\set{E} \in \Af(m, n)$, introduce the functional
$\phi(\set{K}; \set{E}) := \intvol_0(\set{K} \cap \set{E})$.
We easily verify that $\phi(\cdot; \set{E})$ is a continuous
valuation for each choice of $\set{E}$.  Set $$
F(\set{K}) := \int_{\Af(m,n)} \phi(\set{K}; \set{E}) \, \mu_m(\diff{\set{E}})
	= \int_{\Af(m,n)} \intvol_0(\set{K} \cap \set{E}) \, \mu_m(\diff{\set{E}}).
$$
The functional $F$ inherits the continuous valuation property from $\phi$.
The functional $F$ is also invariant under rigid motions.
Indeed, for each $g \in \mathrm{E}(n)$,
\begin{align*}
F(g \set{K}) &= \int_{\Af(m,n)} \intvol_0( (g \set{K}) \cap \set{E}) \, \mu_m(\diff{\set{E}}) \\
	&= \int_{\Af(m,n)} \intvol_0( \set{K} \cap (g^{-1} \set{E}) ) \, \mu_m(\diff{\set{E}})
	= \int_{\Af(m,n)} \intvol_0( \set{K} \cap \set{E}) \, \mu_m(\diff{\set{E}})
	= F(\set{K}).
\end{align*}
This point follows from the rigid motion invariance of the Euler characteristic $\intvol_0$
and the measure $\mu_m$.  An application of Fact~\ref{fact:hadwiger} yields
$$
F(\set{K}) = \int_{\Af(m,n)} \intvol_0( \set{K} \cap \set{E}) \, \mu_m(\diff{\set{E}})
	= \sum_{i=0}^n a_i \, \intvol_i(\set{K}).
$$
To calculate the coefficients $a_i$, select $\set{K} = \lambda \ball{n}$
for a scalar parameter $\lambda \geq 0$.
$$
F(\lambda \ball{n}) = \int_{\Af(m,n)} \intvol_0((\lambda \ball{n}) \cap \set{E}) \, \mu_m(\diff{\set{E}})
	= \kappa_{n-m} \cdot \lambda^{n-m}.
$$
We have used the formula~\eqref{eqn:Af-normalization} to evaluate the integral.  On the other hand,
using Example~\ref{ex:ball},
$$
F(\lambda \ball{n}) = \sum_{i=0}^n a_i \intvol_i(\lambda \ball{n})
	= \sum_{i=0}^n a_i \cdot \binom{n}{i} \frac{\kappa_{n}}{\kappa_{n-i}} \cdot \lambda^i
$$
Matching terms in the polynomials in the last two displays, we can identify the
value of $\alpha_{n-m}$.  The remaining coefficients vanish:
$a_i = 0$ for $i \neq n - m$.  In summary,
$$
F(\set{K}) = \int_{\Af(m,n)} \intvol_0(\set{K} \cap \set{E}) \, \mu_m(\diff{\set{E}})
	= \binom{n}{m} \frac{\kappa_m \kappa_{n-m}}{\kappa_n} \cdot \intvol_{n-m}(\set{K}).
$$
This is Crofton's famous result.

\section{Reweighting Integral Geometry Formulas}
\label{sec:formulas}

Our work involves nonstandard presentations of classic formulas
from integral geometry.  This appendix describes the simple
tools we need to translate the familiar statements into our language.

\subsection{Structure Coefficients}

Schneider \& Weil~\cite{SW08:Stochastic-Integral}
phrase integral geometry results in terms of
the standard intrinsic volumes.
These formulations involve a family of structure coefficients.

\begin{definition}[Structure Coefficients] \label{def:structure}
For nonnegative integers $i$ and $j$, introduce the \term{structure coefficient}
$$
c_j^i := \frac{i! \kappa_i}{j! \kappa_j} $$
This definition is extended to families $(i_1, \dots, i_r)$ and $(j_1, \dots, j_r)$
of nonnegative integers:
$$
c_{j_1,\dots,j_r}^{i_1,\dots,i_r} := \prod_{s=1}^r c_{j_s}^{i_s}
	= \prod_{s=1}^r \frac{i_s! \kappa_{i_s}}{j_s! \kappa_{j_s}}.
$$
\end{definition}

The following simple result allows us to manipulate the structure coefficients.
It is the source of the weights that appear in the rigid motion volumes
and the rotation volumes.

\begin{lemma}[Structure Coefficients] \label{lem:structure}
Consider four nonnegative integers $i_1,i_2,j_1,j_2$
that satisfy the relation $i_1 + i_2 = j_1 + j_2$.
Then the structure coefficient can be written as
$$
c_{j_1, j_2}^{i_1, i_2} = \frac{\omega_{j_1+1}}{\omega_{i_1+1}} \cdot \frac{\omega_{j_2+1}}{\omega_{i_2+1}}.
$$
\end{lemma}

\begin{proof}
This statement follows instantly from Definition~\ref{def:structure}, the formulas~\eqref{eqn:ball-sphere} for the volume and surface area of the Euclidean ball, and the Legendre duplication formula.
\end{proof}

\subsection{Integral Geometry with Intrinsic Volumes}

For completeness, we give the usual statements of the integral geometry
formulas that we have studied in this paper.  These formulations provide
additional intuition and may be easier to interpret.  
Using Lemma~\ref{lem:structure}, these results
lead to the corresponding statements
in terms of the rotation volumes or the rigid motion volumes.

\subsubsection{Cartesian Products}

First, we note that the intrinsic volumes of a direct product
are given by the convolution of the intrinsic volumes.

\begin{fact}[Cartesian Products] \label{fact:product-intvol}
Consider nonempty convex bodies $\set{K}, \set{M} \subset \R^n$.
For each index $i = 0,1,2,\dots,2n$,
the intrinsic volumes of the product satisfy
$$
\intvol_i(\set{K} \times \set{M})
	= \sum_{j \leq i} \intvol_j(\set{K}) \, \intvol_{i-j}(\set{M}).
$$
\end{fact}

There is a simple proof of Fact~\ref{fact:product-intvol}, due to Hadwiger~\cite{Had75:Willssche},
based on an integral representation (Corollary~\ref{cor:intvol-metric}) of the total intrinsic volume.
See~\cite[Lem.~14.2.1]{SW08:Stochastic-Integral} or~\cite[Cor.~5.4]{LMNPT20:Concentration-Euclidean}.

\subsubsection{Projection Formula}

The projection formula appears as~\cite[Thm.~6.2.2]{SW08:Stochastic-Integral};
the simplest forms date back to work of Cauchy and Kubota.

\begin{fact}[Projection Formula: Intrinsic Volumes] \label{fact:projection-intvol}
Consider a nonempty convex body $\set{K} \subset \R^n$.  For each subspace
dimension $m = 0, 1,2, \dots, n$ and each index $i = 0,1,2,\dots, m$,
$$
\int_{\Gr(m,n)} \intvol_i(\set{K}|\set{L}) \, {\nu}_m(\diff{\set{L}})
	= c^{m,n-i}_{n,m-i} \cdot \intvol_i(\set{K}).
$$
The Grassmannian $\Gr(m, n)$ is equipped with its invariant probability measure
$\nu_m$; see Appendix~\ref{sec:invariant-grass}.
\end{fact}

The case $m = i$ is called \term{Kubota's formula};
it shows that the $i$th intrinsic volume of $\set{K}$
is proportional to the average $i$-dimensional
volume of projections of $\set{K}$ onto $i$-dimensional subspaces.
The case $i = n - 1$ is called \term{Cauchy's formula};
it expresses the surface area of $\set{K}$ in terms of
the average $(n-1)$-dimensional volume of the projections
of $\set{K}$ onto hyperplanes through the origin.

\subsubsection{Slicing Formula}

The slicing formula is due to Crofton,
and it appears as~\cite[Thm.~5.1.1]{SW08:Stochastic-Integral}.

\begin{fact}[Slicing Formula: Intrinsic Volumes] \label{fact:slicing-intvol}
Consider a nonempty convex body $\set{K} \subset \R^n$.
For each affine space dimension $m = 0, 1, 2, \dots, n$ and each index $i = 0,1,2,\dots, m$,
$$
\int_{\Af(m,n)} \intvol_i(\set{K} \cap \set{E}) \, \mu_m(\diff{\set{E}})
	= c^{m,n-m+i}_{i,n} \cdot \intvol_{n-m+i}(\set{K}).
$$
The affine Grassmannian $\Af(m,n)$ is equipped with its
invariant measure $\mu_m$; see Appendix~\ref{sec:invariant-Af}.
\end{fact}

Of particular interest is the case $i = 0$.  Indeed, the Euler characteristic
$\intvol_0(\set{K} \cap \set{E})$ registers whether the affine space $\set{E}$
intersects the set $\set{K}$.  Thus, the slicing formula shows that the
intrinsic volume $\intvol_{n-m}(\set{K})$ is proportional to the
measure of the set of $m$-dimensional affine spaces that hit $\set{K}$.

\subsubsection{Rotation Mean Formula}

The rotation mean formula is attributed to Hadwiger;
see~\cite[Thm.~6.1.1]{SW08:Stochastic-Integral}.

\begin{fact}[Rotation Mean Value Formula: Intrinsic Volumes] \label{fact:rotmean-intvol}
Consider two nonempty convex bodies $\set{K}, \set{M} \subset \R^n$.
Then
$$
\int_{\SO(n)} \intvol_i(\set{K} + \mtx{O} \set{M}) \, \nu(\diff{\mtx{O}})
	= \sum_{j=0}^i c_{n-i,n}^{n+j-i,n-j} \cdot \intvol_j(\set{K}) \, \intvol_{i-j}(\set{M}).
$$
The special orthogonal group $\SO(n)$ is equipped with its invariant probability
measure $\nu$; see Appendix~\ref{sec:invariant-SO}.
\end{fact}

The special case $\set{M} = \lambda \ball{n}$ is equivalent to the
Steiner formula~\eqref{eqn:steiner-intro}.

\subsubsection{Kinematic Formula}

The kinematic formula goes back to work of Blaschke, Santal{\'o},
Chern and others.  We have drawn the statement from~\cite[Thm.~5.1.3]{SW08:Stochastic-Integral}.

\begin{fact}[Kinematic Formula] \label{fact:kinematic-intvol}
Consider two nonempty bodies $\set{K}, \set{M} \subset \R^n$.
For each $i = 0,1,2,\dots,n$,
$$
\int_{\RM(n)} \intvol_i(\set{K} \cap g \set{M}) \, \mu(\diff{g})
	= \sum_{j=i}^n c_{i,n}^{j,n+i-j} \cdot \intvol_j(\set{K}) \, \intvol_{n+i-j}(\set{M}).
$$
The group $\RM(n)$ of proper rigid motions on $\R^n$ is equipped with
its invariant measure $\mu$; see Appendix~\ref{sec:invariant-RM}.
\end{fact}

The special case $i = 0$ is called the \term{principal kinematic formula},
and it describes the measure of the set of rigid motions that bring $\set{M}$
into contact with $\set{K}$.

\subsection{History: Renormalization}

The idea of renormalizing the intrinsic volumes to simplify integral geometry
formulas is due to Nijenhuis~\cite{nijenhuis1974chern};
see also~\cite[pp.~176--177]{SW08:Stochastic-Integral}.  In addition to rescaling
the intrinsic volumes, Nijenhuis also rescales the invariant measures.
In contrast, we work with the standard scaling of the invariant measures
and only modify the normalization of the intrinsic volumes.  Although
our reformulations are quite trivial, we have not been able to locate
a reference for this approach in the literature.  Therefore, we have given
an appropriate amount of detail.

\section{Intrinsic Volumes: Refined Concentration}
\label{app:intvol}

Last, we state and prove a refined concentration inequality for
the intrinsic volume random variable of a convex body,
introduced in Section~\ref{sec:intvol-def}.

\subsection{Variance and Cgf}

Our main result for intrinsic volumes gives detailed bounds for the
variance of the intrinsic volume random variable and its cgf.
By standard manipulations, these statements lead to concentration
inequalities.

\begin{theorem}[Intrinsic Volumes: Variance and Cgf] \label{thm:intvol-conc}
Let $\set{K} \subset \R^n$ be a nonempty convex body with
intrinsic volume random variable $I_{\set{K}}$, as in
Definition~\ref{def:wintvol-rv}.  Define the central
intrinsic volume and its complement:
$$
\delta(\set{K}) := \Expect[ I_{\set{K}} ]
\quad\text{and}\quad
\delta_{\circ}(\set{K}) := n - \delta(\set{K}).
$$
The variance of $I_{\set{K}}$ satisfies
$$
\Var[ I_{\set{K}} ] \leq \frac{2  \delta(\set{K}) \delta_{\circ}(\set{K})}{n + \delta(\set{K})}
	\leq \frac{2 \delta(\set{K}) \delta_{\circ}(\set{K})}{n} =: \sigma^2(\set{K}).
$$
The cgf  of $I_{\set{K}}$ satisfies
\begin{align*}
\xi_{I_{\set{K}}}(\theta) &\leq \sigma^2(\set{K}) \cdot \frac{\econst^{\beta_{\circ} \theta} - \beta_{\circ} \theta - 1}{\beta_{\circ}^2}
	\quad\text{for $\theta \geq 0$}
	&\text{where}&&
	\beta_{\circ} &:= \frac{2\delta_{\circ}(\set{K})}{n} \leq 2; \\
\xi_{I_{\set{K}}}(\theta) &\leq \sigma^2(\set{K}) \cdot \frac{\econst^{\beta \theta} - \beta \theta - 1}{\beta^2}
	\quad\text{for $\theta \leq 0$}
	&\text{where}&&
	\beta &:= \frac{2\delta(\set{K})}{n} \leq 2. 
\end{align*}\end{theorem}

\noindent
The proof of Theorem~\ref{thm:intvol-conc} occupies the rest of this Appendix.

In combination with Fact~\ref{fact:poisson}, this result yields Poisson-type tail
bounds for the intrinsic volume random variable.
We can also derive weaker Bernstein inequalities, as in~\eqref{eqn:bernstein}.
Using the cgf bounds and~\eqref{eqn:cgf-indep}, we can develop probability
inequalities for sums of independent intrinsic volume random variables;
these results provide concentration for the intrinsic volumes of a Cartesian
product of convex bodies via Fact~\ref{fact:product-intvol}.  We omit detailed statements.

\subsection{Comparison with Prior Work}
\label{sec:intvol-ulc}

Theorem~\ref{thm:intvol-conc} improves substantially
over the the concentration of intrinsic volumes reported in
our previous work~\cite{LMNPT20:Concentration-Euclidean}.
For example, the earlier paper only achieves the variance
bound $\Var[ I_{\set{K}} ] \leq 4n$, while our new approach
comprehends that the variance does not exceed $2n$,
and may be far smaller.

The Alexandrov--Fenchel inequalities~\cite[Sec.~7.3]{Sch14:Convex-Bodies}
imply that the intrinsic volumes of a convex body form an
ultra-log-concave (ULC) sequence~\cite{Che76:Processus-Gaussiens,McM91:Inequalities-Intrinsic}.
Owing to~\cite[Thm.~1.3 and Lem.~5.1]{Joh17:Discrete-Log-Sobolev},
this fact ensures that the intrinsic volume sequence of $\set{K}$
has Poisson-type concentration with variance proportional to $\intvol_1(\set{K})$.
In contrast, Theorem~\ref{thm:intvol-conc} shows that
the intrinsic volumes exhibit Poisson concentration with variance
$\sigma^2(\set{K}) \leq 2 \delta_{\circ}(\set{K})$. Up to the factor $2$, our new result is always sharper than
the bound that follows from the ULC property because
$\intvol_1(\set{K}) \geq \delta_{\circ}(\set{K})$.
This surprising observation can be extracted
from~\cite[Lem.~5.1 and Lem.~5.3]{Joh17:Discrete-Log-Sobolev}.

After this paper was written, Aravinda, Marsiglietti, and Melbourne~\cite[Cor.~1.1]{AMM21:Concentration-Inequalities}
proved that all ULC sequences exhibit Poisson-type concentration.  Their argument follows by a direct comparison between
the ULC sequence and the Poisson distribution with the same mean.  The concentration inequality they obtain,
however, remains slightly weaker than the one that follows from Theorem~\ref{thm:intvol-conc}.

\subsection{Reduction to Thermal Variance Bound}

Theorem~\ref{thm:intvol-conc} follows from a bound on the thermal variance
of the intrinsic volume random variable, which we establish in the remaining
part of this Appendix.

\begin{proposition}[Intrinsic Volumes: Thermal Variance] \label{prop:intvol-tvar}
Let $\set{K} \subset \R^n$ be a nonempty convex body.  The thermal variance
of the associated intrinsic volume random variable $I_{\set{K}}$ satisfies
$$
\xi_{I_{\set{K}}}''(\theta) \leq \frac{2 \xi_{I_{\set{K}}}'(\theta) \cdot \big[ n - \xi_{I_{\set{K}}}'(\theta) \big]}{n + \xi_{I_{\set{K}}}'(\theta)}
	\leq \frac{2}{n} \cdot \xi_{I_{\set{K}}}'(\theta) \cdot \big[n - \xi_{I_{\set{K}}}'(\theta)\big]
\quad\text{for all $\theta \in \R$.}
$$
\end{proposition}

\begin{proof}[Proof of Theorem~\ref{thm:intvol-conc} from Proposition~\ref{prop:intvol-tvar}]
The argument is essentially the same as the proof of Theorem~\ref{thm:rmvol-conc}
from Proposition~\ref{prop:rmvol-tvar}.  We use the stronger inequality from
Proposition~\ref{prop:intvol-tvar} to prove the variance bound, and we back off to the
weaker inequality to establish cgf bounds.
\end{proof}

\subsection{Setup}

Let us commence with the proof of Proposition~\ref{prop:intvol-tvar},
the thermal variance bound for the intrinsic volumes.
This argument follows the same pattern as the others.

Fix a nonempty convex body $\set{K} \subset \R^n$,
which we suppress.
For $i = 0, 1, 2, \dots, n$,
let $\intvol_i$ be the $i$th intrinsic volume of $\set{K}$,
and let $\wills$ be the total intrinsic volume.
The intrinsic volume random variable $I$ follows the distribution
\begin{equation} \label{eqn:intvol-rv-pf}
\Prob{ I = n - i } = I_i / \wills
\quad\text{for $i = 0, 1, 2, \dots, n$.}
\end{equation}
Write $\delta = \Expect I $ for the expectation of the intrinsic volume random variable.

\subsection{The Distance Integral}

The first step is to express the exponential moments of the sequence of intrinsic
volumes as a distance integral.

\begin{proposition}[Intrinsic Volumes: Distance Integral] \label{prop:intvol-distint}
For each $\theta \in \R$, define the convex potential
\begin{equation} \label{eqn:intvol-potent}
\intvolJ(\vct{x}) := \pi \econst^{-2\theta} \dist^2(\vct{x})
\quad\text{for $\vct{x} \in \R^n$.}
\end{equation}
For any function $h : \R_+ \to \R$ where the expectations on the right-hand side are finite,
\begin{equation} \label{eqn:intvol-distint}
\int_{\R^n} h(2 \intvolJ(\vct{x})) \, \econst^{-\intvolJ(\vct{x})} \idiff{\vct{x}}
	=\sum_{i=0}^n \Expect[ h(X_{i})] \, \econst^{(n-i)\theta} \cdot \intvol_i. 
\end{equation}
The random variable $X_i \sim \textsc{chisq}(i)$ follows the chi-squared distribution
with $i$ degrees of freedom.  By convention, $X_0 = 0$.
\end{proposition}

\begin{proof}
The potential $\intvolJ$ defined in~\eqref{eqn:intvol-potent} is convex
because of Fact~\ref{fact:dist}\eqref{dist:1}, which states that the squared
distance is convex.
Next, we instantiate Fact~\ref{fact:distance-integral} with the function
$f(r) = h(2\pi \econst^{-2\theta} r^2) \, \econst^{-\pi \econst^{-2\theta} r^2}$
to arrive at
$$
\int_{\R^n} h(2 \intvolJ(\vct{x})) \, \econst^{-\intvolJ(\vct{x})} \idiff{\vct{x}}
	= h(0) \cdot \intvol_n + \sum_{i=0}^{n-1} \left( \int_0^\infty h(2\pi \econst^{-2\theta} r^2) \cdot \econst^{-\pi \econst^{-2\theta} r^2} r^{n-i-1} \idiff{r} \right)
	\omega_{n-i} \cdot \intvol_i.
$$
In the integral, making the change of variables $2\pi\econst^{-2\theta} r^2 \mapsto s$,
we recognize an expectation with respect to the chi-squared distribution:
\begin{multline*}
\omega_{n-i} \left( \int_0^\infty h(2\pi \econst^{-2\theta} r^2) \cdot \econst^{-\pi \econst^{-2\theta} r^2} r^{n-i-1} \idiff{r} \right) \\
	= \frac{2 \pi^{(n-i)/2}}{\Gamma((n-i)/2)} \cdot \frac{1}{2 (2\pi \econst^{-2\theta})^{(n-i)/2}} \int_0^\infty h(s) \cdot \econst^{-s/2} s^{n-i-1} \idiff{s} 
	= \Expect[ h(X_{n-i}) ] \, \econst^{(n-i) \theta}.
\end{multline*}
We have also used the formula~\eqref{eqn:ball-sphere} for the surface area $\omega_{n-i}$ of the sphere.
Combine the displays to complete the proof.
\end{proof}

This result yields alternative formulations for the total intrinsic volume
and the central intrinsic volume.  The result for $\wills$ is due
to Hadwiger~\cite{Had75:Willssche}.

\begin{corollary}[Intrinsic Volumes: Metric Formulations] \label{cor:intvol-metric}
The total intrinsic volume $\wills$ and the central intrinsic volume $\delta$
admit the expressions
\begin{align*}
\wills &= \int_{\R^n} \econst^{-\pi \dist^2(\vct{x})} \idiff{\vct{x}} &\text{and}&&
n - \delta &= \frac{1}{\wills} \int_{\R^n} 2\pi \dist^2(\vct{x}) \, \econst^{-\pi \dist^2(\vct{x})} \idiff{\vct{x}}.
\end{align*}
\end{corollary}

\begin{proof}
Apply Proposition~\ref{prop:intvol-distint} with $h(s) = 1$ and then with $h(s) = n - s$.
\end{proof}

\subsection{A Family of Log-Concave Measures}

Proposition~\ref{prop:intvol-distint} leads to an expression for the mgf $m_{I}$
of the intrinsic volume random variable.  Choose $h(s) = 1$ to find that
\begin{equation} \label{eqn:intvol-mgf}
m_I(\theta)
	\stackrel{\eqref{eqn:mgf}}{=} \Expect \econst^{\theta I}
	\stackrel{\eqref{eqn:intvol-rv-pf}}{=} \frac{1}{\wills} \sum_{i=0}^n \econst^{(n-i)\theta} \cdot \intvol_i
	= \int_{\R^n} \econst^{-\intvolJ(\vct{x})} \idiff{\vct{x}}.
\end{equation}
For each $\theta \in \R$, we introduce a log-concave probability measure $\mu_{\theta}$
with density
$$
\frac{\diff \mu_{\theta}(\vct{x})}{\diff{\vct{x}}} = \frac{1}{\wills} \cdot \frac{1}{m_{I}(\theta)} \cdot \econst^{-\intvolJ(\vct{x})} \quad\text{for $\vct{x} \in \R^n$.}
$$
Indeed, the measure $\mu_{\theta}$ has total mass one because of~\eqref{eqn:intvol-mgf}.
With this notation, we can write Proposition~\ref{prop:intvol-distint}
in a more probabilistic fashion.

\begin{corollary}[Intrinsic Volumes: Probabilistic Formulation] \label{cor:intvol-lc}
Instate the notation from Proposition~\ref{prop:intvol-distint}.
Draw a random vector $\vct{z} \sim \mu_{\theta}$.  Then
$$
\Expect[ h(2\intvolJ(\vct{x})) ]
	= \frac{1}{m_I(\theta)} \sum_{i=0}^n \Expect[ h(X_{n-i}) ] \, \econst^{(n-i)\theta} \cdot \Prob{ I = n - i }.
$$
\end{corollary}

\begin{proof}
This formulation is just a reinterpretation of Proposition~\ref{prop:intvol-distint}.
\end{proof}

\subsection{A Thermal Variance Identity}

With Corollary~\ref{cor:intvol-lc} at hand, we quickly obtain
identities for the thermal mean $\xi_I'$ and the thermal variance $\xi_I''$
of the intrinsic volume random variable $I$.  These results
are framed in terms of the mean and variance of the log-concave
measure $\mu_{\theta}$.

\begin{lemma}[Intrinsic Volumes: Thermal Variance Identity] \label{lem:intvol-tmv}
Draw a random vector $\vct{z} \sim \mu_{\theta}$.  Then
\begin{equation} \label{eqn:intvol-tmv}
\xi_I'(\theta) = \Expect[ 2 \intvolJ(\vct{z}) ]
\quad\text{and}\quad
\xi_{I}''(\theta) = 4 \Var[ \intvolJ(\vct{z}) ] - 2 \xi_{I}'(\theta).
\end{equation}
\end{lemma}

\begin{proof}
The chi-squared random variable $X_{n-i}$ satisfies the identities
$$
\Expect[ X_{n-i} ] = n - i
\quad\text{and}\quad
\Expect[ X_{n-i}^2 - 2X_{n-i} ] = (n-i)^2.
$$
Therefore, Corollary~\ref{cor:intvol-lc} with $h(s) = s$ implies that
\begin{equation} \label{eqn:intvol-tm-pf}
\xi_I'(\theta) \overset{\eqref{eqn:thermal-mean}}{=} \frac{1}{m_I(\theta)} \sum_{i=0}^n (n-i) \, \econst^{(n-i)\theta} \cdot \Prob{ I = n - i }
	= \Expect[ 2\intvolJ(\vct{z}) ].
\end{equation}
Next, apply Corollary~\ref{cor:intvol-lc} with $h(s) = s^2 - 2s$ to obtain
\begin{align*}
\xi_{I}''(\theta) &\stackrel{\eqref{eqn:thermal-var}}{=}
	\left[ \frac{1}{m_I(\theta)} \sum_{i=0}^n (n-i)^2 \, \econst^{(n-i)\theta} \cdot \Prob{ I = n - i } \right] - (\xi_I'(\theta))^2 \\
	&\stackrel{\eqref{eqn:intvol-tm-pf}}{=} \Expect[ (2 \intvolJ(\vct{z}))^2 - 2 (2\intvolJ(\vct{z})) ] - (\Expect[ 2 \intvolJ(\vct{z})])^2
	\stackrel{\eqref{eqn:intvol-tm-pf}}{=} \Var[ 2 \intvolJ(\vct{z}) ] - 2 \xi_I'(\theta).
\end{align*}Since the variance is $2$-homogeneous, this result is equivalent to the statement.
\end{proof}

\subsection{A Variance Bound}

We have now converted the problem of bounding the thermal variance
of the intrinsic volume random variable into a problem about
the variance of a function with respect to a log-concave measure.
Here is the required estimate.

\begin{lemma}[Intrinsic Volumes: Variance Bound] \label{lem:intvol-varbd}
Instate the notation of Lemma~\ref{lem:intvol-tmv}.  For a random variable $\vct{z} \sim \mu_{\theta}$,
\begin{equation} \label{eqn:intvol-varbd}
\Var[ \intvolJ(\vct{z}) ] \leq \frac{n \Expect[ 2\intvolJ(\vct{z}) ]}{n + \Expect[ 2 \intvolJ(\vct{z}) ]}.
\end{equation}
\end{lemma}

To establish Lemma~\ref{lem:intvol-varbd}, we must invoke a dimensional
variance inequality for log-concave measures.

\begin{fact}[Dimensional Brascamp--Lieb] \label{fact:dim-brascamp-lieb}
Let $J : \R^n \to \R_{++}$ be a twice differentiable and strongly convex potential.
Consider the log-concave probability measure $\mu$ on $\R^n$ whose density is proportional
to $\econst^{-J}$.  Draw a random variable $\vct{z} \sim \mu$.
For any differentiable function $f : \R^n \to \R$, $$
\Var[ f(\vct{z}) ] \leq \int \ip{(\Hess J)^{-1} \grad f}{\grad f} \diff{\mu}
	- \frac{\left[ \int Jf \idiff{\mu} - \int J \idiff{\mu} \int f \idiff{\mu} \right]^2}{n - \Var[ J(\vct{z}) ] }.
$$
We have suppressed the integration variable for legibility.
\end{fact}

Fact~\ref{fact:dim-brascamp-lieb} was established in~\cite[Prop.~4.1]{BGG18:Dimensional-Improvements}
by means of an optimal transport argument.  The second term sharpens the classic
Brascamp--Lieb variance inequality~\cite[Thm.~4.1]{BL76:Extensions-Brunn-Minkowski}
for a log-concave measure by taking into account the ambient dimension.
We can obtain a weaker version of Lemma~\ref{lem:intvol-varbd} by a direct application of Brascamp--Lieb.

\begin{proof}[Proof of Lemma~\ref{lem:intvol-varbd}]
Fix the parameter $\theta \in \R$.
We cannot apply the dimensional Brascamp--Lieb inequality directly because the potential $\intvolJ$ is not strongly convex.  Instead, for each $\eps > 0$, we construct
the strongly convex potential
$$
J^{\eps}(\vct{x}) := \intvolJ(\vct{x}) + \half\eps \norm{\vct{x}}^2,
\quad\text{recalling that}\quad
\intvolJ(\vct{x}) = \pi \econst^{-2\theta} \dist^2(\vct{x}).
$$
We will consider the function $f = \intvolJ$, without any perturbation.

Our goal is to evaluate the quadratic form induced by the inverse Hessian of $J^\eps$
at the vector $\grad f$.  To that end, note that the original potential $\intvolJ$ satisfies
\begin{equation} \label{eqn:intvolJ-grad}
\grad \intvolJ( \vct{x} ) = \pi \econst^{-2 \theta} \grad \dist^2(\vct{x}) = 2\pi \econst^{-2\theta} \dist(\vct{x}) \, \vct{n}(\vct{x}).
\end{equation}
We have used Fact~\ref{fact:dist}\eqref{dist:2}.
Owing to Fact~\ref{fact:dist}\eqref{dist:3},
the Hessian of the perturbed potential $J^{\eps}$ satisfies
$$
( \Hess J^{\eps}(\vct{x}) ) \, \grad J(\vct{x})
	= (2 \pi \econst^{-2\theta} + \eps) \grad \intvolJ(\vct{x}).
$$
Multiply both sides by the inverse Hessian, and take the inner product with $\grad \intvolJ$ to arrive at
$$
\ip{( \Hess J^{\eps}(\vct{x}) )^{-1} \, \grad \intvolJ(\vct{x})}{ \grad \intvolJ(\vct{x}) }
	= \frac{(2\pi \econst^{-2\theta})^2}{2\pi\econst^{-2\theta} + \eps} \dist^2(\vct{x})
	= \frac{2\pi \econst^{-2\theta}}{2\pi\econst^{-2\theta} + \eps} \cdot 2 \intvolJ(\vct{x}).
$$
We have used~\eqref{eqn:intvolJ-grad} and the property that the normal vector $\vct{n}(\vct{x})$ 
has unit length when $\dist(\vct{x}) > 0$.  Taking the limit as $\eps \downarrow 0$,
$$
\ip{( \Hess \intvolJ(\vct{x}) )^{-1}  \, \grad \intvolJ(\vct{x})}{ \grad \intvolJ(\vct{x}) }
	= 2\intvolJ(\vct{x}).
$$
This computation delivers the integrand in the dimensional Brascamp--Lieb inequality.

Fact~\ref{fact:dim-brascamp-lieb} applies to the strongly convex potential $J^{\eps}$
and the function $f = \intvolJ$.  By dominated convergence, the inequality remains
valid with the limiting potential $\intvolJ$.  Thus, for $\vct{z} \sim \mu_{\theta}$,
$$
\Var[ \intvolJ(\vct{z}) ]
	\leq \int 2\intvolJ \idiff \mu - \frac{\Var[ \intvolJ(\vct{z}) ]^2}{n - \Var[ \intvolJ(\vct{z}) ]}.
$$
The remaining integral coincides with $\Expect[ 2\intvolJ(\vct{z})]$.
Solve for the variance to complete the proof.
\end{proof}

\subsection{Proof of Proposition~\ref{prop:intvol-tvar}}

We are now prepared to bound the thermal variance $\xi_{I}''(\theta)$
of the intrinsic volume random variable.  Together, Lemmas~\ref{lem:intvol-tmv}
and~\ref{lem:intvol-varbd} imply that
$$
\xi_{I}''(\theta)
	\stackrel{\eqref{eqn:intvol-tmv}}{=} 4 \Var[ \intvolJ(\vct{z}) ] - 2 \xi_I'(\theta)
	\stackrel{\eqref{eqn:intvol-varbd}}{\leq} \frac{4n \Expect[2\intvolJ]}{n + \Expect[2\intvolJ]} - 2 \xi_I'(\theta)
	\stackrel{\eqref{eqn:intvol-tmv}}{=} \frac{2 \xi_I'(\theta) (n - \xi_I'(\theta))}{n + \xi_I'(\theta)}.
$$
This completes the proof of Proposition~\ref{prop:intvol-tvar}
and, with it, the proof of Theorem~\ref{thm:intvol-conc}.

\section*{Acknowledgments}

This paper realizes a vision that was put forth by our colleague Michael McCoy in 2013.
The project turned out to be more challenging than anticipated.

The authors would like to thank Jiajie Chen and De Huang
for their insights on the concentration of information inequality.
Franck Barthe, Arnaud Marsiglietti, Michael McCoy, James Melbourne,
Ivan Nourdin, Giovanni Peccati, Rolf Schneider,
and Ramon Van Handel provided valuable feedback on the first draft
of this paper.

MAL would like to thank the Isaac Newton Institute for Mathematical Sciences for support and hospitality during the programme ``Approximation, Sampling and Compression in Data Science'', when work on this paper was undertaken. This work was supported by EPSRC grant number EP/R014604/1.

JAT gratefully acknowledges funding from ONR awards N00014-11-1002, N00014-17-12146, and N00014-18-12363.
He would like to thank his family for support during these trying times.

\bibliographystyle{myalpha}

\begin{thebibliography}{AAGM15}

\bibitem[AAGM15]{artstein2015asymptotic}
S.~Artstein-Avidan, A.~Giannopoulos, and V.~D. Milman.
\newblock Asymptotic geometric analysis, part {I}.
\newblock {\em Mathematical Surveys and Monographs}, 202, 2015.

\bibitem[ALMT14]{ALMT14:Living-Edge}
D.~Amelunxen, M.~Lotz, M.~B. McCoy, and J.~A. Tropp.
\newblock Living on the edge: {P}hase transitions in convex programs with
  random data.
\newblock {\em Inf. Inference}, 3(3):224--294, 2014.

\bibitem[AMM21]{AMM21:Concentration-Inequalities}
H.~Aravinda, A.~Marsiglietti, and J.~Melbourne.
\newblock Concentration inequalities for ultra-log-concave distributions.
\newblock Available at \url{https://arXiv.org/abs/2105.05054}, Apr. 2021.

\bibitem[Bar60]{barbier1860note}
E.~Barbier.
\newblock Note sur le probl{\`e}me de l'aiguille et le jeu du joint couvert.
\newblock {\em Liouville J.}, 5, 1860.

\bibitem[Bar02]{Bar02:Course-Convexity}
A.~Barvinok.
\newblock {\em A course in convexity}, volume~54 of {\em Graduate Studies in
  Mathematics}.
\newblock American Mathematical Society, Providence, RI, 2002.

\bibitem[Ber72]{Ber89:Calcul-Probabilites}
J.~Bertrand.
\newblock {\em Calcul des probabilit\'{e}s}.
\newblock Chelsea Publishing Co., Bronx, N.Y., 1972.
\newblock R\'{e}impression de la deuxi\`eme \'{e}dition de 1907.

\bibitem[BGG18]{BGG18:Dimensional-Improvements}
F.~Bolley, I.~Gentil, and A.~Guillin.
\newblock Dimensional improvements of the logarithmic {S}obolev, {T}alagrand
  and {B}rascamp-{L}ieb inequalities.
\newblock {\em Ann. Probab.}, 46(1):261--301, 2018.

\bibitem[BL76]{BL76:Extensions-Brunn-Minkowski}
H.~J. Brascamp and E.~H. Lieb.
\newblock On extensions of the {B}runn--{M}inkowski and
  {P}r\'{e}kopa--{L}eindler theorems, including inequalities for log concave
  functions, and with an application to the diffusion equation.
\newblock {\em J. Functional Analysis}, 22(4):366--389, 1976.

\bibitem[BL09]{BL09:Weighted-Poincare-Type}
S.~G. Bobkov and M.~Ledoux.
\newblock Weighted {P}oincar{\'e}-type inequalities for {C}auchy and other
  convex measures.
\newblock {\em The Annals of Probability}, 37(2):403--427, 2009.

\bibitem[BLM88]{BLM88:Minkowski-Sums}
J.~Bourgain, J.~Lindenstrauss, and V.~D. Milman.
\newblock Minkowski sums and symmetrizations.
\newblock In {\em Geometric aspects of functional analysis (1986/87)}, volume
  1317 of {\em Lecture Notes in Math.}, pages 44--66. Springer, Berlin, 1988.

\bibitem[BLM13]{boucheron2013concentration}
S.~Boucheron, G.~Lugosi, and P.~Massart.
\newblock {\em Concentration inequalities: A nonasymptotic theory of
  independence}.
\newblock Oxford university press, 2013.

\bibitem[Bor75]{Bor75:Convex-Set}
C.~Borell.
\newblock Convex set functions in {$d$}-space.
\newblock {\em Period. Math. Hungar.}, 6(2):111--136, 1975.

\bibitem[Che76]{Che76:Processus-Gaussiens}
S.~Chevet.
\newblock Processus {G}aussiens et volumes mixtes.
\newblock {\em Z. Wahrscheinlichkeitstheorie und Verw. Gebiete}, 36(1):47--65,
  1976.

\bibitem[Cro89]{Cro89:Probability}
M.~W. Crofton.
\newblock Probability.
\newblock In {\em Encyclopaedia Britannica}, volume~19. Charles Scribner's
  Sons, New York, 9th edition, 1889.

\bibitem[FMW16]{FMW16:Optimal-Concentration}
M.~Fradelizi, M.~Madiman, and L.~Wang.
\newblock Optimal concentration of information content for log-concave
  densities.
\newblock In {\em High dimensional probability {VII}}, volume~71 of {\em Progr.
  Probab.}, pages 45--60. Springer, [Cham], 2016.

\bibitem[GNP17]{GNP17:Gaussian-Phase}
L.~Goldstein, I.~Nourdin, and G.~Peccati.
\newblock Gaussian phase transitions and conic intrinsic volumes: {S}teining
  the {S}teiner formula.
\newblock {\em Ann. Appl. Probab.}, 27(1):1--47, 2017.

\bibitem[Had51]{Had51:Funktionalsatzes}
H.~Hadwiger.
\newblock Beweis eines {F}unktionalsatzes f\"{u}r konvexe {K}\"{o}rper.
\newblock {\em Abh. Math. Sem. Univ. Hamburg}, 17:69--76, 1951.

\bibitem[Had52]{Had52:Additive-Funktionale}
H.~Hadwiger.
\newblock Additive {F}unktionale {$k$}-dimensionaler {E}ik\"{o}rper. {I}.
\newblock {\em Arch. Math.}, 3:470--478, 1952.

\bibitem[Had57]{Had57:Vorlesungen}
H.~Hadwiger.
\newblock {\em Vorlesungen \"{u}ber {I}nhalt, {O}berfl\"{a}che und
  {I}soperimetrie}.
\newblock Springer-Verlag, Berlin-G\"{o}ttingen-Heidelberg, 1957.

\bibitem[Had75]{Had75:Willssche}
H.~Hadwiger.
\newblock Das {W}ills'sche {F}unktional.
\newblock {\em Monatsh. Math.}, 79:213--221, 1975.

\bibitem[Joh17]{Joh17:Discrete-Log-Sobolev}
O.~Johnson.
\newblock A discrete log-{S}obolev inequality under a {B}akry-\'{E}mery type
  condition.
\newblock {\em Ann. Inst. Henri Poincar\'{e} Probab. Stat.}, 53(4):1952--1970,
  2017.

\bibitem[KR97]{KR97:Introduction-Geometric}
D.~A. Klain and G.-C. Rota.
\newblock {\em Introduction to geometric probability}.
\newblock Lezioni Lincee. [Lincei Lectures]. Cambridge University Press,
  Cambridge, 1997.

\bibitem[Lec35]{leclerc1733essai}
G.-L. Leclerc.
\newblock Essai d'arithm{\'e}tique morale.
\newblock {\em Histoire de l'Acad{\'e}mie royale des sciences. Ann{\'e}e
  MDCCXXXIII}, Suppl{\'e}ment {\`a} l'Histoire Naturelle, Tome
  Quatri{\`e}me:46--148, 1735.

\bibitem[Led97]{Led96:Talagrands-Deviation}
M.~Ledoux.
\newblock On {T}alagrand's deviation inequalities for product measures.
\newblock {\em ESAIM: Probability and statistics}, 1:63--87, 1997.

\bibitem[LMN{\etalchar{+}}20]{LMNPT20:Concentration-Euclidean}
M.~Lotz, M.~B. McCoy, I.~Nourdin, G.~Peccati, and J.~A. Tropp.
\newblock Concentration of the intrinsic volumes of a convex body.
\newblock In B.~Klartag and E.~Milman, editors, {\em Geometric Aspects of
  Functional Analysis --- Israel Seminar (GAFA) 2017-2019}, number 2256 and
  2266 in Lecture Notes in Math. Springer, 2020.

\bibitem[Mau12]{Mau12:Thermodynamics-Concentration}
A.~Maurer.
\newblock Thermodynamics and concentration.
\newblock {\em Bernoulli}, 18(2):434--454, 2012.

\bibitem[McM91]{McM91:Inequalities-Intrinsic}
P.~McMullen.
\newblock Inequalities between intrinsic volumes.
\newblock {\em Monatsh. Math.}, 111(1):47--53, 1991.

\bibitem[Mil71]{Mil71:New-Proof}
V.~D. Milman.
\newblock A new proof of {A}. {D}voretzky's theorem on cross-sections of convex
  bodies.
\newblock {\em Funkcional. Anal. i Prilo\v{z}en.}, 5(4):28--37, 1971.

\bibitem[Mil90]{Mil90:Note-Low}
V.~Milman.
\newblock A note on a low {$M^*$}-estimate.
\newblock In {\em Geometry of {B}anach spaces ({S}trobl, 1989)}, volume 158 of
  {\em London Math. Soc. Lecture Note Ser.}, pages 219--229. Cambridge Univ.
  Press, Cambridge, 1990.

\bibitem[MP01]{Pet01:Sur-Maniere}
M.~Michel~Petrovitch.
\newblock Sur une mani\`ere d'\'{e}tendre le th\'{e}or\`eme de la moyenne aux
  \'{e}quations diff\'{e}rentielles du premier ordre.
\newblock {\em Math. Ann.}, 54(3):417--436, 1901.

\bibitem[MT14]{MT14:Steiner-Formulas}
M.~B. McCoy and J.~A. Tropp.
\newblock From {S}teiner formulas for cones to concentration of intrinsic
  volumes.
\newblock {\em Discrete Comput. Geom.}, 51(4):926--963, 2014.

\bibitem[Ngu13]{Ngu13:Inegalites-Fonctionelles}
V.~H. Nguyen.
\newblock {\em In{\'e}galit{\'e}s Fonctionelles et Convexit{\'e}}.
\newblock PhD thesis, Universit{\'e} Pierrre et Marie Curie (Paris VI), 2013.

\bibitem[Ngu14]{Ngu14:Dimensional-Variance}
V.~H. Nguyen.
\newblock Dimensional variance inequalities of {B}rascamp--{L}ieb type and a
  local approach to dimensional {P}r{\'e}kopa's theorem.
\newblock {\em Journal of Functional Analysis}, 266(2):931 -- 955, 2014.

\bibitem[Nij74]{nijenhuis1974chern}
A.~Nijenhuis.
\newblock On {C}hern's kinematic formula in integral geometry.
\newblock {\em J. Diff. Geom.}, 9(3):475--482, 1974.

\bibitem[Roc70]{Roc70:Convex-Analysis}
R.~T. Rockafellar.
\newblock {\em Convex analysis}.
\newblock Princeton Mathematical Series, No. 28. Princeton University Press,
  Princeton, N.J., 1970.

\bibitem[RW09]{RW98:Variational-Analysis}
R.~T. Rockafellar and R.~J.-B. Wets.
\newblock {\em Variational analysis}, volume 317.
\newblock Springer Science \& Business Media, 2009.

\bibitem[Sch14]{Sch14:Convex-Bodies}
R.~Schneider.
\newblock {\em Convex bodies: The {B}runn--{M}inkowski theory}, volume 151 of
  {\em Encyclopedia of Mathematics and its Applications}.
\newblock Cambridge University Press, Cambridge, expanded edition, 2014.

\bibitem[SW08]{SW08:Stochastic-Integral}
R.~Schneider and W.~Weil.
\newblock {\em Stochastic and integral geometry}.
\newblock Probability and its Applications (New York). Springer-Verlag, Berlin,
  2008.

\bibitem[Ver18]{Ver18:High-Dimensional-Probability}
R.~Vershynin.
\newblock {\em High-dimensional probability}, volume~47 of {\em Cambridge
  Series in Statistical and Probabilistic Mathematics}.
\newblock Cambridge University Press, Cambridge, 2018.
\newblock An introduction with applications in data science, With a foreword by
  Sara van de Geer.

\bibitem[Wan14]{wang2014heat}
L.~Wang.
\newblock {\em Heat capacity bound, energy fluctuations and convexity}.
\newblock PhD thesis, Yale University, 2014.

\bibitem[Wil73]{Wil73:Gitterpunktanzahl}
J.~M. Wills.
\newblock Zur {G}itterpunktanzahl konvexer {M}engen.
\newblock {\em Elem. Math.}, 28:57--63, 1973.

\end{thebibliography}
\newcommand{\etalchar}[1]{$^{#1}$}

\end{document}